\documentclass[a4paper,11pt,oneside]{amsbook}

\numberwithin{section}{chapter}
\numberwithin{equation}{chapter}

\def\frontmatter{\clearpage\pagenumbering{roman}}
\def\mainmatter{\clearpage\pagenumbering{arabic}}

\makeindex

\usepackage{amssymb, amsmath, amsfonts,mathrsfs,bussproofs,verbatim}
\usepackage[matrix,arrow,curve]{xy}
\CompileMatrices

\usepackage[left=4cm,right=1in,top=3cm,bottom=1in]{geometry}

\usepackage{setspace}
\DeclareMathAlphabet{\mathpzc}{OT1}{pzc}{m}{it}

\newcommand{\Hom}{\ensuremath{\mathrm{Hom}}}
\newcommand{\C}{\ensuremath{\mathscr{C}}}

\newcommand{\Cat}{\ensuremath{\mathbf{Cat}}}
\newcommand{\F}{\ensuremath{\mathcal{F}}}
\newcommand{\T}{\ensuremath{\mathcal{T}}}
\renewcommand{\L}{\ensuremath{\mathscr{L}}}

\newcommand{\tuple}{\ensuremath{ \F; \T \,|\, \E_\F\,;\,
    \E_\T}}

\newcommand{\Fr}{\ensuremath{\mathbb{F}}}

\newcommand{\N}{\ensuremath{\mathbb{N}}}

\newcommand{\lmb}{\ensuremath{\mathrm{lmb}}}

\newcommand{\of}{\ensuremath{\cdot}}
\newcommand{\M}{\ensuremath{\mathscr{M}}}
\newcommand{\D}{\ensuremath{\mathscr{D}}}
\newcommand{\Struct}{\ensuremath{\mathrm{Struct}}}

\newcommand{\n}{\ensuremath{\mathcal{N}}}

\newcommand{\tensor}{\ensuremath{\otimes}}
\renewcommand{\S}{\ensuremath{\mathbb{S}}}
\newcommand{\Rr}{\ensuremath{\mathbb{R}}}

\newcommand{\Shape}{\ensuremath{\mathrm{Shape}}}
\newcommand{\var}{\ensuremath{\mathrm{Var}}}
\newcommand{\Red}{\ensuremath{\mathrm{Red}}}

\newcommand{\R}{\ensuremath{\mathcal{R}}}
\newcommand{\Out}{\ensuremath{\mathrm{Out}}}
\newcommand{\Nat}{\ensuremath{\mathrm{Nat}}}
\newcommand{\op}{\ensuremath{\mathrm{op}}}
\newcommand{\Sing}{\ensuremath{\mathrm{Sing}}}
\newcommand{\Th}{\ensuremath{\mathrm{Th}}}
\newcommand{\diag}{\ensuremath{\bigtriangleup}}
\newcommand{\iso}{\ensuremath{\cong}}
\newcommand{\supp}{\ensuremath{\mathrm{supp}}}
\newcommand{\Cn}{\ensuremath{{\mathbb{C}_n}}}
\newcommand{\SCn}{\ensuremath{{\mathbb{SC}_n}}}

\newcommand{\Sub}{\ensuremath{\mathrm{Sub}}}
\newcommand{\V}{\ensuremath{\mathcal{V}}}
\newcommand{\dom}{\ensuremath{\mathrm{dom}}}
\newcommand{\sub}{\ensuremath{\mathrm{sub}}}
\newcommand{\ETh}{\ensuremath{\langle \F \,|\, \E_\F\rangle_\V}}

\newcommand{\ltrX}{\ensuremath{\langle \F; \T\, |\, \E_{\F}  \rangle_X}}

\newcommand{\ltr}{\ensuremath{\langle \F; \T\, |\, \E_{\F}  \rangle}}

\newcommand{\rtt}{\ensuremath{\langle\tuple\rangle}}
\newcommand{\E}{\ensuremath{\mathcal{E}}}
\newcommand{\cat}{\ensuremath{\langle \F; I(\E_\F) \,|\,
    \varnothing; \E_{I(\E_\F)}\rangle}}

\newcommand{\Span}{
  \ensuremath{
    \xymatrix@1{
      {u_1} & {s} \ar[l]_{\varphi_1}
      \ar[r]^{\varphi_2} & {u_2}}
  }
}

\newtheorem{theorem}{Theorem}[section]
\newtheorem{lemma}[theorem]{Lemma}
\newtheorem{definition}[theorem]{Definition}
\newtheorem{proposition}[theorem]{Proposition}
\newtheorem{corollary}[theorem]{Corollary}

\newtheorem{example}[theorem]{Example}

\newtheorem{q}{Question}
\newtheorem{project}[q]{Project}
\newcommand{\arrow}[2]{\ensuremath{
    \def\objectstyle{\scriptstyle}
    \def\labelstyle{\scriptstyle}
    \xymatrix@1{
      {#1} ~\to~  {#2}
    }
  }
}

\newcommand{\blanknonumber}{\newpage\thispagestyle{empty}}

\begin{document}
 
\frontmatter
\pagenumbering{roman}

%\title{Coherence for rewriting $2$-theories}
%\subtitle{General}
%\author{Jonathan A. Cohen}
%\date{\today}

%\begin{titlepage}

\begin{center}
      \hrule height 4pt
    \vspace{4mm}
    \huge{
    \textbf{Coherence for rewriting $2$-theories}\\
  }
 
\vspace{2mm}
 \noindent\large{\bf
    General theorems with applications to presentations
    of Higman-Thompson groups and iterated monoidal categories.}
\end{center}
\vspace{4mm}
    \hrule height 4pt

\vspace{1in}
\begin{center}
  \Large{\textbf{Jonathan Asher Cohen}}
\end{center}

\vspace{1in}
\begin{center}
\large{
    A thesis submitted for the degree of Doctor of Philosophy\\ of the
    Australian National University. 
}
\end{center}

\vspace{1in}
\begin{center}
\large{
    June 2008
}
\end{center}

%\end{titlepage}

\mbox{}
\vspace{2.2cm}
%\vspace*{\fill}
\begin{center}
  \textbf{
    \Large{Declaration}
  }
\end{center}
\vspace{0.8cm}
\thispagestyle{empty}
The work in this thesis is my own except where otherwise stated.

\vspace{1in}

\hfill\hfill\hfill
Jonathan Asher Cohen%
\hspace*{\fill}
\blanknonumber
  \thispagestyle{empty}
\vspace*{\fill}

\begin{center}\emph{
Dedicated to the memory of my grandmother\\
Shirley Esther Lipinski (1930--2006)\\
who always listened to my ramblings.}
\end{center}

\vfill\vfill\vfill
\blanknonumber
\mbox{}
\vspace{2.2cm}
%\vspace*{\fill}
\begin{center}
  \textbf{
    \Large{Acknowledgements}
  }
\end{center}
\vspace{0.8cm}
Foremost, I would like to thank Mike Johnson for taking me on as a
student despite his many commitments and for guiding my work over the
past two and a half years. Mike has the uncanny ability to always say
just the right thing to get me thinking along fruitful paths. 

Since $2006$, I have been based at Macquarie University in
the Department of Computing. The department and the university have
been unfailingly hospitable during my lengthy visit. The members of
the Australian Category Seminar have been particularly welcoming and
the weekly seminars have been a highlight of my candidature. Thanks in
particular to Steve Lack for comments on earlier drafts of my thesis,
to Michael Batanin for useful conversations on iterated monoidal 
categories and to all of the members of the categories group for the
social camaraderie. Thanks are also due to Robin Cockett for
introducing me to coherence and encouraging me to attend StreetFest.

Thanks to Rajeev Gor\'{e} for supervising me while at ANU and for being
supportive of my lengthy absence. Conducting a PhD from several
hundred kilometers away is only possible with the aid of dedicated
administrators. I am particularly lucky to have had Di Kossatz,
Michelle Moravec and Suzanne van Haeften providing advice and
support from a distance. 

Thanks to Greg Restall and the members of the Department of Philosophy
at Melbourne University for their hospitality. Even though proof
theoretic semantics did not end up making it into my thesis, the stay
was still influential in shaping my thinking. 

Thanks to my family for all of their enthusiastic support. My friends
in Sydney, Canberra and Perth ensured that I had ample
welcome distractions. Eve Slavich and Saritha Manickam never let me
forget about my  thesis though.

Thanks to Anna for all the little things that make life a lot happier. \blanknonumber
\mbox{}
\vspace{2.2cm}
%\vspace*{\fill}
\begin{center}
  \textbf{
    \Large{Abstract}
  }
\end{center}
\vspace{0.8cm}
\addcontentsline{toc}{chapter}{Abstract}

  The problems of the identity of proofs, equivalences of reductions
  in term rewriting systems and coherence in categories all share the
  common goal of describing the notion of equivalence generated by a
  two-dimensional congruence. This thesis provides a unifying setting
  for studying such structures, develops general tools for determining
  when a congruence identifies all reasonable parallel pairs of
  reductions and examines specific applications of these results within
  combinatorial algebra. The problems investigated fall under the
  umbrella of ``coherence'' problems, which deal with the
  commutativity of diagrams in free categorical structures ---
  essentially a two-dimensional word problem. It is categorical
  structures equipped with a congruence that collapses the free
  algebra into a preorder that are termed ``coherent''.  

  The first main result links coherence problems with
  algebraic invariants of equational theories. It is shown that a
  coherent categorification of an equational theory yields a
  presentation of the associated structure monoid. It is
  subsequently shown that the higher Thompson groups $F_{n,1}$ and
  the Higman-Thompson groups $G_{n,1}$ arise as structure groups of
  equational theories, setting up the problem of obtaining coherent
  categorifications for these theories. 

  Two general approaches to obtaining coherence theorems are
  presented. The first applies in the case where the 
  underlying rewriting system is confluent and terminating. A general
  theorem is developed, which applies to many coherence problems
  arising in the literature. As a specific application of the result,
  coherent categorifications for the theories of higher order
  associativity and of higher order associativity and commutativity
  are constructed, yielding presentations for $F_{n,1}$ and $G_{n,1}$,
  respectively. 

  The second approach does not rely on the confluence of the
  underlying rewriting system and requires only a weak form of
  termination. General results are obtained in this setting for the
  decidability of the two-dimensional word problem and for determining
  when a structure satisfying the weakened properties is coherent. A
  specific application of the general theorem is made to obtain a
  conceptually straightforward proof of the coherence theorem for
  iterated monoidal categories, which form a categorical model of
  iterated loop spaces and fail to be confluent.
\blanknonumber
\tableofcontents
\mbox{}\newpage

\mainmatter

\chapter{Introduction}\label{ch:introduction}

Coherence problems arise in category theory when one wishes to
describe the free algebra generated by a particular
structure. Typically, this problem boils down to solving a sequence of
word problems: Which functors are equal? Which natural transformations
are equal? Which modifications are equal? And so on up the
dimensions. Our main interest here is in two-dimensional categorical
structures. Within this context, 
coherence problems are related to several other problems: When are two
proofs of the same theorem equivalent? When do two interpretations of
the same sentence assign the same meaning? When do two programs
implement the same algorithm? In order to gain some insight into the
importance and  meaning of coherence problems, we explore the analogy
with natural and artificial languages slightly deeper.  

A written language may be thought of as a collection of symbols together
with rules for manipulating and combining them. A sequence of such
symbols is called a sentence. A sentence is grammatical if it
can be constructed via the rules of the language.

Attempting to ascribe meaning to sentences of a language is
potentially fraught with difficulty. For a simple mathematical
language, such as arithmetic, the meaning of a sentence is abundantly
clear --- it is the natural number obtained by carrying out the
described calculation. For more complicated constructions, such as
natural language, the problem can be significantly more difficult. 

One typically wishes to assign a meaning to every possible grammatical
sentence of a language. If one considers sentences to be completely
independent of each other, then, for any reasonably complex language,
one would need to decide on the meaning of infinitely many
sentences. Such a task is unreasonable in practice. One way
in which to resolve this situation is to suppose that the language is
compositional.  That is, that the meaning of a sentence is composed
from the meaning of its 
subparts. It is important to note that two related claims are being
made here. First, there is a collection of basic syntactic structures,
which carry meaning. These can be words, such as ``dog'', ``cat'',
``table'', ``chair'' etc., or they may be more complicated phrases or
sentences. The second claim is that the meaning of a sentence built
from these basic pieces is a composition of the meanings of the
pieces. This compositionality principle is appealing on a number of
levels, not least of all because it provides a reasonable explanation
for a person's ability to comprehend sentences that they hear for the
first time. A more technical reason is that one can show that any
recursively enumerable language can be captured by some compositional
grammar \cite{janssen:compositionality}. 

Within richly expressive languages, there is the potential for
\emph{structural ambiguity}. That is, a given sentence may have two
distinct meanings even though the meanings of the individual words
remains constant. For example, the sentence ``The shooting of the
hunters was terrible'' may mean that the hunters had terrible aim, or
that it was a shame that the hunters were shot. Within the framework
of compositionality, the two meanings could only have arisen from
composing the words in a different manner.

When designing a computer programming language, one typically wishes
to avoid the presence of any structural ambiguity. More technically,
any two proofs (``compositions'') of the same typing judgement
(``sentence'') must carry the same meaning
\cite{Curien:coherence,Reynolds:coherence}. 

A language that contains no structural ambiguity whatsoever is termed
``coherent''. The name stems from Mac Lane's construction of a
coherent language for a monoidal structure on a category
\cite{MacLane_natural}, which is the real starting point for this
thesis. 

Mac Lane \cite{MacLane:TopLogic} attributes the motivation for his
development of the theory of monoidal categories to a question of
Norman Steenrod: When is there a canonical map between two specified
\emph{formal} combinations of modules? Steenrod was considering the
category of all modules over a commutative ring and the combinations
of such modules by applying the functors $\tensor$ and
$\Hom$. Monoidal categories abstract the structure of the tensor
product of modules to create a bifunctor $\tensor$ on an arbitrary
category. The main result of \cite{MacLane_natural} says, essentially,
that any two $n$-fold products that contain the same objects in the
same order are naturally isomorphic via a unique canonical natural
isomorphism. Interpreting the $n$-fold products as parsings of
sentences and natural isomorphisms as weak equivalences between
parsings, this result is akin to saying that monoidal categories do
not contain any structural ambiguity.   

The investigation of coherence is certainly not limited to monoidal
categories. Indeed, one may hope for a version of Mac Lane's theorem
for many different types of covariant structures. A covariant
structure on a category $\C$ consists of: 
\begin{itemize}
  \item A collection of basic functors of the form $\C^n \to \C$.
  \item A collection of equations between certain pairs of formally
    different terms built from the basic functors.
  \item A collection of natural transformations between certain terms
    formed from the basic functors.
  \item A collection of equations between pairs of formally different
    natural transformations constructed via a 
    sequence of compositions and substitutions of the basic natural
    transformations. These are typically called \emph{coherence axioms}.
\end{itemize}

One of the most basic covariant structures is that of a coherently
associative bifunctor $\tensor$. This structure consists of a natural
isomorphism $\alpha: a\tensor (b \tensor c) \to (a \tensor b)\tensor
c$ together with a coherence axiom stipulating that the following
diagram commutes:

\[
\vcenter{
  \begin{xy}
    (0,0)*+{a\tensor (b\tensor(c\tensor d)))}="1";
    (-25,-18)*+{(a\tensor b)\tensor (c \tensor d)}="2";
    (25,-18)*+{a\tensor((b\tensor c) \tensor d)}="3";
    (-25,-36)*+{((a\tensor b) \tensor c)\tensor d}="4";
    (25,-36)*+{(a\tensor (b\tensor c))\tensor d}="5";
    {\ar@{->}_>>>>>>>{\alpha}@/_/ "1"; "2"}
    {\ar@{->}^>>>>>>>{1\tensor\alpha}@/^/ "1"; "3"}
    {\ar@{->}_{\alpha} "2"; "4"}
    {\ar@{->}^{\alpha} (20,-21); (20,-32)}
    {\ar@{->}^{\alpha\tensor 1} (5,-36); (-12,-36)}
  \end{xy}
}
\]
A special case of Mac Lane's coherence theorem for monoidal categories
states that any other diagram constructed from $\alpha,\tensor$ and
the identity natural isomorphisms commutes by virtue of the
commutativity of the above diagram. 

In endeavouring to construct an analogous coherence theorem for an
arbitrary covariant structure carried by a category, one may ask two
related questions: 
\begin{enumerate}
  \item Is a given covariant structure coherent?
  \item What coherence axioms are required in order to make a given
    covariant structure coherent?
\end{enumerate} 

The main goal of this thesis is to tackle the above questions in the
greatest possible generality as well as to develop applications of the
resulting coherence theorems. In the following section, we give a more
complete outline.

\section*{Outline}

\noindent \textbf{Chapter 2:} The chapter starts by developing a
definition of rewriting $2$-theories. These form the main framework
for our investigations and the chapter  describes the free algebra
generated by a rewriting $2$-theory before showing that a rewriting
$2$-theory defines a Lawvere $2$-theory and, hence, a covariant
structure. After briefly 
discussing relations to other existing systems, the coherence problem
is rigorously defined within the context of rewriting
$2$-theories. Categorifications are introduced as a method for
weakening an equational variety into a categorical structure and it is
shown that a coherent 
categorification of an equational variety defines an equivalent
categorical structure to the variety. Finally, some useful general
tools for working with rewriting $2$-theories are introduced.\\

\noindent \textbf{Chapter 3:} Dehornoy \cite{Dehornoy:varieties}
introduced \emph{structure monoids} as algebraic invariants of
equational varieties. The main result of the chapter shows how to
construct a presentation of the structure monoid of an equational
variety $\E$ from a coherent categorification of $\E$. In certain
situations, the structure monoid forms a group in a natural way and
the result is extended to this setting.\\

\noindent \textbf{Chapter 4:} The main direction of this chapter is to
generalise Mac Lane's proof of coherence for monoidal categories to
rewriting $2$-theories that are confluent and
terminating. ``Terminating'' means that there are no infinite chains
of non-identity morphisms, while ``confluence'' is the property that 
every span may be completed into a square, as in the following diagram:
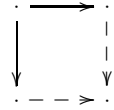
\[
\def\objectstyle{\scriptstyle}
\def\labelstyle{\scriptstyle}
\vcenter{
  \xymatrix{
    {\cdot} \ar[r] \ar[d] &  {\cdot} \ar@{-->}[d] \\
    {\cdot} \ar@{-->}[r] & {\cdot}
  }
}
\]
Subsequently, a general coherence theorem is developed for rewriting
$2$-theories describing invertible covariant structures, which
directly generalises the situation of monoidal categories.\\

\noindent \textbf{Chapter 5:} This chapter develops a surprising
application of the results of Chapter \ref{ch:unf}. Dehornoy
\cite{Dehornoy:thompson} has previously shown that Thompson's group
$F$ is the structure group of the variety of semigroups and that
Thompson's group $V$ is the structure group of the variety of
commutative semigroups. Dehornoy also constructed presentations of
these groups using Mac Lane's coherence axioms for the associated
categorifications. In light of the results of Chapter
\ref{ch:structure}, these presentations are not too
surprising. Indeed, the work in Chapter \ref{ch:structure} was
directly motivated by these results. Chapter \ref{ch:catalan} begins
by constructing varieties for higher-order associativity and
higher-order associativity and commutativity. It is shown that the
structure groups for these are the higher Thompson groups $F_{n,1}$
and the Higman-Thompson groups $G_{n,1}$, respectively. The chapter
goes on to construct categorifications of these varieties and thereby
to obtain new presentations of $F_{n,1}$ and $G_{n,1}$. The coherence
axioms for the categorifications directly generalise Mac Lane's axioms
for the binary case, although a new class of coherence axioms is
required in the higher-order case that are not present in the binary
situation. \\

\noindent \textbf{Chapter 6:} It is not the case that every coherent
rewriting $2$-theory is terminating and confluent. This chapter
develops general coherence theorems for rewriting $2$-theories that
are not confluent and only weakly terminating, in a precise sense. The
techniques are radically different from those of Chapter
\ref{ch:unf}. The driving philosophy is that a parallel pair of
morphisms are equal if and only if they admit a subdivision, each face
of which commutes. As such, the approach is primarily through
topological graph theory, where a subdivision is defined as a certain
ambient-isotopy class of planar graph embeddings whose boundary
consists of the parallel pair of maps under investigation. The resulting
coherence theorem is also used to construct examples of finitely
presented rewriting $2$-theories that cannot be made coherent via
only finitely many coherence axioms, but are otherwise well
behaved. The related coherence problem of when there exists a decision
procedure for the commutativity of diagrams arising from a rewriting
$2$-theory is also briefly investigated.  \\

\noindent \textbf{Chapter 7:} Iterated monoidal categories
\cite{iterated} arose as a categorical model of iterated loop
spaces. As a rewriting $2$-theory, they are particularly interesting
because they possess a nontrivial equational theory on both objects
and morphisms, as well as being non-confluent. A highly technical
proof that iterated monoidal categories are coherent is given in
\cite{iterated}. After introducing iterated monoidal categories, this
chapter goes on to exploit the results of Chapter \ref{ch:nunf} in
order to obtain a new, conceptually straightforward proof of
coherence. \\

The inter-dependence between chapters is indicated in the following
Hasse diagram:

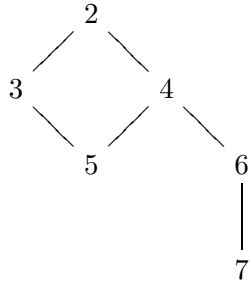
\[
\begin{xy}
  (0,0)*+{2}="2";
  (-10,-10)*+{3}="3";
  (10,-10)*+{4}="4";
  (0,-20)*+{5}="5";
  (20,-20)*+{6}="6";
  (20,-34)*+{7}="7";
  {\ar@{-} "2"; "3"};
  {\ar@{-} "2"; "4"};
  {\ar@{-} "3"; "5"};
  {\ar@{-} "4"; "5"};
  {\ar@{-} "4"; "6"};
  {\ar@{-} "6"; "7"};
\end{xy}
\]

Throughout this thesis, we read $f \of g$ as ``$f$ followed by $g$''.\blanknonumber
\chapter{Rewriting $2$-theories}\label{ch:twostructures}

Our main goal in this chapter is to define the class of
two-dimensional algebraic structures that form the  basis of the
following chapters. The definition that we develop uses a base set of
variables. This is in contradistinction with the standard approach to
two-dimensional universal algebra, which prefers a variable-free
approach via categorical constructions. The reason for choosing to
work with variables is rather utilitarian: it retains a strong link to
first-order term rewriting theory and, therefore, preserves the strong
link with various computational and linguistic constructions.  A more
pragmatic reason for our definition in terms of variables is that it
is precisely what allows us to bring various computational and
combinatorial techniques to bear on otherwise categorical
constructions. The choice of working with variables has two technical
implications. First, it makes the transition from a presentation of a
theory to a concrete algebraic structure on an arbitrary category
slightly more difficult than it otherwise might be. Second, it does
not allow us to distinguish between certain different categorical
structures. For instance, the map $\iota: A \tensor A \to A \tensor A$
gives rise to two different possible semantic interpretations: a map
that preserves the order of the factors and one that reverses the
order of the factors. Our construction 
blurs the distinction between these two semantic interpretations;
indeed, either choice would  provide an
adequate semantics for the map. It is important to note that the claim
being made here is that the two maps arise purely from two different
semantic interpretations and that there is, a priori, no syntactic way
in which to distinguish two interpretations. no More fundamentally, the combinatorial
properties of the categorical structure are unaffected by the
particular semantic interpretation of the maps. Indeed, we shall see
in this chapter that all of the different choices of semantic
interpretations yield isomorphic structures. Before jumping into the
world of two-dimensional algebra, we seek some intuition from classical
one-dimensional algebra.  

When developing a classical definition of equational varieties, one
starts with a graded set of function symbols $\F$ and imposes a
collection of equations, $\E$, on the absolutely free term algebra generated
by $\F$ on some set of variables $X$, which we denote by
$\Fr_\F(X)$. Quotienting out by the smallest congruence generated by
$\E$ on $\Fr_\F(X)$ yields the free $\langle\F|\E\rangle$-algebra on
$X$, which we denote by $\Fr_{\langle \F | \E\rangle}(X)$. It is
at this point that we run into a conceptual problem: the set of
variables $X$ holds a privileged position in the construction. If we wish
to obtain the free $\langle\F|\E\rangle$-algebra on some other set
$Y$ then we run into a problem before we even start --- the very concept of an
$\langle\F|\E\rangle$-algebra was defined with the aid of $X$! The
traditional way around this problem is to define an
$\langle\F|\E\rangle$-algebra to be an algebra $\mathbb{A}$ of type
$\F$ such that for any equation $(s,t) \in \E$ and any homomorphism
$\rho: \Fr_\F(X) \to \mathbb{A}$, we have $\rho(s) = \rho(t)$
\cite{BurrisSankappanavar}.

The viewpoint of algebras as being induced by homomorphisms from some
parti\-cular free algebra is the starting point of Lawvere theories
\cite{Lawvere:functorial}. Here, we 
consider a function symbol of arity $n$ to be a function $n \to 1$
and use the Cartesian structure of {\bf Set} in order to permute,
duplicate and delete variables as we please. This allows us to replace
equations with commutative diagrams and yields the category
$\mathrm{Th}(\langle \F | \E\rangle)$. This category has finite
products; indeed, its objects are just the natural numbers, where a
number $n$ is considered to be the $n$-fold cartesian product of
$1$. The category of finite product preserving functors 
$\mathrm{Th}(\langle \F | \E\rangle) \to {\bf Set}$ forms the analogue
of algebras {\it qua} homomorphisms in the classical case, allowing us to
transfer the structure inherent in $\langle \F|\E\rangle$ to an
arbitrary set. 

Our basic strategy in this chapter is to replicate the above arguments
in the two-dimensional setting in order to provide an abstract framework
for categories with algebraic structure definable in a
variable-based manner. Our essential objects of study are rewriting
$2$-theories,  which consist of a first order term rewriting system
modulo a two dimensional congruence. This retains a strong link with
computational structures. Indeed, syntactically, rewriting 2-theories
can be seen as a generalisation of unconditional rewriting logic,
which arose primarily in the study of concurrent systems
\cite{Meseguer_rl}. The syntax and algebraic semantics of rewriting
$2$-theories is covered in sections \ref{sect:syntax} and
\ref{sect:semantics}, respectively. Connections with rewriting logic
and other systems are briefly outlined in Section \ref{sect:relations}.

The fundamental focus of this thesis is coherence for rewriting
$2$-theories and we introduce this concept formally in Section
\ref{sect:coherence}. Subsequently, we explore the relationship
between equational varieties and coherent rewriting $2$-theories in
Section \ref{sect:categorification} before sketching some basic
results in Section \ref{sect:basic} that will be of frequent use.

\section{Syntax}\label{sect:syntax}

The purpose of this section is to introduce a general class of term
rewriting systems whose semantics correspond to categories with
an additional covariant structure. Concretely, we work with a term
rewriting theory modulo a two-dimensional congruence. That is, a term
rewriting system equipped with an equational theory on terms and an
equational theory on reductions, together with an associated calculus
of proof terms. 

Syntactically, we shall be working with structures of the form
$\rtt$, where $\F$ is a set of function symbols, $\T$ is a set of
reduction (or transformation) rules, $\E_\F$ is an equational theory
on $\F$ and $\E_\T$ is an equational theory on $\T$ containing a
certain basic congruence. Our main task in this section is to describe
the structure that this data generates, which forms our
two-dimensional analogue of $\Fr_{\langle\F|\E\rangle}$. We begin by
building the one-dimensional   aspect of the structure.

\begin{definition}[Term Algebra]
  Given a graded set of function symbols $\F := \sum_n \F_n$ and a set
  $X$, the absolutely free term algebra generated by $\F$ on $X$ is
  denoted by $\Fr_\F(X)$. 
\end{definition}

The next layer of structure adds an equational theory to
$\Fr_\F(X)$:

\begin{definition}
  Given a graded set of function symbols $\F$, a set $X$ and a set of
  equations $\E_\F$ on $\Fr_\F(X)$, we denote by
  $\Fr_{\langle\F|\E_\F\rangle}(X)$ the quotient of $\Fr_\F(X)$ by the
  smallest congruence generated by $\E_\T$. We write $[t]$ for the
  image of a term $t$ under the canonical homomorphism $\Fr_\F(X) \to
  \Fr_{\langle\F|\E_\F\rangle}(X)$. 
\end{definition}

We can now begin to describe a two-dimensional term rewriting
theory. Our first step is to define a labelled term rewriting theory. 

\begin{definition}[Labelled term rewriting theory]\label{defn:trt}
A \emph{labelled term rewriting theory} is a structure $\langle \F;
\mathcal{L}; \T \,|\, \E_\F \rangle_X$, where $\F$ is a graded set of
function symbols, $X$ is a set of variables, 
$\E_\F$ is a system of $\Fr_\F(X)$-equations, $\mathcal{L}$ is
a set of labels and $\T$ is a subset of $\mathcal{L} \times
(\Fr_{\langle\F|\E_\F\rangle}(X))^2$ satisfying the following consistency 
conditions: 
\begin{center}
  If $(\alpha,s_1,t_1)$ and $(\alpha,s_2,t_2)$ are in $\T$ then
    $s_1 = s_2$ and $t_1 = t_2$.
\end{center}
If $(\alpha,s,t) \in \T$, we write $\alpha:s \to t$. A member
of $\T$ is called a \emph{labelled reduction rule}. 
\end{definition}

Given a labelled term rewriting theory $\langle \F;
\mathcal{L}; \T \,|\, \E_\F \rangle_X$, the particular choice of
$\mathcal{L}$ is irrelevant.  What is important is simply that there 
are sufficiently many labels for the number of reduction
rules. Accordingly, we shall henceforth 
suppress explicit mention of the labels and write
$\ltrX$ for a labelled term rewriting
theory. For the remainder of this thesis, we fix an arbitrary
countable infinite set $X$ and write $\ltr$ for $\ltr_X$ when the
particular choice of variable set is unimportant. A
labelled term rewriting theory embodies the basic reductions 
that are to generate all others. The next step is to obtain an
analogue of the absolutely free term algebra for this higher
dimensional layer of structure. This is achieved by the following
definition, where the notation $\overline{x}^n$ is an abbreviation for
$x_1,\dots,x_n$ and $F(\overline{s}^n/\overline{x}^n)$ denotes the
uniform substitution of the free variables $\overline{x}^n$ by
$\overline{s}^n$.

\begin{definition}\label{def:genreds}
  Given a labelled term rewriting theory $\L :=\ltr_X$, the
  set of \emph{reductions generated by $\L$} is
  denoted $\Fr_\L(X)$ and is constructed inductively by the following rules:   
      
  \medskip

  \begin{center}
    \begin{tabular}{cl}
      \AxiomC{}
      \UnaryInfC{$1_s:[s] \to [s]$}
      \DisplayProof
      
      &  {\rm (Identity)}

      \bigskip\\
      
      \AxiomC{$\varphi_1:[s_1] \to [t_1]~ \dots~ \varphi_n:[s_n] \to [t_n]$} 
      \UnaryInfC{$F(\varphi_1,\dots,\varphi_n):[F(s_1,\dots,s_n)] \to
        [F(t_1,\dots,t_n)]$} 
      \DisplayProof
      
      & {\rm (Structure)}

      \bigskip\\
      
      \AxiomC{$\tau:[F(\overline{x}^n)] \to [G(\overline{x}^n)]$}
      \AxiomC{$(\varphi_i:[s_i] \to [t_i])_{i=1}^n$}
      \BinaryInfC{$\tau(\varphi_1,\dots,\varphi_n):[F(\overline{s}^n/
        \overline{x}^n)] \to [G(\overline{t}^n/\overline{x}^n)]$} 
      \DisplayProof
      
      & {\rm (Replacement)}

      \bigskip\\
      
      \AxiomC{$\varphi: [s] \to [u]$}
      \AxiomC{$\psi: [u] \to [t]$}
      \BinaryInfC{$(\varphi\of\psi):[s] \to [t]$}
      \DisplayProof
      
      &  {\rm (Transitivity)}

    \end{tabular}
  \end{center}
  In the {\rm (Identity)} rule, $[s] \in
  \Fr_{\langle\F|\E_\F\rangle}(X)$. In the 
  {\rm (Structure)} rule, $F$ is a function symbol of rank $n$. In the
  {\rm (Replacement)} rule $\tau$ is a reduction rule of rank
  $n$. When the particular choice of $X$ is irrelevant, we write
  $\Fr(\L)$ for $\Fr_\L(X)$.
\end{definition}

\begin{example}\label{example:associative}
  
  Let $\L$ be the labelled rewriting theory consisting  of  a single  
  binary function symbol $\tensor$, an empty equational theory on
  terms and the single reduction rule:
  \[ \alpha(t_1,t_2,t_3) : t_1\tensor(t_2\tensor t_3) \to (t_1\tensor
  t_2)\tensor t_3.\] 
  
  A derivation of $$\arrow{[A \tensor (B \tensor (C \tensor 
    D))]}{[(A\tensor B)\tensor(C\tensor D)]}$$   
  in $\Fr(\L)$ is given by:
    \begin{prooftree}
      \AxiomC{}
      \UnaryInfC{$\arrow{1_A~:~[A]}{[A]}$}
      \AxiomC{}
      \UnaryInfC{$\arrow{1_B~:~[B]}{[B]}$}
      \AxiomC{}
      \UnaryInfC{$\arrow{1_C~:~[C]}{[C]}$}
      \AxiomC{}
      \UnaryInfC{$\arrow{1_D~:~[D]}{[D]}$}
      \BinaryInfC{$\arrow{1_C\tensor 1_D~:~ [C\tensor D]}{[C\tensor
          D]}$} 
      \TrinaryInfC{$\arrow{\alpha(1_A,1_B,1_C\tensor
          1_D)~:~[A\tensor(B\tensor(C\tensor D))]}{[(A\tensor B)\tensor
          (C\tensor D)]}$} 
    \end{prooftree}
\end{example}

The consistency condition in Definition \ref{defn:trt} easily yields
the following lemma, which asserts that we may equate reductions with
their labels, thus providing a term calculus for the reductions.

\begin{lemma}
  Let $\L$ be a labelled  term rewriting theory.
    If $\alpha:s \to t$ and $\alpha:s' \to t'$ are in
    $\Fr(\L)$, then $s = s'$ and $t = t'$.\qed
    
\end{lemma}

At this point, we have in hand a notion of a labelled rewriting
theory, which corresponds to the usual abstract setting of rewriting
modulo an equational theory on terms. We now proceed to add an
equational theory on 
reductions to this framework. This allows us to consider problems
relating to equivalences of reductions in general rewriting
systems. We impose two restrictions on this structure. The first is
that we may only set two reductions to be equal if they have common
sources and targets since, in applications, we very rarely have a
sound ontological basis for equating \emph{arbitrary} reductions. The
second is that we enforce the presence of certain equations that
equate reductions differing only in the order of rewriting
nested and/or disjoint subterms. As we shall see in the following section,
this is precisely what is needed in order to ensure a sound
categorical semantics. The computational effect is to equate
orthogonal reductions --- those that do not rewrite a critical
pair. This congruence is usually dubbed the ``permutation congruence''
in the term rewriting literature \cite{terese_equivalence}. The
permutation congruence is also known as ``causal equivalence'' and the 
congruence classes that it generates correspond to the notion of
Mazurkiewicz traces arising in concurrency theory. The following
definition states these concepts more formally. 

\begin{definition}[Rewriting $2$-Theory]\label{def:rtt}
  A \emph{Rewriting $2$-Theory} is a tuple $\R := \rtt$, where
  $\ltr$ is a labelled  term rewriting theory and $\E_\T$ is a set
  of equations on $\Fr(\ltr)$ satisfying the following
  consistency condition:
  \begin{center}
    If $(\varphi_1,\varphi_2) \in \E_\T$ and $\varphi_1: [s_1] \to [t_1]$
    and $\varphi_2:[s_2] \to [t_2]$, then $[s_1] = [s_2]$ and $[t_1] = [t_2]$.
  \end{center}
  We further stipulate that the following
  equations are satisfied. We refer to these equations collectively as
  the \emph{standard congruence} and denote them by $S(\R)$ 
  \begin{center}
    \begin{tabular}{cl}
      $1_s\of\varphi = \varphi$ & {\rm (ID 1)}\\
      $\varphi\of 1_t = \varphi$ & {\rm (ID 2)}\\
      $\varphi\of(\psi\of\rho) = (\varphi\of\psi)\of\rho$ & {\rm (Assoc)}\\
      $F(\varphi_1,\dots,\varphi_n)\of F(\psi_1,\dots,\psi_n) =
      F(\varphi_1\of\psi_1,\dots,\varphi_n\of\psi_n)$ & {\rm (Funct)}\\
      $\varphi(\varphi_1,\dots,\varphi_n) =
      s(\varphi_1,\dots,\varphi_n)\of\varphi(1_{t_1},\dots,1_{t_n})$
      & {\rm (Nat 
        1)}\\
      $\varphi(\varphi_1,\dots,\varphi_n) = \varphi(1_{s_1},\dots,1_{s_n})\of
      t(\varphi_1,\dots,\varphi_n)$ & {\rm (Nat 2)}
    \end{tabular}
  \end{center}
  In the above, $F \in \F_n$ and
  $\varphi,\psi,\rho,\varphi_1,\dots,\varphi_n,\psi_i,\dots,\psi_n$ are
  reductions in $\Fr(\R)$ such that the above compositions are well defined.
\end{definition}

One of the benefits of allowing additional equations on reductions
beyond that provided by the standard congruence is that it allows
us to study invertible reduction rules, which arise when we recast an
equational theory as a rewriting system. Moreover, it provides enough
flexibility for us to be able to place equations on non-invertible
reduction rules, which model phenomena such as non-reversible
computations. 

\begin{definition}[Invertible]
  Given a rewriting $2$-theory $\rtt$, a reduction rule $\varphi:[s] \to [t]$
  in $\T$ is \emph{invertible} if there is a reduction rule
  $\psi:[t] \to [s]$ in $\T$ and a variable substitution $\sigma:X
  \to X$ such that the equations $\varphi^\sigma\of\psi^\sigma =
  1_{s^\sigma}$ and $\psi^\sigma\of\varphi^\sigma = 1_{t^\sigma}$ are both
  in $\E_\T$. 
  A rewriting $2$-theory is invertible if all of its reduction rules are
  invertible. We say that $\psi$ is an \emph{inverse} of $\varphi$. 
\end{definition}

In defining particular rewriting $2$-theories, we shall often just say
that a reduction is invertible, without explicitly giving the data for
its inverse. That is, if we say that a rewriting $2$-theory contains
an invertible reduction rule $\rho$, we mean that it also contains the
inverse $\rho^{-1}$ together with the necessary equations. Before
proceeding, we give an example of a rewriting $2$-theory.

\begin{example}\label{ex:AU}
 
  This example gives a presentation of an invertible rewriting $2$-theory
  involving associativity and unit reduction rules. We shall see in
  Section \ref{sect:semantics} that this example gives a presentation
  of the free monoidal category on a discrete category.
  
  The theory consists of  a binary
  function symbol $\otimes$ and a nullary function symbol $I$. We
  write $\otimes$ in infix notation. It has the following invertible
  reduction rules: 
\begin{eqnarray*}
      \alpha(t_1,t_2,t_3)&:& t_1\otimes(t_2\otimes t_3)
      \stackrel{\sim}{\longrightarrow} (t_1
      \otimes t_2)\otimes t_3\\
      \lambda(t)&:& I \otimes t  \stackrel{\sim}{\longrightarrow} t\\
      \rho(t)&:& t \otimes I  \stackrel{\sim}{\longrightarrow} t
    \end{eqnarray*}
    It has equations stating that the following diagrams commute: 
    \begin{eqnarray*}
  \begin{xy}
    (0,0)*+{a\tensor (b\tensor(c\tensor d)))}="1";
    (-25,-18)*+{(a\tensor b)\tensor (c \tensor d)}="2";
    (25,-18)*+{a\tensor((b\tensor c) \tensor d)}="3";
    (-25,-36)*+{((a\tensor b) \tensor c)\tensor d}="4";
    (25,-36)*+{(a\tensor (b\tensor c))\tensor d}="5";
    {\ar@{->}_<<<<<<<{\alpha}@/_/ "1"; "2"}
    {\ar@{->}^<<<<<<<{1\tensor\alpha}@/^/ "1"; "3"}
    {\ar@{->}_{\alpha} "2"; "4"}
    {\ar@{->}^{\alpha} (20,-21); (20,-32)}
    {\ar@{->}^{\alpha\tensor 1} (5,-36); (-12,-36)}
  \end{xy}
    &
    \xymatrix@d{
      {a \otimes (I \otimes b)} \ar[r]^{\alpha}
      \ar[d]_{1\otimes \lambda} & {(a \otimes I)\otimes b}
      \ar[dl]^{\rho\otimes 1}\\
      {a \otimes b}
    }
  \end{eqnarray*}\qed
\end{example}

\begin{definition}
  If $\R := \rtt$ is a rewriting $2$-theory, then 
  $[\E_\T + S(\R)]$ denotes the smallest congruence generated by
  $\E_\T$ and $S(R)$ on $\Fr(\R)$. It is generated inductively by
  the following rules: 

  \begin{center}
    \begin{tabular}{cll}
    
      \bigskip
      \AxiomC{}
      \UnaryInfC{$\varphi=\varphi$}
      \DisplayProof
      
      & {\rm (Identity)}
      
      & $\varphi \in \T$\\
      \bigskip
      
      \AxiomC{}
      \UnaryInfC{$\varphi_1 = \varphi_2$}
      \DisplayProof
      
      & {\rm (Inheritance)}

      & $(\varphi_1,\varphi_2) \in \E_\T + S(\R)$\\
      \bigskip
      
      \AxiomC{$\varphi = \psi$}
      \UnaryInfC{$\psi = \varphi$}
      \DisplayProof

      &  {\rm (Symmetry)}\\
      \bigskip
      \AxiomC{$\varphi_1=\psi_1~ \dots~ \varphi_n = \psi_n$} 
      \UnaryInfC{$F(\varphi_1,\dots,\varphi_n) = F(\psi_1,\dots,\psi_n)$} 
      \DisplayProof
    
      &  {\rm (Structure)}
      & $F \in \F_n$\\
      \bigskip
      
      \AxiomC{$\varphi_1=\psi_1~ \dots~ \varphi_n = \psi_n$}
      \UnaryInfC{$\tau(\varphi_1,\dots,\varphi_n) =
        \tau(\psi_1,\dots,\psi_n)$}  
      \DisplayProof
      
      & {\rm (Replacement)}
      & $\tau \in \T_n$\\
      \bigskip
      \AxiomC{$\varphi = \psi$}
      \UnaryInfC{$\varphi^\sigma = \psi^\sigma$}
      \DisplayProof
      & {\rm (Substitution)}
      & {\rm $\sigma$ a substitution.}\\
      
      \AxiomC{$(\varphi_1 = \psi_1): s \to u$}
      \AxiomC{$(\varphi_2=\psi_2): u \to t$}
      \BinaryInfC{$(\varphi_1\of\psi_1 = \varphi_2\of\psi_2) :s \to t$} 
      \DisplayProof
      
      &  {\rm (Transitivity)}
      
    \end{tabular}

  \end{center}
\end{definition}

All that remains is to quotient out by the congruence generated by an
equational theory on reductions.

\begin{definition}
Given a rewriting $2$-theory $\R := \rtt_X$, we use $\Fr_{\rtt}(X)$ to denote the
quotient $\Fr_{\ltr}(X)/[\E_\T + S(\R)]$.  Where explicit mention of the
set $X$ is not necessary, we write $\Fr(\R)$ for $\Fr_\R(X)$.
\end{definition}

In the following section, we investigate the semantics of rewriting
$2$-theories and establish that a rewriting $2$-theory provides a
presentation of a free structure carried by a discrete category.

\section{Semantics}\label{sect:semantics}

In this section, we shall provide a semantics for rewriting
$2$-theories akin to the semantics that Lawvere theories provide for
syntactically defined equational varieties. The appropriate
generalisation of Lawvere theories to this setting is a special case of
discrete enriched Lawvere theories --- algebraic theories on categories whose
hom-sets carry additional structure \cite{Power:enriched,
  Power:discrete}. The presense of the standard congruence on the set
of reductions is precisely what puts us in the $2$-categorical
setting. Had we omitted the requirement that $\E_\T$ contains the
standard congruence, then we would instead be in the more general setting of
sesquicategories, whose relationship with term rewriting was
investigated by Stell \cite{Stell:sesqui}. As we are in the $2$-categorical setting, the
hom-sets are themselves categories. We shall not require any deep
enriched category theory but shall make some use of the language of
$2$-dimensional categories, an introduction to which may be found in
\cite{KellyStreet:review}.

\begin{definition}[Lawvere $2$-theory]\label{def:lawvere2}
  A discrete finitary Lawvere $2$-theory is a small $2$-category $\L$
  with finite $2$-products, together with a finite-$2$-product preserving
  identity-on-objects $2$-functor $\iota: \Nat^\op \to \L$, where $\Nat$
  is the $2$-category of natural numbers and all maps between them. A map of
  discrete finitary Lawvere $2$-theories $\L \to
  \L'$ is a finite-product preserving $2$-functor $\Theta$ making the
  following diagram commute:
  \[
  \xymatrix{
    {\L} \ar[r]^{\Theta}  & {\L'}
    \\
    {\Nat^\op} \ar[u]^{\iota} \ar[ur]_{\iota'}
  }
  \]
\end{definition}

Since we shall not require any more sophisticated notion of Lawvere
$2$-theory, we use ``Lawvere $2$-theory'' to mean ``discrete finitary
Lawvere $2$-theory''. These are an alternative categorical
presentation of strongly finitary $2$-monads on $\mathbf{Cat}$, studied
in \cite{KellyLack:2monads}. The way in which to visualise a Lawvere $2$-theory
is to think of each object as $\D^n$ for some arbitrary category
$\D$ (although, strictly speaking, the objects are simply natural
numbers). The arrows are then maps $\D^n \to \D^m$ and the two-cells
are maps between such arrows. For us, all of the arrows of $\L$ will
be generated by basic arrows $\D^n \to \D$, corresponding to function 
symbols, and all of the two-cells will be generated by reduction
rules.  

A Lawvere $2$-theory is essentially a two-dimensional analogue of a
free algebra. As in the one-dimensional case, we define a category
having the structure specified by $\L$ by product-preserving functors
out of $\L$.

\begin{definition}
  A model of a Lawvere $2$-theory $\L$ in $\Cat$ is a
  finite-product preserving $2$-functor $M:\L \to \Cat$.
\end{definition}

In order to relate rewriting $2$-theories with Lawvere $2$-theories,
we need to show how to generate a Lawvere $2$-theory $\Th(\R)$ from a
given rewriting $2$-theory $\R$. This would allow us to translate the
purely syntactic $\R$ into an object that specifies an additional
structure on a category. 

In general, there is not a strictly unique way in which to construct
$\Th(\R)$, since there may be many possible ways in which to express a
given reduction rule, particularly in the case where $\R$ contains
reduction rules such as $A \tensor A \to A$. However, as we shall see,
$\Th(\R)$ is unique up to $2$-isomorphism of Lawvere $2$-theories, so
the distinction is inessential for our purposes.

\begin{definition}
Let $\R := \rtt$ be a rewriting $2$-theory. A Lawvere $2$-theory
associated to $\R$ is a Lawvere $2$-theory $\L$ containing precisely
the following structure:
\begin{enumerate}
  \item For every term $t \in \Fr(\R)$ of arity $n$, there is a
    one-cell $|t|: n \to 1$ in $\L$.
  \item For every reduction $\rho:[s]\to [t]$ in $\Fr(\R)$, there is a
    $2$-cell $|\rho|: |s| \to |t|$ in $\L$.
  \item $s = t$ if and only if $|s| = |t|$, for terms $s,t \in \Fr(\R)$.
  \item $\sigma = \tau$ if and only if $|\sigma| = |\tau|$, for
    reductions $\sigma,\tau \in \Fr(\R)$.
\end{enumerate}
\end{definition}

It is immediate from the definition that any two Lawvere $2$-theories
associated to a rewriting $2$-theory differ only in the precise way in
which the function symbols and reduction rules are represented. This
immediately implies the following lemma.

\begin{lemma}\label{lem:isotheories}
  Any two Lawvere $2$-theories associated to a rewriting $2$-theory
  $\R$ are $2$-isomorphic. \qed
\end{lemma}

In light of the previous lemma, the following is well-defined:

\begin{definition}
  $\Th(\R)$ is the Lawvere $2$-theory associated to the the rewriting
  $2$-theory $\R$. 
\end{definition}

As it stands, the relationship between $\Th(\R)$ and $\R$ is still
quite vague. In the remainder of this section, we shall see how to
construct $\Th(\R)$ from $\R$ and we shall also see that no ``extra''
equations arise from the $2$-categorical nature of $\Th(\R)$. 

Let $\R := \rtt$ be a rewriting $2$-theory and let $\L$ be the initial
Lawvere $2$-theory. That is, $\L$ contains no structure other than
that implied by the existence of a finite-product preserving
identity-on-objects functor $\iota:\Nat^\op\to \L$. For each function
symbol $F \in \F$ of arity $n$, add a one-cell $|F|:n \to 1$ to
$\L$. Extend this inductively to terms by setting:
  \[
  |t| = \begin{cases}
    |F|(|s_1|,\dots,|s_n|) & \text{if $t = F(s_1,\dots,s_n)$}\\
    |t| & \text{if $t \in \F$}\\
  \end{cases}
  \]
For each equation $(s,t) \in \E_\F$, we enforce an equality $|s| =
|t|$ by making use of the cartesian structure of $\L$. In particular,
we may make use of the following operations:

\begin{itemize}
  \item We may duplicate an object by making use of the diagonal map
    $\diag: 1 \to 2$.
  \item We may delete the left hand-side of a pair of variables by
    making use of the first projection $\pi_1:2 \to 1$.
  \item We may delete the right hand-side of a pair of variables by
    making use of the second projection $\pi_2:2 \to 1$.
  \item We may commute two variables by making use of the twist map
    $\tau: 2 \to 2$. 
\end{itemize}

As an example, suppose that $\R$ contains the binary function symbol
$\tensor$ and the equation 
\[\tensor(a,\tensor(b,c)) = \tensor(\tensor(b,b),a).\]
This equation can be represented by saying that the following diagram
commutes: 
\[
\xymatrix{
  {3} \ar[r]^{1\times\tensor} \ar[d]_{1\times \pi_1} & {2}
  \ar[r]^{\tensor} & {1}\\
  {2} \ar[d]_{\tau}\\
  {2} \ar[r]_{\diag\times 1} & {3} \ar[r]_{\tensor\times 1} & {2}
  \ar[uu]_{\tensor}
}
\]

We may interpret the above diagram as saying that the following two
deductions are equal:

\begin{center}
  \begin{tabular}{ll}
    \AxiomC{$a,b,c$}
    \RightLabel{$1\times \tensor$}
    \UnaryInfC{$a,\tensor(b,c)$}
    \RightLabel{$\tensor$}
    \UnaryInfC{$\tensor(a,\tensor(b,c))$}
    \DisplayProof
    & \qquad

    \AxiomC{$a,b,c$}
    \RightLabel{$1\times\pi_1$}
    \UnaryInfC{$a,b$}
    \RightLabel{$\tau$}
    \UnaryInfC{$b,a$}
    \RightLabel{$\diag\times 1$}
    \UnaryInfC{$b,b,a$}
    \RightLabel{$\tensor\times 1$}
    \UnaryInfC{$\tensor(b,b), a$}
    \RightLabel{$\tensor$}
    \UnaryInfC{$\tensor(\tensor(b,b),a)$}
    \DisplayProof
  \end{tabular}
\end{center}

Of course, there are other ways in which to represent the
equation. However, any choice of diagram to represent the equation
induces the same congruence on one-cells.

Next, we need to construct a two-cell $|\rho|$ in $\L$ for every
reduction $\rho \in \Fr(\R)$. We accomplish this by constructing a
two-cell $|\rho|: |s| \to |t|$ for every reduction $\rho:[s] \to [t]$
in $\T$ and extending the construction inductively to arbitrary
reductions as follows:
  \[
  |\rho| = \begin{cases}
    |\rho'|(|\sigma_1|,\dots,|\sigma_n|) & \text{if $\rho =
      \rho'(\sigma_1,\dots,\sigma_n)$}\\
    |\rho_1|\of |\rho_2| & \text{if $\rho = \rho_1\of \rho_2$}\\
    |\rho| & \text{if $\rho \in \T$}
  \end{cases}
  \]

As in the construction of a congruence in $\L$ from $\E_\F$, there
is a choice as to how to construct $|\rho|$ for a given $\rho \in
\T$. However, in light of Lemma \ref{lem:isotheories}, this particular
choice is inconsequential. Finally, we enforce the equation
$|\sigma|=|\tau|$ for every $(\sigma,\tau) \in \E_\T$.

From our construction of $\Th(\R)$, we have that any equation that
holds in $\Fr(\R)$ holds, after suitable translation, in
$\Th(\R)$. The converse result holds but is not immediately
obvious. That is, it is not clear that the fact that $\L$ is a
$2$-category does not introduce any extra equations.

Since we are only interested in models of $\L$ in $\Cat$, checking
that all of the $2$-categorical axioms are satisfied in $\Fr(\R)$
amounts to checking the axioms for functoriality and naturality. 

Let $\tau:s \to t$ be a reduction rule of rank $n$. Naturality of $\tau$
amounts to the assertion that for all reductions $\sigma_i:s_i \to
t_i$, we have:
\[
\tau(1_{s_1},\dots,1_{s_n})\of t(\sigma_1,\dots,\sigma_n) =
s(\sigma_1,\dots,\sigma_n)\of \tau(1_{t_1},\dots,1_{t_n}).
\]
This follows immediately from the combination of (Nat1) and (Nat2).

Suppose that $F \in \F_n$.  Without loss of generality, we may assume
that $F$ is binary. The functoriality of $F$ is established as
follows:

  \begin{eqnarray*}
    F(\varphi,1)\of F(1,\psi) =& F(\varphi\of 1, 1\of \psi) &\text{by
      (Funct)} \\
    =& F(\varphi,\psi) &\text{by (ID1) and (ID2)}\\
    =& F(1\of \varphi, \psi\of 1) & \text{by (ID1) and (ID2)}\\
    =& F(1,\psi)\of F(\varphi,1) & \text{by (Funct)}
  \end{eqnarray*}

Our construction of $\Th(\R)$ from $\R$ carries the message that we
may view function symbols in $\R$ as functors and reduction rules in
$\R$ as natural transformations. Thus, a rewriting $2$-theory can be
seen as giving a syntactic specification of an additional structure
carried by a category.

\begin{example}\label{ex:monoidal}
  Let $\R$ be the rewriting $2$-theory from Example
  \ref{ex:AU}. Then, $\Th(\R)$ is the Lawvere $2$-theory for monoidal
  categories. 
\end{example}

In the following section, we discuss several systems related to
rewriting $2$-theories.

\section{Relation to other systems}\label{sect:relations}

Our basic structure of a rewriting $2$-theory, $\R := \rtt$,
simultaneously generalises several other systems, which we cover in
order of increasing generality.

\begin{enumerate}
  \item {\bf First order rewriting:} If both $\E_\F$ and
    $\E_\T$ are empty and we do not impose the standard congruence,
    then we are in the setting of first order term rewriting. However,
    our construction of $\Fr(\R)$ adds in identity reductions, which
    are not usually assumed to be present in term rewriting systems.
  \item {\bf Rewriting modulo an equational theory:} If $\E_\T$ is
    empty and we do not impose the standard congruence, then we are in
    the setting of rewriting modulo an equational theory, with the
    same caveat as for standard first order term rewriting.
  \item {\bf Calculus of Structures:} If $\E_\T$ is empty and we do
    not impose the standard congruence, then we
    are also in the setting of the Calculus of Structures
    \cite{Gugl:si,GugStrass:mell}. This is a proof theoretic framework
    that extends one sided Gentzen systems with the ability for
    inference rules to act arbitrarily deeply within a sequent. 
  \item {\bf Rewriting logic:} If $\E_\T$ is empty, then we are in the
    setting of unconditional rewriting logic \cite{Meseguer_rl}. This
    system has its roots in concurrency theory and particularly in the
    notions of causal equivalence and Mazurkiewicz trace languages. 
  \item {\bf Clubs:} The notion of a fully covariant club was
    introduced by Kelly \cite{Kelly_coherence} as a unified framework
    for covariant structures carried by a category. His description of
    fully covariant clubs is very similar to our notion of a rewriting
    $2$-theory, with several points of difference. First, Kelly's
    calculus of proof terms is provided implicitly by the categorical
    setting, whereas our calculus is generated inductively. Second,
    Kelly works purely within the framework of two-dimensional
    category theory, whereas we prefer an approach via term rewriting
    systems, which highlights the connection with computational
    notions. A more substantial technical point of differentiation is that
    Kelly gives a variable-free presentation, which does not allow the
    expression of certain equations at the term level, such as the
    commutativity of a binary function symbol. Indeed, the only
    equations expressible at the term level in Kelly's setting are the
    strongly regular ones --- those equations $(s,t)$ where $\var(s) =
    \var(t)$, each variable appears precisely once in both $s$ and $t$
    and the order in which the variables appear in $s$ is the same as
    the order in which they appear in $t$. 
\end{enumerate}

In the following section, we introduce the coherence problem, which
will be our main focus throughout the thesis.

\section{Coherence}\label{sect:coherence}

There are many interrelated problems that go under the name of
``coherence''. Ultimately, all of these questions relate to describing
the free algebra generated by some algebraic structure on a
category. The original manifestation of this problem was in Mac Lane's
investigation of monoidal categories \cite{MacLane_natural}. Since all
diagrams commute in the free monoidal category on a discrete category,
it was this phenomenon that was originally associated with the term
``coherence''. This was in keeping with work in algebraic topology on
defining algebraic operations on topological spaces together with
equations that hold only up to homotopy
\cite{Stasheff:associativity}. 

It is not the case that all algebraically defined structures on
categories enjoy the same strong coherence property that monoidal
categories do. For instance, simply removing one of the coherence
axioms from the definition of a monoidal category destroys this
property. This observation led to Kelly reformulating the coherence
problem to ask which diagrams commute purely as a result of the axioms
\cite{Kelly_coherence}. However, even this question may be too strong,
for we may not be able to even decide if a \emph{given} diagram
commutes as a result of the axioms. The view that a coherence problem
is essentially concerned with deciding whether given diagrams
commute has its roots in Lambek's investigation of residuated
structures arising in mathematical linguistics \cite{Lambek:coh1}. 

The main thrust of this thesis is the investigation of various
coherence problems for structures defined by rewriting
$2$-theories. While this does not cover the complete array of possible
categorical structures, it is sufficiently broad so as to encompass
many interesting and pathological examples. In this section, we set
out precise definitions of the various coherence problems.

One difficulty that arises when investigating coherence problems is that
the commutativity of a particular diagram may have no bearing on the
question at hand. For this reason, we need to carefully define those
diagrams and reductions that are of importance for us. These are the
diagrams that are in ``general position''. That is, they contain the
maximum number of distinct variables. Before making this precise, we
need the concept of the \emph{shape} of a reduction. 
\begin{definition}
  Let $\R$ be a rewriting $2$-theory. The \emph{Shape} of a
  reduction $\alpha \in \Fr(\R)$ is defined recursively by the following:
  \begin{equation*}
    \Shape(\alpha) = 
    \begin{cases}
      \Shape(\alpha_1)\of\Shape(\alpha_2) & \text{if $\alpha =
        \alpha_1\of\alpha_2$} \\
      \tau(\Shape(\alpha_1),\dots,\Shape(\alpha_n)) & \text{if $\alpha
        = \tau(\alpha_1,\dots,\alpha_n)$}\\
      F(\Shape(\alpha_1),\dots,\Shape(\alpha_n)) & \text{if $\alpha =
        F(\alpha_1,\dots,\alpha_n)$}\\
      \circ & \text{otherwise}
    \end{cases}
  \end{equation*}
\end{definition}

In the system from Example \ref{example:associative}, we have:
\begin{equation*}
  \Shape(\alpha(1_A,1_B,1_C)) = \Shape(\alpha(1_A,1_A,1_A)) =
  \alpha(\circ,\circ,\circ)
\end{equation*}
We now need a precise definition of the variables present in a
reduction.

\begin{definition}
  Given a rewriting $2$-theory $\R$, the set of variables in a
  reduction $\alpha \in \Fr(\R)$ is defined
  recursively as follows: 
  \begin{equation*}
    \var(\alpha) =
    \begin{cases}
      \var(\alpha_1)\cup\var(\alpha_2) &\text{if $\alpha =
        \alpha_1\of\alpha_2$}\\
      \bigcup_{i=1}^n \var(\alpha_i) & \text{if $\alpha =
        \tau(\alpha_1,\dots,\alpha_n)$}\\
      \bigcup_{i=1}^n \var(\alpha_i) & \text{if $\alpha =
        F(\alpha_1,\dots,\alpha_n)$}\\
      \alpha & \text{otherwise}
    \end{cases}
  \end{equation*}
\end{definition}

Returning to Example \ref{example:associative}, we find that
\[\var(\alpha(1_A,1_B,1_C)) = \{1_A,1_B,1_C\},\] whereas
\[\var(\alpha(1_A,1_A,1_A)) = \{1_A\}.\] We can finally nail down what
we mean when we say a reduction has the maximum possible number of
variables.

\begin{definition}
  Given a rewriting $2$-theory $\R$, a
  reduction $\alpha\in \Fr(\R)$ is \emph{in general position} if 
  \begin{equation*}
    |\var(\alpha)| = \max\{|\var(\tau)| ~:~ \tau \in
    \Fr(\R) ~ \textrm{and}~ \Shape(\tau) =
    \Shape(\alpha)\}.
  \end{equation*}
\end{definition}

\begin{example}
  Consider the system from Example \ref{example:associative} augmented
  with the following reduction rule:
  \[
  \beta(x) : x\tensor x \to x
  \]
  Then, 
  $$\alpha(1_A,1_A,1_B)\of(\beta(1_A)\tensor 1_B): A \tensor (A
  \tensor B) \to A \tensor B$$ 
  is in general position, whereas
  $$\alpha(1_A,1_A,1_B)~:~A\tensor(A\tensor B) \to (A\tensor
  A)\tensor B$$ 
  is not in general position. 
\end{example}

For coherence problems, we only need to focus on those diagrams whose
reductions are all in general position. This allows us to define the
various problems that will be our focus.

\begin{definition}
  Let $\R:= \rtt$ be a rewriting $2$-theory.
  \begin{enumerate}
    \item $\R$ is \emph{Mac Lane coherent} if any two parallel
      reductions in general position in  $\Fr(\R)$ are equal.
    \item $\R$ is \emph{Lambek coherent} if there is a decision
      procedure for the commutativity of diagrams in general position in $\Fr(\R)$.
  \end{enumerate}
\end{definition}

Unfortunately, deciding whether a given $2$-theory is coherent in
either the Mac Lane or Lambek sense is often impossible.

\begin{theorem}\label{thm:undecidable}
The decision problems for Mac Lane coherence and Lambek coherence are
undecidable over the class of finitely presented rewriting $2$-theories.
\end{theorem}
\begin{proof}
  Following the work of Markov \cite{Markov:undecidable}, we know that
  many problems are undecidable for finitely presented monoids. Our
  basic strategy is to show how to encode a monoid as a rewriting
  $2$-theory.   Let $M := \langle X |R\rangle$ be a finite presentation for
  a monoid. Let $\R(M)$ be the rewriting
  $2$-theory  consisting of a single unary function symbol $F$,
  reductions $\tau_i : F(x) \to F(x)$ for every $\tau_i \in X$ and
  relations $(\omega_i,\omega_j)$ for every $(\omega_i,\omega_j) \in
  R$. If we could solve the Mac Lane coherence problems for
  $\R(M)$, then we could decide whether $M$ is trivial. Similarly, if 
  we could solve the Lambek coherence problem for $\R(M)$, then we
  could solve the word problem for $M$. Since both of these monoid
  problems are undecidable in general, so too are the associated
  coherence problems.
\end{proof}

The notions of coherence that we have introduced here are focused
entirely on the congruence present on reductions. Historically, this
arose because equations on terms can often be converted into coherent
natural isomorphisms. We explore this phenomenon in the following
section.

\section{Categorification}\label{sect:categorification}

The fundamental group of a topological space is usually defined as
the group of homotopy-equivalence classes of based loops in the
space. This definition forgets the particular relationships between
any two loops lying in a given equivalence class. An alternative
approach might be to define a group structure on the space of based
loops together with explicit homotopies between elements. This notion
of algebraic structure ``up to homotopy'' was introduced in
\cite{Stasheff:associativity} and has been extended to handle quite
general structures \cite{Rosicky:homotopy}. A problem that arises with
this approach is that one then needs to examine the relationships
between the homotopies themselves. 

Translated into our language, the above process takes a labelled
rewriting theory theory $\L := \ltr$ and replaces it with a rewriting
$2$-theory $\R(\L)$ in which each equation in
$\E_\F$ is replaced with an invertible reduction. In order to
retain the link between the $\L$ and $\R(\L)$, one needs to show that
$\Th(\L)\simeq\Th(\R(\L))$. However, this can only be the case if any two
sequences of the new invertible reductions in $\R(\L)$ having the same
source and target are equal. In other words, one needs to construct
$\R(\L)$ in such a way that it is Mac Lane coherent. 

\begin{definition}
  A \emph{categorification} of a labelled rewriting theory $\L := \ltr$
  is a rewriting $2$-theory $\R(\L) := \langle \F; \T\cup I(\E_\F)\,|\,
  \varnothing; \E_{\T\cup I(\E_\F)}\rangle$, where:
  \begin{enumerate}
    \item $I(\E_\F)$ consists of reductions $\rho_{s,t}:s \to t$
      and $\rho_{s,t}^{-1}:t\to s$ for each $(s,t) \in \E_\F$.
    \item $\E_{\T\cup I(\E_\F)}$ contains $\E_t$, as well as the equations
      \begin{eqnarray*}
        \rho_{s,t}\of\rho_{s,t}^{-1} &=& 1_s\\
        \rho_{s,t}^{-1}\of\rho_{s,t} &=& 1_t
      \end{eqnarray*}
      for each $(s,t) \in \E_\F$.
  \end{enumerate}
  $\R(\L)$ is a \emph{coherent categorification} of $\L$ if it is Mac
  Lane coherent.
\end{definition}

\begin{example}
  Monoidal categories, as defined in Example \ref{ex:AU} are a
  coherent categorification of the theory for strict monoidal
  categories. This theory consists of a binary function symbol
  $\tensor$, a nullary function symbol $I$ as well as equations
  \begin{eqnarray*}
    a\tensor (b\tensor c) &=& (a\tensor b)\tensor c\\
    a \tensor I &=& a\\
    I \tensor a &=& a
  \end{eqnarray*}
  If we take models for the theory of strict monoidal categories to be
  product preserving functors into $\mathbf{Set}$, then we recover the
  variety of monoids. 
\end{example}

A categorification of an equational theory $\E :=
\langle\F\,|\,\E_\F\rangle$, is a categorification of the labelled
term rewriting theory
$\langle\F;\varnothing\,|\,\E_\F\rangle$. Similarly, we can define
$\Th(\L)$ for a labelled rewriting $\L := \ltr$ to be $\Th(\langle
\F;\T\,|\, \E_\F, \varnothing\rangle)$. The Lawvere $2$-theory
$\Th(\E)$ associated to an equational theory $\E$ is defined
analogously.

\begin{theorem}\label{thm:categorification}
  Let $\R(\E)$ be a  categorification of the equational theory
  $\E$. There is a biequivalence of 
  $2$-categories $\Th(\R(\E)) \simeq \Th(\E)$ if and only if $\R(\E)$
  is coherent.
\end{theorem}
\begin{proof}

Let $\E$ be an equational theory and let $\R(\E)$ be a
categorification of $\E$. 

Suppose that $\R(\E)$ is coherent. For each congruence class
of $1$-cells $[s] \in \Th(\E)$, pick a distinguished element $r([s])$.
Define a pseudofunctor $F:\Th(\E) \to \Th(\R(\E))$ by:
\begin{itemize}
  \item $0$-Cells: Identity
  \item $1$-Cells: $F([s]) = r([s])$
  \item $2$-Cells: $\Th(\E)$ contains only identity $2$-cells. Define
    $F(1_{[s]}) = 1_{r([s])}$.
\end{itemize}

Next, define a $2$-functor $G: \Th(\R(\E)) \to \Th(\E)$ by:
\begin{itemize}
  \item $0$-Cells: Identity
  \item $1$-Cells: $G(s) = [s]$
  \item $2$-Cells: Since there is a $2$-Cell $s \to t$ in
    $\Th(\R(\E))$ precisely when $[s] = [t]$ in $\Th(\E)$, we can
    define $G(\rho:s \to t) = 1_{[s]}$.
\end{itemize}

It follows from the definitions that $F\of G = 1_{\Th(\E)}$. Since
$\R(\E)$ is coherent, the two legs of the following diagram commute:

\[
\xymatrix{
  {G\of F (s)} \ar[r]^{\iso} \ar[d]_{G\of F(\rho)} & {s}
  \ar[d]^{\rho}\\
  {G\of F(t)} \ar[r]^{\iso} & {t}
}
\]
It follows that $G\of F \iso 1$, so $\Th(\R(\E)) \simeq \Th(\E)$.

Conversely, suppose that $\Th(\R(\E)) \simeq \Th(\E)$. Then, there
exist functors $F:\Th(\E) \to \Th(\R(\E))$ and $G:\Th(\R(\E)) \to
\Th(\E)$ such that $F\of G \iso 1$. Suppose that $\rho_1,\rho_2:s \to t$
are a parallel pair of $2$-cells in $\Th(\R(\E))$. Then,
$F\of G(\rho_1) = F(1_{[s]}) = F\of G(\rho_2)$. Thus, $\R(\E)$ is
coherent.
\end{proof}

\begin{example}
  It follows from Theorem \ref{thm:categorification} that the
  theory for monoidal categories is biequivalent to the theory for
  strict monoidal categories. 
\end{example}

Given two rewriting $2$-theories $\R_1$ and $R_2$, we define
$\Th(\R_1) \cup \Th(\R_2) := \Th(\R_1 \cup \R_2)$. 

\begin{corollary}\label{cor:categorification}
  Let $\R := \rtt$ be a labelled rewriting theory and let 

\noindent $\langle \F;
  I(\E_\F)\,|\, \varnothing; \E_{I(F)}\rangle$ be a coherent
  categorification of $\langle \F | \E_\F\rangle$. Then
  \[\Th(\R) \simeq \Th(\langle \F; \T \cup I_{\E_\F} \,|\, \varnothing,
  \E_\T \cup \E_{I(\F)} \rangle).\]
\end{corollary}
\begin{proof}
  By Theorem \ref{thm:categorification}, we have
  \begin{eqnarray*}
    \Th(\R) &=& \Th(\langle \F; \varnothing \,|\, \E_\F;
    \varnothing \rangle \cup \langle \varnothing; \T \, |\,
    \varnothing; \E_\T\rangle)\\
    &=& \Th(\langle \F; \varnothing \,|\, \E_\F;
    \varnothing \rangle) \cup \Th(\langle \varnothing; \T \, |\,
    \varnothing; \E_\T\rangle)\\
    &\simeq& \Th(\langle \F;
    I(\E_\F)\,|\, \varnothing; \E_{I(F)}\rangle) \cup
    \Th(\langle \varnothing; \T \, |\, \varnothing;
    \E_\T\rangle)\\ 
    &=& \Th(\langle \F; \T \cup I(\E_f) \,|\, \varnothing, \E_\T \cup
  \E_{I(\F)} \rangle).
  \end{eqnarray*}
\end{proof}

The above corollary roughly states that, for a given rewriting
$2$-theory, we can switch between an equational theory on terms and a
coherent invertible theory on terms as we please. This ability is very
useful in investigating coherent structures. In the following section,
we introduce some other useful concepts for investigating coherence.

\section{Basic properties}\label{sect:basic}

This section is predominantly intended as a collection of basic
concepts and results that will prove useful throughout the thesis.

Given a rewriting $2$-theory $\R$, we shall frequently need to break
up a reduction in $\Fr(\R)$ into a composite of smaller
reductions. Since all of the reductions in $\Fr(\R)$ are generated by
a set of reduction rules, this process must ultimately
terminate. However, it is important that we have some understanding of
the resulting normal forms.

\begin{definition}[Singular]
  Let $\R:= \rtt$ be a rewriting $2$-theory. The set of singular
  reductions in $\Fr(\R)$ is denoted $\Sing(\R)$ and is generated as
  follows: 
  \begin{itemize}
    \item If $\rho \in \T_n$ and
      $[t_1],\dots,[t_n]$ are congruence classes of terms in
      $\Fr(\R)$, then $\rho(1_{t_1},\dots, 1_{t_n})$ is singular. 
    \item If $F \in \F_n$ and $\rho$ is a singular reduction and $1
      \le i \le n$, then 
      \[{F}(\overbrace{1,\dots,1}^{i-1},{\rho},
      \overbrace{1,\dots,1}^{n-i}),\]
      is singular.
  \end{itemize}
\end{definition}

\begin{example}
  In the system from example \ref{ex:AU}, the reduction
  \[
  1_a\tensor\alpha(1_b,1_c,1_d): a\tensor(b\tensor(c\tensor d)) \to
  a\tensor((b\tensor c)\tensor d)
  \]
  is singular, whereas the reduction
  \[
  \alpha(1_a,\alpha(1_b,1_c,1_d),1_e):a\tensor((b\tensor(c\tensor
  d))\tensor e) \to (a\tensor((b\tensor c)\tensor d))\tensor e
  \]
  is not singular.
\end{example}

\begin{lemma}\label{lem:irred}
  Let $\R$ be a rewriting $2$-theory. Every non-identity reduction in
  $\Fr(\R)$ is equal to a composite of finitely many singular
  reductions. 
\end{lemma}
\begin{proof}
  Let $\R$ be a rewriting $2$-theory and let $\rho$ be a reduction in
  $\Fr(\R)$. Define the rank of $\rho$ to be
  \[
  R(\rho) = \begin{cases}
    R(\rho_1) + R(\rho_2) &\text{If $\rho = \rho_1\of \rho_2$}\\
    \sum_{i=1}^n R(\tau_i) &\text{If $\rho =
      F(\tau_1,\dots,\tau_n)$}\\
    \sum_{i=1}^n R(\tau_i) &\text{If $\rho =
      \sigma(\tau_1,\dots,\tau_n)$}\\
    1 &\text{If $\rho \in \Sing(\R)$}
  \end{cases}
  \]

  We proceed by induction on $R(\rho)$ to show that $\rho$ is a
  composite of singular morphisms. If $\R(\rho) = 1$, then $\rho$ is
  singular.

  Suppose that $R(\rho) > 1$. Suppose that $\rho = \rho_1\of\rho_2$,
  where neither $\rho_1$ nor $\rho_2$ is an identity reduction. Then by
  induction each of $\rho_1$ and $\rho_2$ is a composite of finitely
  many singular reductions. Suppose that $\rho =
  \sigma(\tau_1,\dots,\tau_n)$, where $\sigma : s \to t$ and $\tau_i:s_i
  \to t_i$ are reductions in $\Fr(\R)$ such that at least one $\tau_i$
  is not an identity map. Then, by (Nat 1) in Definition
  \ref{def:rtt}, we may rewrite $\rho$ as $s(\tau_1,\dots,\tau_n)\of
  \sigma(1_{t_1},\dots, 1_{t_n})$. Since $\sigma(1_{t_1},\dots, 1_{t_n})$
  is singular by induction, we may assume that $\rho =
  F(\tau_1,\dots,\tau_n)$, where $\tau_i:[s_i] \to [t_i]$. Without
  loss of generality, suppose that $n = 2$. It follows from the
  functoriality of $F$ that $\rho = F(\tau_1,1_{s_2})\of
  F(1_{t_1},\tau_2)$. By induction, each of $\tau_1$ and $\tau_2$ is a
  composite of singular reductions. It follows then from the
  functoriality  of $F$ that $\rho$ is equal to a composite of
  $R(\rho)$-many singular reductions.
\end{proof}

In light of the above lemma, we know that any particular reduction is
equal to a composite of only finitely many singular
reductions. However, $\Fr(\R)$ might still contain an infinite
sequence of composable reductions. 

\begin{definition}[Terminating]
  A rewriting $2$-theory $\R$ is \emph{terminating} if any infinite
  sequence of composable singular reductions in $\Fr(\R)$
  contains cofinitely many identity reductions.
\end{definition}

Of particular importance in many investigations of various kinds of
term rewriting systems are those terms that are not the source of any
non-identity reduction. Often, one would like to assign such a term to
an arbitrary term.

\begin{definition}[Normal Form]
  Let $\R$ be a rewriting $2$-theory and let $[s]$ be a term in
  $\Fr(\R)$. A \emph{normal form} for $[s]$ is a term $[t]$ such that
  there is a reduction $[s] \to [t]$ in $\Fr(\R)$ and there are no
  non-identity reductions whose source is $[t]$ in $\Fr(\R)$. We say
  that $\R$ has normal forms if every term in $\R$ has a normal form.
\end{definition}

In an arbitrary rewriting $2$-theory $\R$, a given term may or may not have
a normal form. If $\R$ is terminating, then every term has at least
one normal form. In the fortunate situation where every term
in $\R$ has a \emph{unique} normal form, many investigations become
somewhat simpler. In order to guarantee this property, we need further
restrictions on $\R$.

\begin{definition}[Confluent]
  A rewriting $2$-theory $\R$ is \emph{confluent} if any diagram 
  \[
  \xymatrix{ {[t_1]} & {[s]} \ar[l]_{\rho_1} \ar[r]^{\rho_2} &
    {[t_2]}}
  \]
  in $\Fr(\R)$ can be completed into a (not necessarily commutative) square:
  \[
  \xymatrix{
    {[s]} \ar[r]^{\rho_2} \ar[d]_{\rho_1} & {[t_2]}
    \ar@{-->}[d]^{\gamma_2}\\ 
    {[t_1]} \ar@{-->}[r]_{\gamma_1} & {[u]}
  }
  \]
\end{definition}

\begin{definition}
  A rewriting $2$-theory is \emph{complete} if it is terminating and
  confluent.Otherwise it is \emph{incomplete}. 
\end{definition}

\begin{lemma}
A complete rewriting $2$-theory has unique normal forms.
\end{lemma}
\begin{proof}
  Let $\R$ be a complete rewriting $2$-theory and let $t$ be a term in
  $\Fr(\R)$. Since $\R$ is terminating, $t$ has at least one normal
  form. Suppose that $\n_1(t)$ and $\n_2(t)$ are normal forms for
  $t$. If $\n_1(t) \ne \n_2(t)$, then since $\R$ is confluent there
  must be a term $v$ and reductions $\n_i(t) \to v$ for $i \in
  \{1,2\}$ in $\Fr(\R)$, contradicting the normality of these terms.  
\end{proof}

Our investigation of coherence for rewriting $2$-theories splits into
two cases, corresponding to whether the theories are assumed to be
complete or not with the latter case being somewhat more
delicate. 

In the following chapter, we establish a link between algebraic
invariants and coherent categorifications of equational theories. This
link is particularly useful for complete rewriting $2$-theories.
\blanknonumber
\chapter{Structure monoids}\label{ch:structure}

In Theorem \ref{thm:categorification}, we saw that a coherent
categorification of an equational variety is equivalent to the
original variety in the sense that it has an equivalent Lawvere
$2$-theory. The main purpose of this chapter is to highlight how this
phenomenon arises in combinatorial algebra within the realm of
structure monoids. Later, in Chapter \ref{ch:catalan}, we shall
exploit this connection in order to construct new presentations of
some famous algebraic objects. 

Structure monoids were introduced by Dehornoy \cite{Dehornoy:varieties}
as algebraic invariants of a certain class of equational
varieties. Dehornoy subsequently showed that Higman's groups $F$ and
$V$ arise as algebraic invariants of the varieties of semigroups and
of commutative semigroups, respectively \cite{Dehornoy:thompson}. In
particular, he showed how to construct presentations of these groups
using Mac Lane's pentagon and hexagon coherence axioms for
coherently associative and commutative bifunctors.

The relations in Dehornoy's presentations consist of two
parts. First, there are the so-called geometric relations, which arise
purely from the fact that a semigroup is, in the first instance, a
magma. The second class of relations arise from the particular
equational structure of the variety at hand. In the case of
$F$, one additional class of relations are added corresponding
to the Stasheff-Mac Lane pentagon \cite{MacLane_natural} and in the
case of $V$, the presentation further contains a class of relations
corresponding to the Mac Lane hexagon, which encodes the essential
interaction between associativity and commutativity.

The goal of this chapter is to place Dehornoy's constructions in a
more general context. More precisely, we consider coherent
categorifications of equational varieties. Within this setting,
Dehornoy's geometric relations correspond to the functoriality and
naturality of the associated categorical structure with the remaining
relations arising from the coherence axioms.

We recall the definition of structure monoids in Section
\ref{sec:structure} and go on, in Section \ref{sec:presentations} to show
that a coherent categorification of an equational variety gives rise
to a presentation of the associated structure monoid. In certain
favourable situations, the structure monoid can be turned into a group
and we show that the construction of a presentation from a coherent
categorification carries over to this setting.

\section{Structure monoids}\label{sec:structure}

In this section, we recall Dehornoy's construction of an
inverse monoid associated to a balanced equational theory
\cite{Dehornoy:varieties}.  

We begin by briefly recalling and exapnding upon some definitions from
the previous chapter. For a graded set of function symbols $\F$ and a
set $X$, we denote by  
$\Fr_\F(X)$ the absolutely free term algebra generated by $\F$ on
$X$. An equational theory is a tuple $\ETh$, where $\V$ is a set
of variables, $\F$ is a graded set of function symbols and $\E_\F$ is an
equational theory on $\Fr_\F(\V)$. A map $\varphi:\V \to \Fr_\F(\V)$
is called a \emph{substitution} and it extends inductively to an
endomorphism $\Fr_\F(\V) \to \Fr_\F(\V)$. By abuse of notation, we
label this latter map by $\varphi$ as well. We use $[\V,\Fr_\F(\V)]$
to denote the set of all substitutions. For a term $s \in \Fr_\F(\V)$
and a substitution $\varphi \in [\V,\Fr_\F(\V)]$, we use $s^\varphi$
to denote the image of $s$ under $\varphi$. The \emph{support} of a
term $s$ is the set of variables appearing in it. A pair of terms
$(s,t)$ is \emph{balanced} if they have the same support and an
equational theory is balanced if every defining equation is balanced.
\begin{definition}
  Given a balanced pair of terms $(s,t)$ in $\Fr_\F(\V)$, we use
  $\rho_{s,t}$ to denote the partial function $\Fr_\F(\V) \to
  \Fr_\F(\V)$ with graph
  \[
  \{(s^\varphi,t^\varphi) ~|~ \varphi \in [\V,\Fr_\F(\V)]\}.
  \]
\end{definition}
For a balanced pair of terms $(s,t)$, the partial function
$\rho_{s,t}$ is functional since the support of $t$ is a subset of the
support of $s$. The stronger restriction that the pair is balanced is
required since we wish to utilise the inverse partial function
$\rho_{t,s}$ as well.

Given an equational theory $\E := \ETh$, we use $[\E_\F]$ to denote the
congruence generated by $\E_\F$ on $\Fr_\F(\V)$ and we use $\Fr_\E(\V)$
to denote the quotient $\Fr_\F(\V)/[\E_\F]$. Similarly, we use $[s]$ to
denote the congruence class of a term $s$ in $\Fr_\E(\V)$. It is clear
that $[u] = [\rho_{s,t}(u)]$ for any balanced
equation $(s,t) \in \E_\F$ and any term $u \in \dom(\rho_{s,t})$. However,
the collection of all partial maps $\rho_{s,t}$ for 
$(s,t) \in \E_\F$ is not sufficient to generate $[\E_\F]$, since equations
apply to subterms as well. To this end, we introduce translated
versions of the maps $\rho_{s,t}$, that apply to arbitrary subterms.

A subterm $s$ of a term $t$ is naturally specified by the node where its
root lies in the term tree of $t$, which in turn is completely
specified by the unique path from the root of $t$ to the root of $s$
in the term tree. A path in a term tree may be specified by an
alternating sequence of function symbols and numbers, where the
numbers indicate an argument of a function symbol. More formally, we
have the following situation.

For a graded set $\F := \coprod_n \F_n$, we set 
\[
A_\F := \bigcup_n\bigcup_{F \in \F_n} \{(F,1),\dots,(F,n)\}.
\]
The set of \emph{addresses associated to $\F$} is denoted by $A_\F^*$ and
is the free monoid generated by $A_\F$ under concatenation, with the
unit being the empty string $\lambda$. For a term $t \in \Fr_\F(\V)$
and an address $\alpha \in A_\F^*$, we use $\sub(t,\alpha)$ to denote
the subterm of $t$ at the address $\alpha$. Note that $\sub(t,\alpha)$
only exists if the term tree of $t$ contains the path $\alpha$ and
that $\sub(t,\lambda) = t$.

\begin{example}
  Suppose that $\F := \{F,G\}$, where $F$ is a binary function symbol
  and $G$ is a ternary function symbol. Suppose that $\V$ is a set of
  variables. Then, the term $t := F(w,G(x,y,z))$ is in
  $\Fr_\F(\V)$. The term tree of $t$ is given in Figure
  \ref{fig:termtree}.
  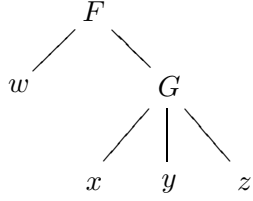
\begin{figure}[ht]
  $
  \begin{xy}
    (0,0)*+{F}="F";
    (-10,-10)*+{w}="w";
    (10,-10)*+{G}="G";
    (0,-23)*+{x}="x";
    (10,-23)*+{y}="y";
    (20,-23)*+{z}="z";
    {\ar@{-} "F"; "w"}
    {\ar@{-} "F"; "G"}
    {\ar@{-} (6,-13); (0,-20)}
    {\ar@{-} (7,-13); (7,-20)}
    {\ar@{-} (8,-13); (14,-20)}
  \end{xy}
  $
  \caption{The term tree of $F(w,G(x,y,z))$}\label{fig:termtree}
  \end{figure}
  The term $t$ has the following subterms:
  \begin{center}
    \begin{tabular}{ll}
      \medbreak
      {$\sub(t,(F,1)) = w$} & \quad {$\sub(t,(F,2)) = G(x,y,z)$} \\\medbreak
      {$\sub(t,(F,1)(G,1) = x$} & \quad {$\sub(t,(F,1)(G,2) = y$}\\
      {$\sub(t,(F,1)(G,3)) = z$}
    \end{tabular}
  \end{center}
\end{example} 

\begin{definition}[Orthogonal]
Given a graded set $\F$ and addresses $\alpha,\beta \in A_\F^*$, we
say that $\alpha$ and $\beta$ are orthogonal and write $\alpha \perp
\beta$ if neither $\alpha$ nor $\beta$ is a prefix of the other. Given
a term $t$, and addresses $\alpha$ and $\beta$, the subterms
$\sub(t,\alpha)$ and $\sub(t,\beta)$ are orthogonal if
$\alpha\perp\beta$. 
\end{definition}

Our current addressing system is sufficient to describe translated
copies of the basic operators.

\begin{definition}
  Given a graded set of function symbols $\F$, a variable set $\V$, a
  balanced pair of terms $(s,t) \in \Fr_\F(\V)$ and an address $\alpha
  \in A_\F^*$, the $\alpha$-translated copy of $\rho_{s,t}$ is denoted
  $\rho_{s,t}^\alpha$ and is the partial map $\Fr_F(\V) \to
  \Fr_\F(\V)$ defined as follows:
  \begin{itemize}
  \item A term $u \in \Fr_\F(\V)$ is in the domain of $\rho_{s,t}^\alpha$ if
    $\sub(u,\alpha)$ is defined and is in the domain of $\rho_{s,t}$.
  \item For $u \in \dom(\rho_{s,t}^\alpha)$, the image $\rho_{s,t}^\alpha(u)$
    is defined by \[\sub(\rho_{s,t}^\alpha(u),\alpha) =
    \rho_{s,t}(\sub(u,\alpha))\] and $\sub(\rho_{s,t}^\alpha(u),\beta) =
    \sub(u,\beta)$ for every address $\beta$ orthogonal to $\alpha$.
  \end{itemize}
  Note that $\rho_{s,t}^\lambda = \rho_{s,t}$.
\end{definition}

We are finally in a position to introduce the structure monoid
generated by an equational theory.

\begin{definition}[Structure Monoid]
  Given a balanced equational theory $\E := \ETh$, the \emph{structure
    monoid of $T$}, denoted $\Struct(T)$, is the monoid of partial
  endomorphisms of $\Fr_\F(\V)$ generated by the following maps under
  composition:
  \[
  \left\{ \rho_{s,t}^\alpha ~|~ (s,t) ~\text{or}~ (t,s) \in \E_\F
    ~\text{and}~ \alpha \in A_\F^* \right\}
  \]
\end{definition}

The structure monoid of an equational theory is readily seen to
completely capture the equational theory.

\begin{lemma}[Dehornoy \cite{Dehornoy:varieties}]\label{lem:capture}
  Let $\E := \ETh$ be a balanced equational theory and let $t,t' \in
  \Fr_\F(\V)$. Then $t =_{\E} t'$ if and only if there is some $\rho
  \in \Struct(\E)$ such that $\rho(t) = t'$. \qed
\end{lemma}

Given an equational theory $\E = \ETh$ and maps $\rho_{s_1,t_1},
\rho_{s_2,t_2} \in \Struct(T)$, the composition
$\rho_{s_1,t_1}\of\rho_{s_2,t_2}$ may be empty. It is nonempty
precisely when there exist substitutions $\varphi,\psi \in
[\V,\Fr_F(\V)]$ such that $t_1^\varphi = s_2^\psi$. In this case, we
say that the pair $(t_1,s_2)$ is \emph{unifiable} and that
$(\varphi,\psi)$ is a \emph{unifier} of the pair. In the case where
$(t_1,s_2)$ is not unifiable, the composition
$\rho_{s_1,t_1}\of\rho_{s_2,t_2}$ results in the empty operator, which
we denote by $\varepsilon$. Note that, for any operator $\rho \in
\Struct(T)$, we have $\rho\of\varepsilon = \varepsilon\of\rho =
\varepsilon$. The existence of the empty operator makes freely
computing with inverses in $\Struct(T)$ impossible. 

\begin{definition}[Composable]
  An equational theory $\ETh$ is \emph{composable} if any pair of terms in 
  $
  \bigcup_{(s,t)\in\E_\F} \{s,t\}
  $
  are unifiable.
\end{definition}

Recall that an inverse monoid $M$ is one in which for each element $x
\in M$, there is an element $y \in M$ such that $xyx = x$ and $yxy =
y$. Dehornoy \cite{Dehornoy:preprint} shows that $\Struct(T)$ always
forms an inverse monoid and contains the empty operator precisely
when $\E$ is not composable. One way in which to transform
$\Struct(G)$ into a group is by passing to the universal group of
$\Struct(\E)$, which we denote by $\Struct_G(\E)$, by collapsing all
idempotents to $1$. In the case where $\E$ is composable, the
idempotent elements of $\Struct(T)$ are precisely those operators that
act as the identity on their domain. A particular class of composable
theories is provided by a certain class of linear theories. Recall
that an equation $s=t$ is \emph{linear} if it is balanced and each
variable appears precisely once in both $s$ and $t$. An equational
theory is linear if each of its defining equations is linear.

\begin{lemma}[Dehornoy \cite{Dehornoy:preprint}]\label{lem:composable}
  A linear equational theory containing precisely one function symbol
  is composable. \qed
\end{lemma}

It follows from the above lemma that each linear equational
theory containing precisely one function symbol gives rise to a
structure \emph{group}.

\begin{example}\label{ex:FV}
 The equational theories for semigroups, $S$, and for commutative
 semigroups, $C$, are both linear. Since these theories involve a single
 binary operator, Lemma \ref{lem:composable} implies that they are
 composable. In this case we have
 that $\Struct_G(S)$ is Thompson's group $F$ and $\Struct_G(C)$ is
 Thompson's group $V$ \cite{Dehornoy:thompson}. 
\end{example}

In the following section, we shall see how structure monoids and
groups relate to coherent categorifications of equational varieties. 

\section{Structure monoids via coherence
  theorems}\label{sec:presentations}

The main goal of this section is to show how coherent
categorifications of equational theories give rise to presentations of
structure monoids. We 
base our analysis at the level of theories, rather than of equational
varieties. While this is seemingly at odds with Dehornoy's result
\cite{Dehornoy:varieties} that structure monoids are independent of
the particular equational presentation of a variety, differing
presentations of the same variety lead to distinct categorifications
and thence to distinct presentations of the structure monoid.

Dehornoy's utilisation of the pentagon and hexagon coherence axioms in
order to obtain presentations of Thompson's groups
\cite{Dehornoy:thompson} is indicative of a more general relationship
between structure monoids and coherent categorifications of
equational theories. The first step on the road to formalising this
relationship is to construct a monoid presentation out of a
categorification of an equational theory. In light of Lemma
\ref{lem:irred}, a good candidate for the generators of the monoid is
provided by the singular morphisms of the categorification. Since we
shall be moving back and forth between the structure monoid and a
categorification, $\R(\E)$, of an equational theory $\E$, there is
some danger of confusion about whether a symbol ``$\rho$'' lies in
$\Struct(\E)$ or in $\R(\E)$. Thus, in this section, we adopt the
convention that an element marked as ``$\widehat{\rho}$'' lies in
$\R(\E)$ and an unmarked element ``$\rho$'' lies in $\Struct(\E)$. A
second notational difficulty arises due to the differing way in which
elements of $\Struct(\E)$ and morphisms in $\Sing(\R(\E))$ are
represented. For this reason, we give a way of rewriting singular
morphisms to more closely resemble elements of $\Struct(\E)$. If
$\R(\E) = \cat$, then a reduction $\widehat{\rho} \in I(\E_\F)$ with
source $s$ and target $t$ is written as $\widehat{\rho}_{s,t}$. 

\begin{definition}[Type/Address]
  Let $\R(\E)$ be a categorification of the equational theory
  $\E$. The \emph{type}, $T(\widehat{\rho}_{s,t})$ of a singular morphism 
  $\widehat{\rho}_{s,t} \in \Sing(\R(\E))$ is defined inductively by:
  \[
  T(\widehat{\rho}_{s.t}) = \begin{cases}
    T(\widehat{\sigma}_{u,v}) & \text{if $\widehat{\rho}_{s,t} =
      F(1,\dots,1,\widehat{\sigma}_{u,v},1,\dots,1)$}\\
    \widehat{\rho}_{s,t} &\text{otherwise.}
  \end{cases}
  \]
  The address,
  $A(\widehat{\rho})$ is the word of $A^*_{\F}$ constructed as follows:
  \[
  A(\widehat{\rho}) = \begin{cases}
    (F,i)A(\widehat{\sigma}) &\text{if
      $\widehat{\rho}= F(\overbrace{1,\dots,1}^{i-1},\widehat{\sigma},
      1,\dots, 1)$}\\
    \lambda &\text{otherwise.}
  \end{cases}
  \]
\end{definition}

Given a categorification $\R(\E)$ of an equational theory $\E$, we can
now construct a monoid 
whose generators are the singular reductions of $\R(\E)$ and whose
relations are generated by the functoriality, naturality and coherence
axioms. 

\begin{definition}\label{defn:monoid}
  Let $\E := \ETh$ be a balanced equational theory and let $\R(\E) := \cat$
  be a categorification of $\E$. The monoid $\S(\R(\E))$ is the monoid
  generated by 
  \[
  \{T(\widehat{\rho})^{A(\widehat{\rho})} ~|~ \widehat{\rho} \in
  \Sing(\R(\E))\} \cup
  \{\hat{\rho}_{s,s}^{\alpha}~|~ \text{$(s,t)$ or $(t,s)$ in
    $\E_\F$ and $\alpha \in A^*_{\F}$}\} 
  \]
  if $\R(\E)$ is composable and by
  \[
  \{T(\widehat{\rho})^{A(\widehat{\rho})} ~|~ \widehat{\rho} \in
  \Sing(\R(\E_\F))\} \cup
  \{\hat{\rho}_{s,s}^\alpha~|~ \text{$(s,t)$ or $(t,s)$ in
    $\E_\F$ and $\alpha \in A^*_{\F}$}\} \cup
  \{\widehat{\varepsilon}\}
  \]
  otherwise, subject to the following relations.
  \begin{itemize}
    \item \textbf{Identity:}
      \begin{eqnarray*}
        \hat{\rho}^{\alpha}_{s,s}\of
        \hat{\rho}^{\alpha}_{s,t}  &=&
        \hat{\rho}^{\alpha}_{s,t}\\
        \hat{\rho}^{\alpha}_{s,t} \of
        \hat{\rho}^{\alpha}_{t,t} &=&
        \hat{\rho}^{\alpha}_{s,t}
      \end{eqnarray*}

      \item \textbf{Composition:} If $t_1$ and $s_2$ are not unifiable
        then 
        \[\hat{\rho}^{\alpha}_{s_1,t_1}\of
        \hat{\rho}^{\alpha}_{s_2,t_2} =
        \widehat{\varepsilon}
        \]
      \item \textbf{Empty operator}:
        \begin{eqnarray*}
          \hat{\rho}^{\alpha}_{s,t} \of
          \widehat{\varepsilon} &=& \widehat{\varepsilon}\\
          \widehat{\varepsilon} \of
          \hat{\rho}^{\alpha}_{s,t}
          &=& \widehat{\varepsilon}
        \end{eqnarray*}
      \item \textbf{Functoriality:} For
        $\alpha\perp\beta$: 
        \[
        \hat{\rho}^{\alpha}_{s,t}\of\hat{\rho}^{\beta}_{u,v}
        = \hat{\rho}^{\beta}_{u,v}\of\hat{\rho}^{\alpha}_{s,t}
        \]
      \item \textbf{Naturality:} Suppose that
        $\hat{\rho}_{s,t} \in I(\E_\F)$ is a generator and that
        some variable $x$ appears at addresses
        $\beta_1,\dots, \beta_p$ in $s$ and
        at addresses $\gamma_1,\dots,\gamma_q$ in
        $t$. Then, for all addresses
        $\alpha,\delta$ and each $\hat{\rho}_{u,v} \in I(\E_\F)$:
        \[
        \hat{\rho}_{s,t}^{\alpha}\of\hat{\rho}_{u,v}^{\alpha
          \gamma_1\delta}\of\dots\of
        \hat{\rho}_{u,v}^{\alpha\gamma_q\delta}
        =
        \hat{\rho}_{u,v}^{\alpha\beta_1\delta}\of\dots 
        \of
        \hat{\rho}_{u,v}^{\alpha\beta_p\delta}
        \of\hat{\rho}_{s,t}^{\alpha}        
        \]
      \item \textbf{Coherence:} For $(\sigma_1\of\dots\of\sigma_p,
        \tau_1\of\dots\of\tau_q) \in \E_{I(\E_\F)}$, where each $\sigma_i$ and
        $\tau_j$ is singular, set:
        \[
        T(\sigma_1)^{A(\sigma_1)}\of\dots\of T(\sigma_p)^{A(\sigma_p)}
        =
        T(\tau_1)^{A(\tau_1)}\of\dots\of T(\tau_q)^{A(\tau_q)}
        \]
  \end{itemize}
\end{definition}

The relations for functoriality and naturality in $\S(\R(\E))$ are
adapted from \cite{Dehornoy:preprint}. The functoriality relation is
precisely the requirement that each operator $F \in \F$ is
a functor. The naturality condition is, in turn, precisely the
requirement that each $\widehat{\rho} \in I(\E_\F)$ is a natural
transformation. The rather involved
addressing system in the naturality condition is due to the fact that
the same variable may appear multiple times in different positions on
either side of an equation. For naturality, one needs to apply a map
to each of these instances of the variable simultaneously. We now set
about relating $\S(\R(\E))$ to $\Struct(\E)$. 

\begin{lemma}
  Let $\E$ be a balanced equational theory and let $\R(\E)$ be a
  categorification of $\E$. Then $\S(\R(\E))$ is an inverse monoid.
\end{lemma}
\begin{proof}
  For nonempty $\hat{\rho} := \hat{\rho}_{s_1,t_1}^{{\alpha_1}}\of
  \dots \hat{\rho}_{s_k,t_k}^{{\alpha_k}}$, set
  $\hat{\rho}^{-1} := \hat{\rho}_{t_k,s_k}^{{\alpha_k}}\of
  \dots \hat{\rho}_{t_1,s_1}^{{\alpha_k}}$. Since $\rho_{s,t}$
  is the inverse of $\rho_{t,s}$, it follows that
  \begin{eqnarray*}
    \hat{\rho}\of\hat{\rho}^{-1}\of\hat{\rho} &=& \hat{\rho}\\
    \hat{\rho}^{-1}\of\hat{\rho}\of\hat{\rho}^{-1} &=& \hat{\rho}^{-1}
  \end{eqnarray*}
  Since we also have that
  $\widehat{\varepsilon}\of\widehat{\varepsilon}\of\widehat{\varepsilon}
  = \widehat{\varepsilon}$, it follows that $\S(\R(\E))$ forms an
  inverse monoid.
\end{proof}

We now know that both $\S(\R(\E))$ and $\Struct(\E)$ are inverse
monoids. Since there is a clear relationship between the generators of
each, in order to establish that they are in fact isomorphic we need
to focus on the relations. In particular, since $\R(\E)$ is an
arbitrary categorification of $\E$, it might contain inequivalent
reductions with the same source and target. Since the elements of
$\Struct(\E)$ are partial functions completely determined by their
domain and codomain, such a situation cannot occur in
$\Struct(\E)$. These considerations lead one to suspect that if we
require $\R(\E)$ to be a \emph{coherent} categorification of $\E$,
then the two monoids might in fact be isomorphic.

\begin{theorem}\label{thm:monoid}
  Let $\E$ be a balanced equational theory and let $\R(\E)$ be a
  categorification of $\E$. The following map is an 
  epimorphism of inverse monoids and it is an isomorphism if and only
  if $\R(\E)$ is coherent:
  \begin{eqnarray*}
    \S(\R(\E)) &\stackrel{\Theta}{\longrightarrow}& \Struct(\E)\\
    \hat{\rho}_{s_1,t_1}^{{\alpha_1}}\of\dots\of
    \hat{\rho}_{s_k,t_k}^{{\alpha_k}} &\longmapsto&
    \rho^{\alpha_1}_{s_1,t_1}\of\dots\of\rho^{\alpha_k}_{s_k,t_k}
  \end{eqnarray*}
\end{theorem}
\begin{proof}
  By construction, $\Theta$ is a homomorphism of inverse monoids. For
  surjectivity, we need only show that every generator
  $\rho_{s,t}^\alpha \in \Struct(\E)$ corresponds to some singular
  morphism $S(\rho^\alpha_{s,t}) \in \Sing(\R(\E))$. This singular
  morphism can be constructed recursively as follows:
  \[
  S(\rho^\alpha_{s,t}) = \begin{cases}
    F(\overbrace{1,\dots,1}^{i-1},S(\rho^\beta_{s,t}),1,\dots,1)
    & \text{if $\alpha = (\widehat{F},i)\beta$}\\
    \hat{\rho}_{s,t} &\text{if $\alpha = \lambda$}
  \end{cases}
  \]
  It remains to show that $\Theta$ is faithful if and only if
  $\R(\E)$ is coherent. 

  Suppose that $\Theta$ is faithful and let
  $\widehat{\rho_1},\widehat{\rho_2}$ be a parallel pair of morphisms
  in $\Fr(\R(\E))$. Then $\Theta(\widehat{\rho_1}) =
  \Theta(\widehat{\rho_2})$, since $\widehat{\rho_1}$ and
  $\widehat{\rho_2}$ have the same source and target. Since $\Theta$ is
  faithful, it follows that $\widehat{\rho_1} = \widehat{\rho_2}$.

  Conversely, suppose that $\R(\E)$ is coherent and that
  $\Theta(\widehat{\rho_1}) = \Theta(\widehat{\rho_2})$. Then,
  $\widehat{\rho_1}$ and $\widehat{\rho_2}$ have the same
  source and target. Since $\R(\E)$ is coherent, it follows that
  $\widehat{\rho_1} = \widehat{\rho_2}$.
\end{proof}

The above theorem is very closely linked with
Theorem \ref{thm:categorification}. The essential insight is that
$\Struct(\E)$ is simply a monoid encoding of $\Fr(\E)$, while
$\S(\R(\E))$ is a monoid encoding of $\Fr(\R(\E))$. In order to extend
this correspondence to structure groups, we need to modify our
presentations slightly.

\begin{definition}
  Let $\E$ be a balanced composable equational theory and let
  $\R(\E)$ be a categorification of $\E$. The group $\S_G(\R(\E))$
  is generated by
  \[
  \{T(\widehat{\rho})^{A(\widehat{\rho})} ~|~ \widehat{\rho} \in
  \Sing(\R(\E))\},
  \]
  subject to the functoriality, naturality and coherence relations
  from Definition \ref{defn:monoid}, together with the following
  relation:
  \[
  (\hat{\rho}^{{\alpha}}_{s,t})^{-1} =   \hat{\rho}^{{\alpha}}_{t,s}.
  \]
\end{definition}

Following the same line of reasoning as in the proof of Theorem
\ref{thm:monoid}, we obtain the following relationship between
$\S_G(\R(\E))$ and $\Struct_G(\E)$.

\begin{theorem}\label{thm:group}
  Let $\E$ be a balanced, composable equational theory and let
  $\R(\E)$ be a categorification of $\E$. The following map is an
  epimorphism of groups and it is an isomorphism if and only if
  $\R(\E)$ is coherent:
  \begin{eqnarray*}
    \S_G(\R(\E)) &\stackrel{\Theta}{\longrightarrow}& \Struct_G(\E)\\
    \hat{\rho}_{s_1,t_1}^{{\alpha_1}}\of\dots\of
    \hat{\rho}_{s_k,t_k}^{{\alpha_k}} &\longmapsto&
    \rho^{\alpha_1}_{s_1,t_1}\of\dots\of\rho^{\alpha_k}_{s_k,t_k}
  \end{eqnarray*} \qed
\end{theorem}

\begin{example}\label{ex:DehornoyPresentations}
  As we saw in Example \ref{ex:FV}, the structure group for semigroups
  is Thompson's group $F$ and the structure group for commutative
  semigroups is Thompson's group $V$, which is the first known
  finitely presented infinite simple group. It follows from Theorem
  \ref{thm:group} and Mac Lane's coherence theorem for monoidal
  categories \cite{MacLane_natural} that we may construct a
  presentation for $F$ using the pentagon coherence diagram displayed
  in Example \ref{ex:AU}. A categorification of the theory of
  commutative semigroups contains an invertible reduction rule $\tau:
  a\tensor b \stackrel{\sim}{\longrightarrow} b \tensor a$. It
  follows from Mac Lane's results 
  \cite{MacLane_natural} that a coherent categorification of the
  theory is provided by requiring $\tau\of\tau = 1$, together with the
  pentagon axiom and the hexagon axiom, which states that the
  following diagram commutes: 
  \[
  \begin{xy}
    (0,0)*+{a\tensor (b\tensor c)}="1";
    (35,0)*+{(b\tensor c)\tensor a}="2";
    (70,0)*+{b\tensor (c\tensor a)}="3";
    (70,-18)*+{b\tensor (a\tensor c)}="4";
    (0,-18)*+{(a\tensor b)\tensor c}="5";
    (35,-18)*+{(b\tensor a)\tensor c}="6";
    {\ar@{->}^{\tau} "1"; "2"}
    {\ar@{->}^{\alpha^{-1}} "2"; "3"}
    {\ar@{->}^{1\tensor\tau} "3"; "4"}
    {\ar@{->}_{\alpha} "1"; "5"}
    {\ar@{->}_{\tau\tensor 1} (7,-18); (20,-18)}
    {\ar@{->}_{\alpha^{-1}} (40,-18); (53,-18)}
  \end{xy}
  \]
  This coherence theorem allows us, as a result of Theorem
  \ref{thm:group}, to construct a presentation of Thompson's group
  $V$. These presentations for $F$ and $V$ are the same as those
  constructed by Dehronoy \cite{Dehornoy:thompson}.
\end{example}

Paraphrasing theorems \ref{thm:monoid} and \ref{thm:group}, whenever
we have a coherent categorification of a balanced equational theory,
we automatically have a presentation of the associated structure
monoid or group. As we shall see in Chapter \ref{ch:catalan}, this is a
reasonably powerful result, allowing us to obtain presentations of
certain important infinite groups. However, before we can embark upon
that investigation, we need a way of constructing coherent
categorifications and proving that a given rewriting $2$-theory is
coherent. The following chapter solves these problems for rewriting
$2$-theories that are terminating and confluent, which is precisely the
situation that arises in Chapter \ref{ch:catalan}.\blanknonumber
\chapter{Coherence for complete theories}\label{ch:unf}

In order to obtain a general coherence theorem for rewriting
$2$-theories, one needs to find distinguishing features of the
underlying rewriting theory that make the investigation tractable. As
a first port of call, one might examine Mac Lane's proof of coherence
for monoidal categories \cite{MacLane_natural}. Looking at this
proof from the angle of rewriting theory, one notices several
things. First, every reduction rule in the structure, presented in
Example \ref{ex:AU}, is invertible. Second, an analysis of the
rewriting system consisting of only the positive maps $\alpha,
\lambda, \rho$ reveals that this subtheory is complete. Finally, one
only needs to show that each term has a unique 
reduction to its unique normal form in order to rapidly conclude
coherence. An approach along these lines is used by Johnson
\cite{Johnson:thesis} in order to develop a general coherence theorem
for pasting diagrams in $n$-categories.  

Similar considerations led Melli\`{e}s to formulate the notion of
``universal confluence'' for a term rewriting theory within his
framework of axiomatic rewriting theory  \cite
{Mel:residual}. In order to formulate this concept within our setting,
we require the notion of a commuting joining.

\begin{definition}[Span]
  A \emph{span}, $S$, in a rewriting $2$-theory $\R$ is a diagram of
  the form
  \[\Span\]
  in $\Fr(\R)$. A \emph{joining} of $S$ is a
  term $t$ of $\Fr(\R)$ together with reductions $\psi_1:u_1 \to t$
  and $\psi_2:u_2 \to t$ in $\Fr(\R)$. Pictorially, a joining is:
  \[
  \vcenter{
    \xymatrix{
      {s} \ar[d]_{\varphi_1} \ar[r]^{\varphi_2} & {u_2}
      \ar[d]^{\psi_2}\\
      {u_1} \ar[r]_{\psi_1} &  {t}
    }
  }
  \]
  We call $S$ \emph{joinable} if a joining of $S$ exists and we call
  $S$ commuting-joinable if a joining exists such that the above
  diagram commutes.
\end{definition}

Universal confluence is intended to capture the strong version of
confluence present within monoidal categories, which coincides with
the presence of pushouts in the free monoidal category on a discrete
category. More specifically, it may be described as follows:

\begin{quote}
For every span
 $\Span$ and for $i \in \{1,2\}$, there is a commuting joining
 $\psi_i: u_i \to t$  such that for any other commuting
 joining $\tau_i: u_i \to t$, there is a unique map $\rho:t \to 
v$ making the following diagram commute: 

  \begin{displaymath}
  \vcenter{
    \def\objectstyle{\scriptstyle}
    \def\labelstyle{\scriptstyle}
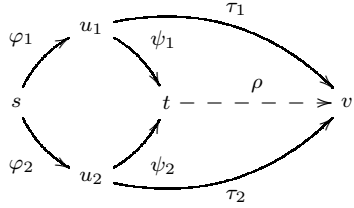
    \begin{xy}
      (0,0)*+{s}="s";
      (10,10)*+{u_1}="u1";
      (10,-10)*+{u_2}="u2";
      (20,0)*+{t}="t";
      (44,0)*+{v}="ns";
      {\ar@{->}@/^0.4pc/^{\varphi_1} "s"; "u1"};
      {\ar@{->}@/_0.4pc/_{\varphi_2} "s"; "u2"};
      {\ar@{->}@/^0.4pc/^-{\psi_1} "u1"; "t"};
      {\ar@{->}@/_0.4pc/_-{\psi_2} "u2"; "t"};
      {\ar@{->}@/^1.3pc/^{\tau_1} "u1"; "ns"};
      {\ar@{->}@/_1.3pc/_{\tau_2} "u2"; "ns"};
      {\ar@{-->}^{\rho} "t"; "ns"};
    \end{xy}
    }
  \end{displaymath}
\end{quote}

For a general rewriting $2$-theory, the map $\rho$ in the above
diagram does not necessarily exist. However, whenever every reduction
rule is invertible and the positive subtheory has unique normal forms,
as is the case for monoidal categories, we can construct $\rho$ quite
easily. Indeed, since there are maps $s \to t$ and $s \to v$, both $t$
and $v$ must have the same normal form $\n(s)$. This means that there
is a map $N_t : t \to \n(s)$ and  a map $N_v : v \to \n(s)$ and we may
simply take $\rho$ to be $N_t \of N_v^{-1}$.

When $\R$ contains non-invertible rules, the existence of $\rho$ is no
longer guaranteed. Surprisingly though, the invertibility of the rules
is not crucial for coherence. This was first demonstrated by Laplaza's
coherence theorem for categories with a directed associativity map
$\alpha: a \tensor (b \tensor c) \to (a \tensor b) \tensor c$ that is
not necessarily invertible \cite{Laplaza:associative}. Remarkably,
the only coherence axiom required for this result is Mac Lane's
pentagon --- precisely what is required in the invertible case. 

We are now in the situation of needing to discern conditions on a
confluent and terminating rewriting $2$-theory that ensure Mac Lane
coherence. Our approach needs to be delicate enough to handle both the
invertible and the non-invertible case, since the same coherence
axioms usually suffice for both. Ultimately, we shall end up with a
slightly weaker and more general concept than universal confluence,
essentially not requiring the existence of the map $\rho$. 

Our approach requires some classical tools and lemmas from first order
term rewriting theory and we briefly cover the required material in
Section \ref{sec:classical}. In Section \ref{sec:coherenceunf}, we
develop a practical general coherence theorem for complete rewriting
$2$-theories and extend this result to invertible theories in Section
\ref{sec:coherenceinv}. 

\section{Classical lemmas}\label{sec:classical}

The focus of this section is on several classical lemmas that make the
examination of confluence for finitely presented rewriting theories
tractable. This analysis essentially reduces to enumerating over the
possible ways in which two reductions can diverge within the
theory. In other words, what we seek is some sort of classification of
all possible spans that can arise from the theory. In light of Lemma
\ref{lem:irred}, we can begin by focussing our attention on singular
reductions. 

\begin{definition}
  A span $\Span$ is \emph{singular} if both $\varphi_1$ and
  $\varphi_2$ are singular.
\end{definition}

In the case where the rewriting theory is terminating, Newman's Lemma
reduces confluence to showing that every singular span is joinable. 

\begin{lemma}[Newman's Lemma \cite{Newman:lemma}]
  A terminating rewriting $2$-theory is confluent if every singular
  span is joinable. \qed
\end{lemma}

A singular span $\Span$ may take one of three forms:

\begin{itemize}
  \item $\varphi_1$ and $\varphi_2$ rewrite disjoint subterms of $s$.
  \item $\varphi_1$ and $\varphi_2$ rewrite nested subterms of $s$.
  \item $\varphi_1$ and $\varphi_2$ rewrite overlaping subterms of $s$.
\end{itemize}

In practice, it is the rewriting of overlapping subterms of $s$ that
can lead to non-confluence. It is, therefore, important to define
precisely what we mean when we say that two reductions overlap. Before
we do this, we need to identify all possible places where a reduction
rule could apply.

\begin{definition}[Redex]
  Let $\L := \ltr_X$ be a labelled rewriting theory and let $\rho:[s]
  \to [t]$ be a reduction rule in $\T$. For a substitution $\sigma: X
  \to \Fr_\F(X)$  and a term $u \in [s]$, the term $u^\sigma$ is
  called a \emph{$\rho$-redex}.
\end{definition}

We are now in a position to define overlapping reduction rules.

\begin{definition}[Overlap]
Let $\L := \ltr_X$ be a labelled term rewriting theory and let $t \in
\Fr_\F(X)$. Two subterms $t_1,t_2$ of $t$ \emph{overlap} if they share
at least one function symbol occurence. Two reduction rules
$\rho_1:[s_1] \to [t_1]$ and $\rho_2:[s_2] \to [t_2]$ in $\T$   
\emph{overlap} if there is a term $t$ containing instances of a
$\rho_1$-redex $r_1$ and a $\rho_2$-redex $r_2$ such that $r_1$ and
$r_2$ overlap. We do not count the trivial overlap between a redex $r$
and itself unless $r$ is a redex of two different reduction rules.
\end{definition}

\begin{example}
In the positive subtheory of the theory for monoidal categories given
in Example \ref{ex:AU}, the reduction rules $\rho$ and $\lambda$
overlap on the term $a \tensor (I \tensor b)$ and the reduction rule
$\alpha$ overlaps nontrivially with itself on the term $a \tensor (b
\tensor (c \tensor d))$.
\end{example}

With Newman's Lemma in mind, we now restrict our focus to singular
reductions that rewrite overlapping terms. Unfortunately there may be
infinitely many such spans, even for finitely presented
theories. However, if we know that a certain span is joinable, then we
automatically know that all substitution instances of it are
joinable. Therefore, we can refocus our investigation on finding a
minimal set of overlapping spans $S$ such that any overlapping span is
a substitution-instance of a member of $S$. 

\begin{definition}
  Let $\F$ be a graded set of function symbols. Given terms $t,u \in
  \Fr_\F(X)$, we say that $u$ is an \emph{instance} of $t$ if there is
  a substitution $\sigma$ such that $u = t^\sigma$.  A term $v \in
  \Fr_\F(X)$ is a \emph{common instance} of the terms $t,u \in
  \Fr_\F(X)$ if it is an instance of both $t$ and $u$. The term $v$ is
  the \emph{most general common instance} of $t$ and $u$ if any other
  common instance of $t$ and $u$ is also an instance of $v$.
\end{definition}

Two terms may not have a common instance but when they do, they are
guaranteed to have a most general common instance. The reader may find
a proof of the following lemma in \cite{Dershowitz:handbook}.

\begin{lemma}\label{lem:mgci}
  Let $\F$ be a graded set of function symbols. If $t,u \in \Fr_\F(X)$
  have at least one common instance, then they have a most general
  common instance.  \qed
\end{lemma}

Given two overlapping reductions, we can bootstrap the notion of most
general common instance in order to obtain the ``most general'' way in
which the two reduction rules can overlap. Before we do this, however,
we need to know precisely how two reduction rules overlap. This
information is provided by the following lemma, a proof of which may
be found in \cite[Lemma 2.7.7]{KlopdeV:FO}. 

\begin{lemma}\label{lem:overlap}
  Two reduction rules $\rho_1:s_1 \to t_1$ and $\rho_2:s_2 \to t_2$ overlap if
  and only if there is a non-variable subterm of $s_1$ that can be
  matched with a $\rho_2$-redex or a non-variable subterm of $s_2$
  that can be matched with a $\rho_1$-redex. \qed
\end{lemma}

In order to facilitate our definition of the ``most general'' overlap
of two reduction rules, we need a way of specifying a distinguished
subterm of a term. To this end, we use the notation $t\{s\}$ to denote
a term $t$ with a distinguished subterm $s$. We may apply rewrites
directly to the subterm $s$. If $\rho:s \to s'$ is some reduction,
then we may apply $t\{\rho\} : t\{s\} \to t\{s'\}$. 

\begin{definition}[Critical span]
  Consider a pair of overlapping reduction rules $\rho_1:[\ell_1] \to
  [r_1]$ and $\rho_2:[\ell_2] \to [r_2]$. By Lemma \ref{lem:overlap}, we may
  assume that $\ell_1 = 
  t\{u\}$ and that there are substitutions $\sigma,\tau$ such that
  $u^\sigma = \ell_2^\tau$. By Lemma \ref{lem:mgci}, we may assume that
  $u^\sigma = \ell_2^\tau$ is a most general common instance of $u$ and
  $\ell_2$. Then, the following span arising from this overlap is called
  a \emph{critical span}:
  \[
  \xymatrix@1{
    {[r_1^\sigma]} & {[t^\sigma\{u^\sigma\}]} \ar[l]_{\rho_1^\sigma}
    \ar[r]^{t\{\rho_2^\tau\}} & {[t\{r_2^\tau\}]}}
\]
\end{definition}

\begin{example}\label{ex:monoidalcrit}
  Consider the positive theory for monoidal categories given in
  Example \ref{ex:AU}. We then have the following reduction rules:

  \begin{eqnarray*}
    \alpha(t_1,t_2,t_3)&:& t_1\otimes(t_2\otimes t_3) \to (t_1
    \otimes t_2)\otimes t_3\\
    \lambda(t)&:& I \otimes t \to t\\
    \rho(t)&:& t \otimes I \to t
  \end{eqnarray*}
  
  By Lemma \ref{lem:overlap}, in order to find all overlaps between
  the reduction rules, we need only insert redexes of reduction rules
  as subterms of redexes of other reduction rules. 

  The reduction rule $\alpha$ contains two instances of
  $\tensor$. Thus, it overlaps nontrivially with itself and leads to
  the following critical span:

    \begin{equation}\label{monoidal:assoc}
    \vcenter{
      \def\labelstyle{\scriptstyle}
      \begin{xy}
        (0,0)*+{a\tensor(b\tensor(c\tensor d))}="a";
        (-20,-14)*+{(a\tensor b)\tensor (c\tensor d)}="b";
        (20,-14)*+{a\tensor((b\tensor c)\tensor d)}="c";
        {\ar@{->}@/^0.8pc/^-{1 \tensor \alpha} "a";"c"};
        {\ar@{->}@/_0.8pc/_-{\alpha} "a"; "b"};
      \end{xy}
    }
    \end{equation}
    
    Furthermore, $\alpha$ overlaps with $\lambda$ and $\rho$ in three
    possible ways, leading to the following critical spans:

    \begin{eqnarray}
      &\xymatrix@1{
        {(I \tensor b) \tensor c} & {I \tensor (b \tensor c)}
        \ar[l]_{\alpha} \ar[r]^{\lambda} & {a \tensor b}
      }&\label{monoidal:unit1}\\
      &\xymatrix@1{
        {(a \tensor I) \tensor c} & {a \tensor (I \tensor c)}
        \ar[l]_{\alpha} \ar[r]^{1 \tensor \lambda} & {a \tensor b}
      }&\label{monoidal:unit2}\\
      &\xymatrix@1{
        {(a \tensor b) \tensor I} & {a \tensor (b \tensor I)}
        \ar[l]_{\alpha} \ar[r]^{1 \tensor \rho} & {a \tensor b}
      }&\label{monoidal:unit3}
    \end{eqnarray}
    
    Finally, $\lambda$ and $\rho$ overlap with each other, leading to
    the following critical span:
    
    \begin{equation}\label{monoidal:unit4}
      \xymatrix@1{
        {I} & {I \tensor I} \ar[l]_{\lambda} \ar[r]^{\rho} & {I}
      }
    \end{equation}

    This exhausts all of the critical spans arising in the theory.
\end{example}

As mentioned previously, the main utility of critical spans is that
they drastically reduce the number of spans we need to check for
joinability when investigating confluence. This result is embodied in
the critical pairs lemma, so named because critical spans are usually
identified with their pair of reduced terms. The reader may find a
proof of the Lemma in \cite[Lemma 2.7.15]{KlopdeV:FO}.

\begin{lemma}[Critical Pairs Lemma]
  Let $\L$ be a labelled rewriting theory. Every singular span
  in $\L$ is joinable if and only if every critical span in $\L$ is
  joinable. \qed
\end{lemma}

In the following section, we develop a general coherence theorem for
rewriting $2$-theories having unique normal forms. Our basic strategy
is to obtain versions of Newman's Lemma and the Critical Pairs Lemma
that take into account the commutativity of the diagrams
involved. This leads to some additional subtleties, but the basic
strategy remains close to this section.

\section{Coherence for directed theories}\label{sec:coherenceunf}

In this section, we develop a coherence theorem for terminating and
confluent rewriting $2$-theories. For this, we shall need to refine
our notion of confluence.

\begin{definition}
  Let $\R$ be a rewriting $2$-theory. A span $S$ in $\Fr(\R)$ is
  \emph{commuting-joinable} if there is a joining of $S$ that commutes
  in $\Fr(\R)$. We say that $\R$ is commuting-confluent if every span in
  $\Fr(\R)$ is commuting-joinable and we say that $\R$ is
  \emph{locally commuting-confluent} if every singular span in
  $\Fr(\R)$ is commuting-joinable. 
\end{definition}

We are now in a position to obtain a strong form of Newman's Lemma
that includes information on the commutativity of diagrams.

\begin{lemma}[Strong Newman's Lemma]\label{lem:newman}
  Let $\R$ be a terminating, locally commuting-confluent
  finitely presented rewriting $2$-theory. Then: 
  \begin{enumerate}
    \item Every term $s \in \Fr(\R)$ has a unique normal form
      $\n(s)$.
    \item Any two reductions from $s$ to $\n(s)$ in $\Fr(\R)$ are
      equal.
  \end{enumerate}
\end{lemma}
\begin{proof}
  Part (1) follows from the classical Newman's Lemma, since $\R$ is
  terminating and confluent. 

  For Part (2), suppose that $s$ is a term in $\Fr(\R)$ and let
  $\varphi$ and $\psi$ be two reductions from $s$ to $\n(s)$ in
  $\Fr(\R)$. Then, it follows from Lemma \ref{lem:irred} that
  $\varphi = \varphi_1 \of \varphi_2$, where $\varphi_1$ is
  singular. Similarly, $\psi = \psi_1\of\psi_2$, where $\psi_1$ is
  singular. Suppose that $\varphi_1:s \to u_1$ and $\psi_1:s \to
  u_2$. Then, these two arrows form a singular span, which by
  assumption has a commuting joining $\tau_i:u_i\to t$, where $t$ is a
  term in $\Fr(\R)$ and $i \in \{1,2\}$. Since there is a reduction
  $s \to t$ in $\Fr(\R)$, it follows from Part (1) that $\n(t) =
  \n(s)$ and there is a reduction $\rho:t \to \n(s)$. For a term $a$,
  let $\mu(a)$ be the length of the longest reduction from $a$ to
  $\n(a)$ in $\Fr(\R)$ that does not contain an identity
  reduction. This is well defined since $\R$ is terminating and
  finitely presented. We proceed by induction on $\mu(s)$ to show that
  $\varphi = \psi$ by showing that the 
  following diagram commutes in $\Fr(\R)$:
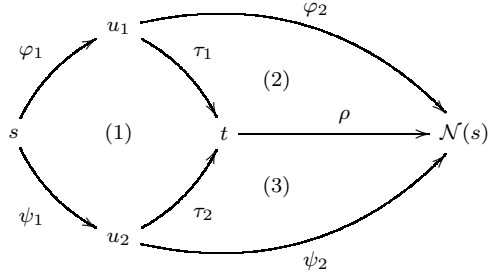
  \begin{displaymath}
  \vcenter{
    \def\objectstyle{\scriptstyle}
    \def\labelstyle{\scriptstyle}
    \begin{xy}
      (0,0)*+{s}="s";
      (14,14)*+{u_1}="u1";
      (14,-14)*+{u_2}="u2";
      (28,0)*+{t}="t";
      (60,0)*+{\n(s)}="ns";
      (14,0)*+{(1)};
      (35,7)*+{(2)};
      (35,-7)*+{(3)};
      {\ar@{->}@/^0.5pc/^{\varphi_1} "s"; "u1"};
      {\ar@{->}@/_0.5pc/_{\psi_1} "s"; "u2"};
      {\ar@{->}@/^0.5pc/^{\tau_1} "u1"; "t"};
      {\ar@{->}@/_0.5pc/_{\tau_2} "u2"; "t"};
      {\ar@{->}@/^1.8pc/^{\varphi_2} "u1"; "ns"};
      {\ar@{->}@/_1.8pc/_{\psi_2} "u2"; "ns"};
      {\ar@{->}^{\rho} "t"; "ns"};
    \end{xy}
    }
  \end{displaymath}
  
  If $\mu(s) = 0$, then $s = \n(s)$ and $\varphi=\psi=1_s$. Suppose that
  $\mu(s) > 0$. Without loss of generality, we may assume that neither
  $\varphi_1$ nor $\psi_1$ is $1_s$. Then, since there is a reduction from
  $s$ to $u_1$ and one from $s$ to $u_2$, it follows from Part (1) that
  $\n(u_1) = \n(u_2) = \n(s)$. Hence, $\mu(u_1) < \mu(s)$ and
  $\mu(u_2) < \mu(s)$ and it follows from induction that the subdiagrams
  labelled $(2)$ and $(3)$ in the diagram
  above commute. Since the diagram labelled $(1)$
  commutes by assumption, we have that $\varphi = \psi$. 
\end{proof}

By the preceding lemma, we know that each term in a terminating,
locally commuting-confluent and finitely presented rewriting
$2$-theory has a unique reduction to a unique normal form. In order to
pass from this fact to a general coherence theorem, we need a way of
extending this result to arbitrary parallel pairs of reductions. The
property that turns out to be most useful for achieving this is for
every reduction to be \emph{monic}. 

\begin{definition}
  A rewriting $2$-theory $\rtt$ is \emph{monic} if whenever
  $(\varphi_1\of\psi,\varphi_2\of\psi) \in \E_\T$ modulo the basic
  congruence, we have $(\varphi_1,\varphi_2)\in \E_\T$. 
\end{definition}

Recall that an arrow $g:b \to c$ in a category $\C$ is called ``monic'' if
for every pair of arrows $f_1,f_2:a \to b$ in $\C$, if $f_1\of g =
f_2\of g$ then $f_1 = f_2$. The following lemma follows immediately
from the construction of $\Fr(\R)$ for a rewriting $2$-theory
$\R$.

\begin{lemma}\label{lem:monic}
  If $\R$ is a monic rewriting $2$-theory, then every arrow in
  $\Fr(\R)$ is monic. \qed
\end{lemma}

We now have all the necessary ingredients for a general coherence
theorem.

\begin{theorem}[Coherence]\label{thm:coherence}
  A finitely presented rewriting $2$-theory is Mac Lane coherent if it
  is  monic, terminating and locally commuting-confluent. 
\end{theorem}
\begin{proof}
  Let $\R$ be a rewriting $2$-theory satisfying the
  hypotheses. Suppose that $\tau_1$ and $\tau_2$ are two reductions $s
  \to t$ in general position in  $\Fr(\R)$. By Lemma
  \ref{lem:newman}, there is a unique arrow 
  $\rho:t \to \n(t)$ and $\n(s) = \n(t)$. So, $\tau_1 \of \rho$ and
  $\tau_2 \of \rho$ are two arrows $s \to \n(s)$ in
  $\Fr(\R)$. Lemma \ref{lem:newman} implies that $\varphi_1\of \rho =
  \varphi_2 \of \rho$. Since $\R$ is monic, it follows from Lemma
  \ref{lem:monic} that  $\varphi_1 = \varphi_2$. 
\end{proof}

Theorem \ref{thm:coherence} effectively reduces the problem of showing
that a rewriting $2$-theory is coherent to showing that the underlying
term rewriting system is terminating and confluent. In order to make
effective use of this coherence theorem, we establish a strong form of
the Critical Pairs Lemma.

\begin{lemma}[Strong Critical Pairs Lemma]\label{lem:cp}
  A rewriting $2$-theory is locally commuting-confluent if and only
  if every critical span is commuting-joinable.
\end{lemma}
\begin{proof}
  By definition, every critical span in a locally commuting-confluent
  rewriting $2$-theory is commuting-joinable. For the converse
  direction, let $\R$ be a rewriting $2$-theory in which every
  critical span is commuting-joinable. Let $S := \xymatrix@1{{u} & {s} 
    \ar[l]_{\varphi}\ar[r]^{\psi} & {v}}$ be a singular span in
  $\Fr(\R)$. We can distinguish three possibilities for this span: 
  \begin{enumerate}
    \item $\varphi$ and $\psi$ rewrite \emph{disjoint subterms} of
      $s$. Without loss generality, we may assume that $s =
      F(t_1,t_2)$, that $\varphi =  F(\varphi',1_{t_2})$ and that $\psi =
      F(1_{t_1},\psi')$; where $\varphi':t_1 \to
      t_1'$ and $\psi':t_2 \to t_2'$. Then, we have 
      $$F(\varphi',\psi'): F(t_1,t_2) \to F(t_1',t_2').$$
      By the functoriality of $F$, we
      have 
      \begin{eqnarray*}
       F(\varphi',1)\of F(1,\psi') &=& F(\varphi'\of 1, 1\of \psi')\\
       &=& F(1 \of \varphi', \psi'\of 1)\\
       &=& F(1,\psi')\of F(\varphi',1)
      \end{eqnarray*}
      So, $S$ is commuting-joinable to $F(t_1',t_2')$.
    \item $\varphi$ and $\psi$ rewrite \emph{nested subterms} of
      $s$. Without loss of generality, we may assume that $s = p\{q\}$,
      that $\varphi = \varphi'\{1_q\}$ and that $\psi = 1_p\{\psi'\}$; where
      $\varphi':p \to p'$ and $\psi: q \to q'$. Then we have
      $$ \varphi'\{\psi'\}:p\{q\} \to p'\{q'\}.$$
      If $\varphi'$ is an instance of a left-linear reduction rule, then
      using (Nat 1) and (Nat 2) we get that $S$ is commuting-joinable 
      to $p'\{q'\}$ via the following chain of equalities:
      \begin{equation*}
        \varphi'\{1_q\}\of {p'}\{\psi\} = \varphi'\{\psi'\} = {p}\{\psi'\}\of
      \varphi\{1_{q'}\}.
      \end{equation*}
      A similar argument works when $\varphi'$ is an instance of a
      non left-linear rule, using step (1) to rewrite the residuals of
      $q$ in parallel.
    \item $\varphi$ and $\psi$ rewrite \emph{overlapping subterms} of
      $s$. Without loss of generality, we may assume that $\varphi:s \to
      s_1$ and $\varphi:s \to s_2$. By the definition of a critical span,
      $S$ is then a substitution instance of a critical span, which is
      commuting-joinable by assumption. 
  \end{enumerate}
  By the constuction of $\Fr(\R)$, it follows that $S$ is
  commuting-joinable. 
\end{proof}

\begin{example}
  In this example we utilise Theorem \ref{thm:coherence} to obtain a
  straightforward proof of the coherence theorem for categories with a
  directed associativity map. This coherence theorem is the main
  result of \cite{Laplaza:associative}. 

  Let $\R$ be the rewriting $2$-theory consisting of a single binary
  function symbol $\tensor$, the reduction rule
  \[
  \alpha(x,y,z): x \tensor (y\tensor z) \to (x\tensor y)\tensor z
  \]
  and the left-hand diagram from Example \ref{ex:AU} as an equation on
  reductions. Then, $\R$ is terminating by induction with the ranking
  function
  \[
  \rho(t) = \begin{cases}
    \rho(a) + 2\rho(b) - 1 & \text{if $t = a\tensor b$}\\
    1 &\text{otherwise}.
  \end{cases}
  \]
  The only critical span in this system arises as

  \[
  \vcenter{
    \def\labelstyle{\scriptstyle}
    \begin{xy}
      (0,0)*+{a\tensor(b\tensor(c\tensor d))}="a";
      (-20,-14)*+{(a\tensor b)\tensor (c\tensor d)}="b";
      (20,-14)*+{a\tensor((b\tensor c)\tensor d)}="c";
      {\ar@{->}@/^0.8pc/^-{1_a \tensor \alpha(1_b,1_c,1_d)} "a";"c"};
      {\ar@{->}@/_0.8pc/_-{\alpha(1_a,1_b,1_a\tensor 1_b)} "a"; "b"};
    \end{xy}
  }
  \]
  By the equation we placed on reductions, this critical pair is
  commuting-joinable to $((a\tensor b)\tensor c)\tensor d$. By Lemma
  \ref{lem:cp}, $\R$ is locally commuting-confluent. By Lemma
  \ref{lem:monic}, $\R$ is also monic. So, we may apply Theorem
  \ref{thm:coherence} and conclude that $\R$ is coherent.
\end{example}

Considering the above results, one may be tempted to massage an
incomplete rewriting $2$-theory into a complete one and thus apply the
coherence theorems. Indeed, the famous Knuth-Bendix completion
algorithm \cite{KnuthBendix} achieves precisely this. Unfortunately, such a procedure
typically adds additional reduction rules to the rewriting theory and
this is certainly the case with the Knuth-Bendix algorithm. Rather
than simplifying the coherence problem, this additional structure
results in a new rewriting $2$-theory with a completely independent
coherence problem whose solution sheds very little light on the
coherence problem for the original theory.

In the following section, we tackle the problem of coherence for
invertible rewriting $2$-theories.

\section{Coherence for invertible theories}\label{sec:coherenceinv}

It is not immediately obvious whether Theorem \ref{thm:coherence} can
be extended in any meaningful way to invertible rewriting
$2$-theories. The reason for this is that such systems are necessarily
non-terminating. We can, however, sidestep this problem by restricting
our attention to an orientation of a rewriting $2$-theory.

\begin{definition}[Orientation]
  Let $\R := \rtt$ be a rewriting $2$-theory. An \emph{orientation}
  of $\R$ is a function $\mathcal{O}:\T \to \{1,-1\}$ such that:
  \begin{enumerate}
    \item $\mathcal{O}(\alpha) = 1$ for any non-invertible rule $\alpha$.
    \item For an invertible pair  of rules $(\alpha,\beta)$, either:
      \begin{itemize}
        \item $\mathcal{O}(\alpha) = 1$ and $\mathcal{O}(\beta) = -1$, or
        \item $\mathcal{O}(\alpha) = -1$ and $\mathcal{O}(\beta) = 1$.
      \end{itemize}
  \end{enumerate}
\end{definition}

Given an orientation on a rewriting $2$-theory $\R$, we can restrict our
attention to a directed subtheory of $\R$.

\begin{definition}
  Given a rewriting $2$-theory $\R := \rtt$ with orientation
  $\mathcal{O}$, a reduction rule $\alpha \in \T$ is \emph{positive}
  if $\mathcal{O}(\alpha) = 1$ and \emph{negative} otherwise. The
  positive subtheory of $\R$ relative to $\mathcal{O}$ arises from
  $\R$ by discarding all negative reduction rules from $\T$ and
  discarding all equations from $\E_\T$ that contain an instance
  of a negative reduction rule.
\end{definition}

Working relative to an orientation, we can now extend Theorem
\ref{thm:coherence} to invertible theories.

\begin{theorem}[Coherence]\label{thm:coherenceinv}
A finitely presented invertible rewriting $2$-theory is Mac Lane
coherent if it has an orientation whose positive subtheory is
terminating and locally commuting-confluent. 
\end{theorem}
\begin{proof}
Let $\R$ be an oriented rewriting $2$-theory satisfying the hypotheses
and let $\R^+$ be its positive subtheory. For a reduction $\psi \in
\Fr(\R)$, we write $\psi^{-1}$ for its inverse. Suppose that
$\varphi:A \to B$ is a reduction in $\Fr(\R)$. By Lemma
\ref{lem:irred}, 
\[
\varphi = A \stackrel{\varphi_1}{\to} s_1 \stackrel{\varphi_2}{\to}
s_2 \to \dots \stackrel{\varphi_{n-1}}{\to} s_{n-1}
\stackrel{\varphi_n}{\to} B 
\]
where each $\varphi_i$ is singular. Say that $\varphi_i$ is
\emph{positive} if it contains an instance of a positive reduction
rule and \emph{negative} otherwise. Since $\R^+$ is terminating and
locally commuting-confluent, Lemma \ref{lem:newman} implies that each
term  $t \in \Fr(\R)$ has a unique positive map $N_t:t \to \n(t)$ to a
unique normal form $\n(t)$. We claim that each rectangle in the
following diagram commutes: 

\[
\vcenter{
  \xymatrix{
    {A} \ar[r]^{\varphi_1} \ar[dd]_{N_A} & {s_1} \ar[r]^{\varphi_2}
    \ar[dd]_{N_{s_1}} & {s_2} \ar[r]^{\varphi_3} \ar[dd]_{N_{s_2}} &
    {\dots} \ar[r]^{\varphi_{n-1}} & {s_{n-1}} \ar[r]^{\varphi_n}
    \ar[dd]_{N_{s_{n-1}}} & {B} \ar[dd]_{N_B}\\
    \\
    {\n(A)} \ar@{=}[r] & {\n(s_1)} \ar@{=}[r] & {\n(s_2)} \ar@{=}[r] &
    {\dots} \ar@{=}[r] & {\n(s_{n-1})} \ar@{=}[r] & {\n(B)} 
  }
}
\]
If $\varphi_i$ is positive, then it follows immediately from Lemma
\ref{lem:newman} that $\varphi_i\of N_{s_i} = N_{s_{i-1}}$. If
$\varphi_i$ is negative, then Lemma \ref{lem:newman} implies that
$\varphi_i^{-1}\of N_{s_{i-1}} = N_{s_i}$, which implies that
$\varphi_i\of N_{s_i} = N_{s_{i-1}}$. Since each rectangle commutes,
we have $\varphi\of N_B = N_A$, which implies that $\varphi = N_A\of
N_B^{-1}$. Since $N_A$ and $N_B$ are unique and we did not rely on a
particular choice of $\varphi$, we conclude that $\R$ is coherent.
\end{proof}

\begin{example}
  In this example, we sketch a proof of Mac Lane coherence for
  monoidal categories. From Example \ref{ex:monoidalcrit}, we know a
  set of critical spans for a certain positive subtheory of monoidal
  categories. This subtheory is terminating, as is readily verified by
  the following ranking function:
  \[
  \rho(t) = \begin{cases}
    \rho(a) + 2\rho(b) - 1& \text{if $t = a\tensor b$}\\
    1 &\text{otherwise.}
  \end{cases}
  \]
  In order to conclude coherence, we need only show that every
  critical span is commuting-confluent. From the definition of
  monoidal categories, we know that critical spans
  (\ref{monoidal:assoc}) and (\ref{monoidal:unit2}) are
  commuting-joinable. In his original definition of monoidal
  categories \cite{MacLane_natural}, Mac Lane included additional
  axioms providing commuting joinings for the remaining critical
  spans. However, Kelly later showed \cite{Kelly:monoidal} that these
  critical spans are commuting-joinable as a consequence of the
  pentagon and triangle axioms for monoidal categories. It follows
  from Lemma \ref{lem:cp} that the positive subtheory for monoidal
  categories is locally commuting-confluent. We may then apply Theorem
  \ref{thm:coherenceinv} and conclude that the theory for monoidal
  categories is Mac Lane coherent. 
\end{example}

The basic approach to Mac Lane coherence outlined in this chapter of
proving termination and then analysing the critical spans can be
successfully used to obtain coherence theorems for various other structures
arising in the literature, such as distributive categories
\cite{Laplaza:distributive} and weakly distributive categories
\cite{CockettSeely:wdc}. In light of the results of Chapter
\ref{ch:structure}, this approach may potentially be used to construct 
presentations of structure monoids and groups. The following chapter
details a successful application of this strategy to constructing
presentations of the Higman-Thompson groups. 

\blanknonumber
\chapter{Catalan categories}\label{ch:catalan}

Thompson's groups $F$ and $V$ \cite{Thompson:F} are
important objects arising within combinatorial group
theory. The group $F$ was originally introduced by Thompson in his
investigation of word problems in finitely generated simple
groups. This group was later rediscovered by homotopy theorists as the
automorphism group of a free homotopy-idempotent
\cite{Dydak:a,Dydak:b,FreydHeller}. The group $F$ has several
interesting properties. For instance, it is finitely presentable, 
has a simple commutator subgroup, has only abelian quotients, does not
contain a nonabelian free group, is totally orderable and has
exponential growth \cite{Thompsonintro}.

In unpublished notes, Thompson showed that the group $V$ is a finitely
presented infinite simple group --- the first known group of this
type. McKenzie and Thompson \cite{McKenzieThompson:V} later described
$F$ as a group generated by the variety of semigroups. As we saw in
Chapter \ref{ch:structure}, the relation of $F$ with associativity is
again reflected by the fact that it is the structure group for the
variety of semigroups. We also saw that the group $V$ is the structure
group for commutative semigroups and sketched how Dehornoy's
presentations for these groups \cite{Dehornoy:thompson} arise from the
coherence theorems for coherently associative and commutative
bifunctors. 

As shown by Higman \cite{Higman:thompson}, $V$ is in fact
a member of the infinite family of groups $G_{n,r}$; where $n > 1$ and
$r > 0$ are integers. In particular, $V \cong G_{2,1}$. These groups
share many of the properties of $V$. For instance, they are infinite,
finitely presentable and are either simple or have a simple subgroup of
index $2$. Brown \cite{Brown:finiteness} subsequently showed that Thompson's
group $F$ fits into a similar infinite family $F_{n,r}$, where $F
\cong F_{2,1}$. 

In Section \ref{sec:groups}, we recall Brown's definitions of $F_{n,1}$
and $G_{n,1}$. These groups are defined in a very similar way to $F$
and $V$, which may lead one to wonder whether they too are structure
groups of certain equational varieties.

In Section \ref{sec:structureFG} we introduce $n$-catalan algebras,
which encode a notion of associativity
for an $n$-ary function symbol and prove that $F_{n,1}$ is the structure
group of the variety of $n$-catalan algebras, directly generalising
the relation between $F$ and associativity. We follow this with a
definition of symmetric $n$-catalan algebras, which encode a notion of
associativity and commutativity for an $n$-ary function symbol and we
show that $G_{n,1}$ is the structure group of the variety of symmetric
$n$-catalan algebras. 

We know from Chapter \ref{ch:structure} that a coherent
categorification of a balanced composable equational variety yields a
presentation for the associated structure group. In Section
\ref{sec:catalancat}, we construct a coherent categorification of the
variety of $n$-catalan algebras, thus obtaining a presentation for
$F_{n,1}$. Finally, in Section \ref{sec:scatalancat}, we construct a
coherent categorification of the variety of symmetric $n$-catalan
categories, thus obtaining a presentation for $G_{n,1}$. These
presentations are closely linked to Dehornoy's presentations for $F$
and $V$, which we sketched in Example \ref{ex:DehornoyPresentations}. 

The rewriting $2$-theories that we construct in sections
\ref{sec:catalancat} and \ref{sec:scatalancat} are also interesting from a
purely categorical point of view as they directly generalise the Mac
Lane pentagon and hexagon coherence axioms for commutative and
associative bifunctors. For functors of arity greater than $2$, new
coherence phenomena appear, which are not present in the classical
binary case. 

\section{The groups $F_{n,1}$ and $G_{n,1}$}\label{sec:groups}

In Chapter \ref{ch:structure}, we saw that Thompson's groups $F$ and
$V$ arise as structure groups of certain balanced equational theories
and we subsequently obtained presentations for these groups via
coherent presentations of their associated categorical theories. In
this section, we introduce generalisations of these groups due to
Brown \cite{Brown:finiteness} and Higman \cite{Higman:thompson}, which
we call $F_{n,1}$ and $G_{n,1}$, respectively. In
the following sections, we shall see how the aforementioned process of
constructing presentations for $F$ and $V$ generalises to this broader
class of groups.

There are several paths to defining the groups $F_{n,1}$ and
$G_{n,1}$, all of which relate to the fact that each of these groups
arises as a subgroup of the automorphism group of a Cantor set. Of
the myriad of definitions available, we choose to follow the
description of Brown \cite{Brown:finiteness}, which utilises certain
equivalence classes of pairs of finite rooted trees.

\begin{definition}[Tree]
  The set of $n$-ary trees is defined inductively as follows:
  \begin{itemize}
    \item The graph consisting solely of a single vertex is an $n$-ary
      tree.
    \item If $T_1,\dots,T_n$ are $n$-ary trees then the
      following is also an $n$-ary tree:
      
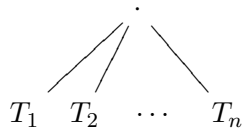
\[
      \begin{xy}
        (0,0)*+{\cdot}="a";
        (-15,-14)*+{T_1}="t1";
        (-7,-14)*+{T_2}="t2";
        (2,-14)*+{\dots}="dots";
        (12,-14)*+{T_n}="tn";
        {\ar@{-} "a"; "t1"};
        {\ar@{-} "a"; "t2"};
        {\ar@{-} "a"; "tn"};
      \end{xy}
      \]
  \end{itemize}
\end{definition}
The \emph{root} of an $n$-ary tree is the unique vertex of valence $0$
or $n-1$. The \emph{leaves} of a rooted tree $T$ are the vertices of
valence $0$ or $1$ and we denote the set of leaves by $\ell(T)$. 

\begin{definition}[Expansion]
  A \emph{simple expansion} of an $n$-ary tree $T$ is the tree obtained by
  replacing a leaf $v$ of $T$ with the following:
  
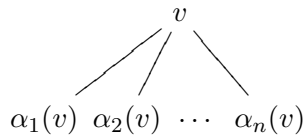
\[
  \begin{xy}
    (0,0)*+{v}="a";
    (-18,-14)*+{\alpha_1(v)}="t1";
    (-7,-14)*+{\alpha_2(v)}="t2";
    (2,-14)*+{\dots}="dots";
    (12,-14)*+{\alpha_n(v)}="tn";
    {\ar@{-} "a"; "t1"};
    {\ar@{-} "a"; "t2"};
    {\ar@{-} "a"; "tn"};
  \end{xy}
  \]
  In the above diagram, each $\alpha_i$ is simply a label for the
  relevant leaf. An \emph{expansion} of an $n$-ary tree is a tree
  obtained by making finitely many succesive simple expansions.    
\end{definition}

Given two trees $T_1$ and $T_2$ having a common expansion $S$, we say
that $S$ is a \emph{minimal} common expansion if any other expansion $S'$ of
$T_1$ and $T_2$ is an expansion of $S$.

\begin{lemma}[Higman \cite{Higman:thompson}]\label{lem:expansion}
  Any two finite $n$-ary trees have a minimal common expansion. \qed
\end{lemma}

The underlying sets of the groups $F_{n,1}$ and $G_{n,1}$ consist of
certain formal expressions called \emph{tree diagrams}.

\begin{definition}[Tree diagram]
  An \emph{$n$-ary tree diagram} is a triple $(T_1,T_2,\sigma)$, where
  $T_1$ and $T_2$ are $n$-ary trees having the same number of leaves
  and $\sigma$ is a bijection $\ell(T_1) \to \ell(T_2)$. 
\end{definition}

As in the case of trees, we may talk about expansions of tree
diagrams.

\begin{definition}
  A \emph{simple expansion} of an $n$-ary tree diagram
  $(T_1,T_2,\sigma)$ is an $n$-ary tree diagram $(T_1',T_2',\sigma')$
  obtained by the following procedure:
  \begin{itemize}
    \item $T_1'$ is the simple expansion of $T_1$ along the leaf $l$. 
    \item $T_2'$ is the simple expansion of $T_2$ along the leaf
      $\sigma(l)$. 
    \item $\sigma'$ is the bijection $\ell(T_1') \to \ell(T_2')$ defined
      by setting $\sigma'(k) = \sigma(k)$ for $k \in
      \ell(T_1)\setminus\{l\}$ and $\sigma'(\alpha_i(l)) =
      \alpha_i(\sigma(l))$.
  \end{itemize}
  An \emph{expansion} of an $n$-ary tree diagram $(T_1,T_2,\sigma)$ is
  any $n$-ary tree diagram obtained by making finitely many succesive
  simple expansions of $(T_1,T_2,\sigma)$.
\end{definition}

Let $\sim$ be the equivalence relation on the set of $n$-ary tree
diagrams obtained by setting $(T_1,T_2,\sigma) \sim (T_1',T_2',\sigma')$
whenever  $(T_1,T_2,\sigma)$ and $(T_1',T_2',\sigma')$ possess a
common expansion. Let $[(T_1,T_2,\sigma)]$ denote the equivalence
class of $(T_1,T_2,\sigma)$ modulo $\sim$. We call
$[(T_1,T_2,\sigma)]$ an \emph{$n$-ary tree symbol}. 

\begin{definition}
For $n\ge 2$, we set $G_{n,1}$ to be the group whose underlying set is
the collection of 
$n$-ary tree symbols, together with the following group structure:
\begin{itemize}
    \item Given two $n$-ary tree symbols $[(T_1,T,\sigma)]$ and
      $[(T',T_2,\sigma')]$, it follows from 

\noindent Lemma~\ref{lem:expansion}
      that we may assume that $T = T'$. We 
      define their product to be
      \[
      [(T_1,T,\sigma)][(T,T_2,\sigma')] = [(T_1,T_2,\sigma\of\sigma')].
      \]
    \item The inverse of $[(T_1,T_2,\sigma)]$ is
      $[(T_2,T_1,\sigma^{-1})]$.
    \item The unit element is $[(T,T,id)]$.
  \end{itemize}
\end{definition}

It follows from the definitions that any $n$-ary tree is an expansion
of the tree consisting solely of a single vertex. Thus, the leaves of
an $n$-ary tree may be seen as a subset of the free monoid on
$\{1,\dots,n\}$. Therefore, we may order the leaves of the tree
lexicographically, which is equivalent to ordering the leaves 
left-to-right when drawn on a page. We say that an $n$-ary tree
symbol $[(T_1,T_2,\sigma)]$ is \emph{order-preserving} if $\sigma$ is
an isomorphism of ordered sets; that is, if $\sigma$ preserves this 
ordering. 

\begin{definition}
  For $n \ge 2$, we set $F_{n,1}$ to be the subgroup of $G_{n,1}$
  consisting of the order-preserving $n$-ary tree symbols.
\end{definition}

The groups $F_{n,1}$ and $G_{n,1}$ generalise Thompson's original
groups $F$ and $V$, since we have $F_{2,1} \cong F$ and $G_{2,1} \cong
V$. They also share several of the interesting properties of $F$ and
$V$ as surveyed in \cite{Scott:tour}. In the following section, we
shall realise $F_{n,1}$ as the structure group of higher-order
associativity and $G_{n,1}$ as the structure group of higher order
associativity and commutativity.

\section{$F_{n,1}$ and $G_{n,1}$ as structure groups}\label{sec:structureFG}

Our goal in this section is to realise $F_{n,1}$ and $G_{n,1}$ as
structure groups. Since both of these groups are built using maps
between $n$-ary trees, we take our set of function symbols to be $\F
:= \{\tensor\}$, where $\tensor$ is an $n$-ary function symbol. For a
set of variables $\V$, there is an obvious bijection between
$\Fr_\F(\V)$ and the set of $n$-ary trees whose leaves are labelled by
members of $\V$. We denote the absolutely free term algebra generated
by $\{\tensor\}$ on the set $\V$ by $\Fr_\tensor(\V)$ and we denote
the free monoid generated by $\V$ under concatenation by $\V^*$.   

Our basic strategy is to first realise $F_{n,1}$ as a
structure group by constructing an equational theory $\E$ such that
$[\E]$ equates any two terms $t_1,t_2 \in \Fr_\F(\V)$ that contain
precisely the same variables in the same order and such that no
variable appears more  
than once in either $t_1$ or $t_2$. In the binary case, this is
achieved by imposing associativity. So, $\E$ ought to be
an analogue of associativity for $n > 2$. Once we have this
realisation of $F_{n,1}$ we need only add the ability to arbitrarily
permute variables in order to obtain a realisation of $G_{n,1}$ as a
structure group.

\subsection{Catalan Algebras and $F_{n,1}$}

Associativity of a binary function symbol is sufficient to establish
that any two bracketings of the same string are equal. The way
in which one establishes this fact is to show that any bracketing of
a string is equal to the left-most bracketing. So, for an $n$-ary
function symbol to be associative, we need equations which imply that
any bracketing of a term is equivalent to the left-most one. In order
to simplify notation, for integers $i \le j$, we use the symbol $x_i^j$
to denote the list $x_i,x_{i+1},\dots,x_j$. If $i > j$, then $x_i^j$
is the empty list.

\begin{definition}[$n$-Catalan algebras]
  For $n \ge 2$, the theory of $n$-Catalan algebras consists of an
  $n$-ary function symbol $\tensor$ together with the following
  equations, where $0 < i < n$:
  \[
  \tensor(x_1^i,\tensor(x_{i+1}^{i+n}),x_{i+n+1}^{2n-1}) =
  \tensor(x_1^{i-1},\tensor(x_{i}^{i+n-1}),x_{i+n}^{2n-1})
  \]
  We denote the theory of $n$-Catalan algebras by $C_n$. 
\end{definition}
The reason for the name of $n$-catalan algebras is that the set of all
terms having $k$ occurrences of the symbol $\tensor$ and containing
precisely one variable is in bijective 
correspondence with the set of $n$-ary trees having $k$ internal
nodes, which has cardinality equal to the generalised Catalan number 
$\frac{1}{(n-1)k+1}{nk \choose k}$, \cite{Stanley:enumerative}.
The rather opaque equational theory of $n$-Catalan algebras is
rendered somewhat more understandable by viewing the induced equations
on the term trees, which for $n=3$, yields the following:

\begin{center}
  \begin{tabular}{clcrc}
    $
    \begin{xy}
      (0,0)*+{\cdot}="a";
      (-10,-12)*+{x_1}="x1";
      (0,-12)*+{x_2}="x2";
      (10,-12)*+{\cdot}="pivot";
      (0,-24)*+{x_3}="x3";
      (10,-24)*+{x_4}="x4";
      (20,-24)*+{x_5}="x5";
      {\ar@{-} "a"; "x1"};
      {\ar@{-} "a"; "x2"};
      {\ar@{-} "a"; "pivot"};
      {\ar@{-} "pivot"; "x3"};
      {\ar@{-} "pivot"; "x4"};
      {\ar@{-} "pivot"; "x5"};
    \end{xy}
    $
    &
    $
    \begin{xy}
      (0,0)*+{};
      (0,-9)*+{=};
    \end{xy}
    $
    &
    $
    \begin{xy}
      (0,0)*+{\cdot}="a";
      (-10,-12)*+{x_1}="x1";
      (0,-12)*+{\cdot}="pivot";
      (10,-12)*+{x_5}="x5";
      (-10,-24)*+{x_2}="x2";
      (0,-24)*+{x_3}="x3";
      (10,-24)*+{x_4}="x4";
      {\ar@{-} "a"; "x1"};
      {\ar@{-} "a"; "pivot"};
      {\ar@{-} "a"; "x5"};
      {\ar@{-} "pivot"; "x2"};
      {\ar@{-} "pivot"; "x3"};
      {\ar@{-} "pivot"; "x4"};
    \end{xy}
    $
    &
    $
    \begin{xy}
      (0,0)*+{};
      (0,-9)*+{=};
    \end{xy}
    $
    &
      $
    \begin{xy}
      (0,0)*+{\cdot}="a";
      (-10,-12)*+{\cdot}="pivot";
      (0,-12)*+{x_4}="x4";
      (10,-12)*+{x_5}="x5";
      (-20,-24)*+{x_1}="x1";
      (-10,-24)*+{x_2}="x2";
      (0,-24)*+{x_3}="x3";
      {\ar@{-} "a"; "x4"};
      {\ar@{-} "a"; "x5"};
      {\ar@{-} "a"; "pivot"};
      {\ar@{-} "pivot"; "x1"};
      {\ar@{-} "pivot"; "x2"};
      {\ar@{-} "pivot"; "x3"};
    \end{xy}
    $
  \end{tabular}
\end{center}

In order to apply the strategy from the binary case to the $n > 2$
case, we need to define what we mean by the left-most bracketing of a
term $t$. Intuitively, this is the term having the same variables as
$t$ in the same order, with all instances of $\tensor$ appearing at
the left.

\begin{definition}[Underlying list]
  Let $t \in \Fr_\tensor(\V)$. The \emph{underlying list} of $t$ is
  the word of $\V^*$ defined inductively by
  \[
  U(t) = \begin{cases}
    U(t_1)\of\dots\of U(t_n) & \text{if $t =
      \tensor(t_1,\dots,t_n)$}\\
    t & \text{otherwise}
  \end{cases}
  \]
\end{definition}

Given the underlying list of a term, we can  define the
left-most bracketing by recursively adding all instances of $\tensor$.

\begin{definition}[Left-most bracketing]
  Let $t \in \Fr_\tensor(\V)$. If $U(t) = t_1\of\dots\of t_{n +
    k(n-1)}$, then the left-most bracketing of $t$ is defined
  recursively by
  \[
  \lmb(t_1^{n+k(n-1)}) = \lmb(\tensor(t_1^n),t_{n+1}^{n+ k(n-1)}).
  \]
\end{definition}

\begin{example}
  In the table below, the right-hand term is the left-most bracketing
  of the left-hand term.\\

  \begin{tabular}{|l|c|c|}
    \hline
    & $t$ & $\lmb(t)$\\
    \hline
    $n = 2$: & $\tensor(a,\tensor(\tensor(b,c),d))$ &
    $\tensor(\tensor(\tensor(a,b),c),d)$\\
    \hline
    $n = 3$: & $\tensor(a,\tensor(b,c,d),\tensor(e,f,g))$ &
    $\tensor(\tensor(\tensor(a,b,c),d,e),f,g)$\\
    \hline
    $n = 4$: & $\tensor(a,b,c,\tensor(d,e,\tensor(f,g,h,i),j))$ &
    $\tensor(\tensor(\tensor(a,b,c,d),e,f,g),h,i,j)$\\
    \hline
  \end{tabular}
\end{example}

We wish to establish that any term  $\Fr_\tensor(\V)$ is equal, in
$\Fr_{C_n}(\V)$, to its left-most bracketing. To this end, we define a
rewriting theory based on $C_n$.

\begin{definition}
  $C_n^\rightarrow$ is the labelled rewriting theory consisting of an
  $n$-ary function symbol $\tensor$, together with the following
  reductions, where $0 < i < n$:
  \[
  \alpha_i: \tensor(x_1^i,\tensor(x_{i+1}^{i+n}),x_{i+n+1}^{2n-1}) \longrightarrow
  \tensor(x_1^{i-1},\tensor(x_{i}^{i+n-1}),x_{i+n}^{2n-1})
  \]
\end{definition}

The reduction rules of $C_n^\rightarrow$ always move a term
``closer'' to its left-most bracketing. This observation is formalised
in the following lemma.

\begin{proposition}\label{prop:cnterminating}
  $C_n^\rightarrow$ is terminating and confluent. Given a term $t \in
  \Fr_\tensor(\V)$, its unique normal form in
  $\Fr_{C_n^\rightarrow}(\V)$ is given by $\lmb(t)$.
\end{proposition}
\begin{proof}
  We construct a ranking function on $\Fr_\tensor(\V)$, which
  establishes that $\Fr_{C_n^\rightarrow}(\V)$ is terminating, that
  for every term $t \in \Fr_{\tensor}(\V)$ there is a reduction $t \to
  \lmb(t)$ and that $\lmb(t)$ is a normal form for $t$. We begin by
  defining the length of $t$.
  \[
  L(t) = \begin{cases}
    \sum_{i=1}^n L(t_i) & \text{if $t = \tensor(t_1^n)$}\\
    1 & \text{otherwise}.
  \end{cases}
  \]
  
  Define the rank, $R(t)$, of $t$ inductively by setting $R(t) = 0$ if
  $t \in \V$ and 
  \[
  R(\tensor(t_1^n)) = \sum_{i=1}^n R(t_i) + \sum_{i=2}^n (i-1)L(t_i) -
  \frac{n(n-1)}{2}.
  \]
 We proceed by double induction on $R(t)$ and $L(t)$.  If $L(t) = 1$ then
 the statement is trivial. If $t = \tensor(t_1^n)$ and $t = \lmb(t)$,
 then $R(t) = R(t_1) + \sum_{i=2}^n (i-1) - \frac{n(n-1)}{2} = R(t_1)$
 and it follows inductively that $R(t) = 0$. Conversely, if $R(t) =
 0$, then $t = \lmb(t)$, since otherwise we would have $L(t_i) > 0$
 for some $2 \le i \le n$, from which it would follow that $R(t) >
 0$. 

  Suppose that $L(t) > 1$ and $R(t) > 0$, so that $t =
 \tensor(t_1^n)$. Let $i$ be the greatest integer with the property that
 $t_i \notin \V$. If $i = 1$, then $t = \lmb(t)$ by induction on
 $L(t)$. If $i > 1$, then $t_i = \tensor(u_1^n)$ and $\alpha_i : t \to
 t'$, where 
\[
 t' = \tensor(t_1^{i-2}, \tensor(t_{i-1}, u_1^{n-1}),u_n,t_{i+1}^n).
 \]
 We then have:
 \begin{eqnarray*}
   R(t)-R(t') &=& R(t_{i-1}) + R(\tensor(u_1^n)) + (i-2)L(t_{i-1}) +
   (i-1)L(\tensor(u_1^n)) 
   \\ && - R(\tensor(t_{i-1},u_1^{n-1}))  - R(u_n) -
   (i-2)L(\tensor(t_{i-1},u_1^{n-1}))\\
   && - (i-1)L(u_n)\\ 
   &=& \sum_{j=2}^n(j-1)L(u_j) + \sum_{j=1}^{n-1}L(u_j) -
   \sum_{j=2}^n(j-1)L(u_{j-1}) \\
  % &=& \sum_{j=1}^{n-1} jL(u_j) + (n-1)L(u_n) - \sum_{j=1}^{n-1}jL(u_j)\\
   &=& (n-1)L(u_n).
 \end{eqnarray*}
Since $L(u_n) \ge 1$, we have $R(t') < R(t)$ and the proposition follows by
induction on  $R(t)$.
\end{proof}

Since each reduction rule in $C_n^\rightarrow$ is a directed version
of an equation in $C_n$, we immediately have the following corollary.

\begin{corollary}\label{corollary:lmb}
  For any $t \in \Fr_\tensor(\V)$, we have $t =_{C_n} \lmb(t)$. \qed
\end{corollary}

In order to manipulate elements of $\Struct(C_n)$ effectively, we
introduce the notion of a seed.

\begin{definition}[Seed]
  Let $\F$ be a graded set of function symbols on some set $\V$ and let
  $\rho$ be a partial function $\Fr_\F(\V) \to \Fr_\F(\V)$. A
  \emph{seed} for $\rho$ is a pair of terms $s,t \in \Fr_\F(\V)$ such
  that the graph of $\rho$ is equal to
  $\{(s^\varphi,t^\varphi)~|~\varphi \in [\V,\Fr_\F(\V)]\}$.
\end{definition}

In particularly nice cases, we can construct seeds for any operator in
a structure monoid. 

\begin{lemma}[Dehornoy \cite{Dehornoy:braids}]\label{lem:seeds}
  Let $\T$ be a balanced equational theory that contains precisely one
  function symbol. Then, each operator $\rho \in \Struct(\T)$ admits a
  seed. 
\end{lemma}

It follows from Lemma \ref{lem:expansion} that $C_n$ is composable
and we may, therefore, form the group $\Struct_G(C_n)$. In order to
facilitate the passage from members of $\Struct_G(C_n)$, to members of
$F_{n,1}$, we introduce the tree generated by a term. 

\begin{definition}
  For a term $t \in \Fr_\tensor(\V)$, let $T(t)$ denote the $n$-ary
  tree obtained via the following construction:
  
  \begin{itemize}
    \item If $t = \tensor(t_1,\dots,t_n)$, then $T(t)$ is equal to:
      
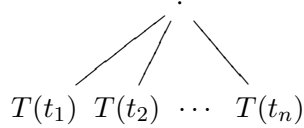
\[
      \begin{xy}
        (0,0)*+{\cdot}="a";
        (-18,-14)*+{T(t_1)}="t1";
        (-7,-14)*+{T(t_2)}="t2";
        (2,-14)*+{\dots}="dots";
        (12,-14)*+{T(t_n)}="tn";
        {\ar@{-} "a"; "t1"};
        {\ar@{-} "a"; "t2"};
        {\ar@{-} "a"; "tn"};
      \end{xy}
      \]
    \item Otherwise, $T(t)$ is the single vertex $\cdot$
\end{itemize}
\end{definition}

We now have all the tools required to show that $F_{n,1}$ is the
structure group of $n$-catalan algebras.

\begin{theorem}\label{thm:Fn1}
  $\Struct_G(C_n) \cong F_{n,1}$.
\end{theorem}
\begin{proof}
We denote the seed of $\rho \in \Struct_G(C_n)$, which exists by Lemma
\ref{lem:seeds}, by $(s_\rho,t_\rho)$. We claim that the following map
is an isomorphism:
\begin{eqnarray*}
    \Struct_G(C_n) &\stackrel{\Theta}{\longrightarrow}& F_{n,1}\\
    \rho &\longmapsto& [(T(s_\rho),T(t_\rho),id)]
\end{eqnarray*}

It is routine to see that $\Theta$ is a homomorphism. Suppose that
$\rho,\rho' \in \Struct_G(C_n)$ and that $\Theta(\rho) =
\Theta(\rho')$. It follows that $\rho$ and $\rho'$ have the same seed,
so $\rho = \rho'$ and $\Theta$ is faithful.

By Lemma \ref{lem:capture}, in order to establish
that $\Theta$ is surjective, we need only show that $t_1 =_{C_n} t_2$
whenever $t_1,t_2 \in \Fr_\tensor(\V)$ and $U(t_1) = U(t_2)$. By
Corollary \ref{corollary:lmb}, we have $t_1 =_{C_n} \lmb(t_1) =_{C_n} \lmb(t_2)
=_{C_n} t_2$, so $\Theta$ is surjective and, hence, an isomorphism.
\end{proof}

\subsection{Symmetric Catalan Algebras and $G_{n,1}$} We saw in
Section \ref{sec:groups} that the leaves of a tree may be ordered by
the lexicographic ordering on their addresses. An $n$-ary tree symbol
$[(T_1,T_2,\sigma)]$ may thereby be viewed as a pair of tree diagrams,
together with a permutation of the leaves of $T_1$. Thus, in order to
obtain an equational theory whose structure group is $G_{n,1}$ we need
to add the ability to arbitrarily permute variables in Catalan
algebras. Recalling that the symmetric group is generated by
transpositions of adjacent elements, we are led to the following
definition.

\begin{definition}[Symmetric $n$-Catalan Algberas]
  The theory of symmetric $n$-catalan algebras extends that of
  $n$-catalan algebras with the following equations, where $1 \le i <
  n$:
  \[
  \tensor(x_1^{i-1},x_i,x_{i+1},x_{i+2}^n) =
  \tensor(x_1^{i-1},x_{i+1},x_i,x_{i+2}^n).         
  \]
  We denote the theory of symmetric $n$-catalan algebras by $SC_n$.
\end{definition}

Symmetric $n$-catalan algebras essentially add an action of the
symmetric group on the indices of $\tensor$. In general, this is
sufficient to induce an action of a symmetric group on the variables
of any term in $\Fr_\tensor(\V)$. In the binary case, we
recover the definition of commutative semigroups.

\begin{theorem}\label{thm:Gn1}
  $\Struct_G(SC_n) \cong G_{n,1}$. 
\end{theorem}
\begin{proof}
  For $\rho \in \Struct_G(SC_n)$, let $(s_\rho,t_\rho)$ represent its
  seed, which exists by Lemma \ref{lem:seeds}. Since $SC_n$ is linear,
  $s_\rho$ and $t_\rho$ are linear and $\supp(s_\rho) =
  \supp(t_\rho)$. Let $\pi(\rho)$ be the permutation of
  $\supp(s_\rho)$ induced by the permutation $U(s_\rho) \to
  U(t_\rho)$.    A similar argument to the proof of Theorem \ref{thm:Fn1}
  establishes that the following map is an isomorphism:
  \begin{eqnarray*}
    \Struct_G(C_n) &\stackrel{\theta}{\longrightarrow}& G_{n,1}\\
    \rho &\longmapsto& [(T(s_\rho),T(t_\rho),\pi(\rho))]
  \end{eqnarray*}
\end{proof}

We now know that $F_{n,1}$ and $G_{n,1}$ are the structure groups of
catalan algebras and of symmetric catalan algebras, respectively. We also
know that if we can construct coherent categorifications of these
algebras, then we can apply Theorem \ref{thm:group} to obtain
presentations of these groups. In the following section, we set about
the task of constructing a coherent categorification of catalan
algebras.

\section{Catalan categories and $F_{n,1}$}\label{sec:catalancat}

In order to obtain a presentation for $\Struct_G(C_n)$ and, hence, for
$F_{n,1}$ along the lines of that provided by Dehornoy for $F$
\cite{Dehornoy:thompson}, we need to obtain a coherent
categorification of $C_n$. The immediate problem is discerning a set
of diagrams whose commutativity imply the commutativity of all
diagrams generated by the categorification. As we shall see in this
section, the following definition suffices for this purpose. While the
coherence axioms that we have chosen may seem slightly cryptic, the
reason for their choice will become apparent in the proof that the
resulting categorification is coherent. We shall make frequent use of
the following useful shorthand: For $1 \le i \le n$ and a morphism
$\rho:t_i \to t_i'$, we set   
  \[
  \tensor^i(\rho) = \tensor(1_{t_1},\dots,
  1_{t_{i-1}},\rho,1_{t_{i+1}},\dots,1_{t_n}).
  \]

\begin{definition}
  The rewriting $2$-theory for $n$-catalan categories is denoted
  $\Cn$ and consists of:
  \begin{itemize}
  \item An $n$-ary function symbol $\tensor$.
  \item For $1 \le i < n$, an invertible reduction rule $\alpha_i$
    of the following form:
    \[
    \alpha_i(x_1^{2n-1}) :
    \tensor(x_1^i,\tensor(x_{i+1}^{i+n}),x_{i+n+1}^{2n-1}) 
    \stackrel{\sim}{\longrightarrow}
    \tensor(x_1^{i-1},\tensor(x_i^{i+ n-1}), x_{i+n}^{2n-1})
    \]
  \end{itemize}

  \noindent \textbf{Pentagon axiom:} For $1 \le i \le n-1$, the following 
    diagram commutes, where $X = x_1^{i-1}$ and  $Z = z_1^{n-i-1}$:  
    \[
    \begin{xy}
      (0,4)*+{{\tensor(X,y_1,\tensor(y_2^n,\tensor(y_{n+1}^{2n})),Z)}}="1";
      (-33,-15)*+{\tensor(X,\tensor(y_1^n),\tensor(y_{n+1}^{2n}),Z)}="2";
      (30,-15)*+{\tensor(X,y_1,\tensor(y_i^{n-1},\tensor(y_n^{2n-1}),y_{2n}),Z)}
      ="3"; 
      (-33,-35)*+{\tensor(X,\tensor(\tensor(y_1^n),y_{n+1}^{2n-1}),y_{2n},Z)}
      ="4";
      (33,-35)*+{\tensor(X,\tensor(y_1^{n-1},\tensor(y_n^{2n-1})),y_{2n},Z)}
      ="5";
      {\ar@{->}_<<<<<<<{\alpha_i}@/_/ "1"; "2"}
      {\ar@{->}^<<<<<<<{\tensor^{i+1}(\alpha_{n-1})}@/^/ "1"; "3"}
      {\ar@{->}^{\alpha_i} (26,-19); (26,-32)}
      {\ar@{->}_{\alpha_i} "2"; "4"}
      {\ar@{->}^{\tensor^i(\alpha_{n-1}\of\dots\of\alpha_1)} @/^1.8pc/
        (22,-38); (-33,-39)} 
    \end{xy}
    \]
    \medskip\\
    \bigbreak
    \noindent\textbf{Adjacent associativity axiom:} For $1 \le i \le
    n-2$ , the following 
    diagram commutes, where $X = x_1^{i-1}$ and  $Z = z_1^{n-i-2}$: 
    
    \[
    \begin{xy}
      (10,4)*+{\tensor(X,y_1,\tensor(y_2^{n+1}),\tensor(y_{n+2}^{2n+1}),Z)}="a";
      (-28,-12)*+{\tensor(X,\tensor(y_1^n),y_{n+1},\tensor(y_{n+2}^{2n+1}),
        Z)}="b"; 
      (28,-24)*+{\tensor(X,y_1,\tensor(\tensor(y_2^{n+1}),y_{n+2}^{2n}),
        y_{2n+1},Z)}="c";
      (-32,-35)*+{\tensor(X,\tensor(y_1^n),\tensor(y_{n+1}^{2n-1}),y_{2n}^{2n+1},
      Z)}="d";
      (28,-46)*+{\tensor(X,\tensor(y_1,\tensor(y_2^{n+1}),y_{n+2}^{2n-1}),
        y_{2n}^{2n-1},Z)}="e";
      (-22,-59)*+{\tensor(X,\tensor(\tensor(y_1^n),y_{n+1}^{2n-1}),y_{2n}^{2n+1},
      Z)}="f";
    {\ar@{->}_<<<<<<<{\alpha_i}@/_1pc/ "a"; (-29,-11)}
    {\ar@{->}^{\alpha_{i+1}}@/^/ "a"; (28,-20)}
    {\ar@{->}_{\alpha_{i+1}} (-32,-18); "d"}
    {\ar@{->}^{\alpha_i} (27,-28); (27,-41)}
    {\ar@{->}_{\alpha_i}  (-35,-38);(-35,-55)}
    {\ar@{->}^{\tensor^i(\alpha_1)}@/^/ (23,-49); (1,-59)}
  \end{xy}
  \]
\end{definition} 

In the case where $n=2$, the pentagon axiom reduces to Mac Lane's
pentagon axiom for monoidal categories from Example \ref{ex:AU}
and the adjacent associativity axiom is empty, so we recover the usual
definition of a coherently associative bifunctor. 

In the special case where $n=3$, the adjacent associativity axiom
leads to a single coherence axiom, illustrated by the following
diagram. The other axioms given in this chapter may be unpacked in
this special case in a similar manner.

 
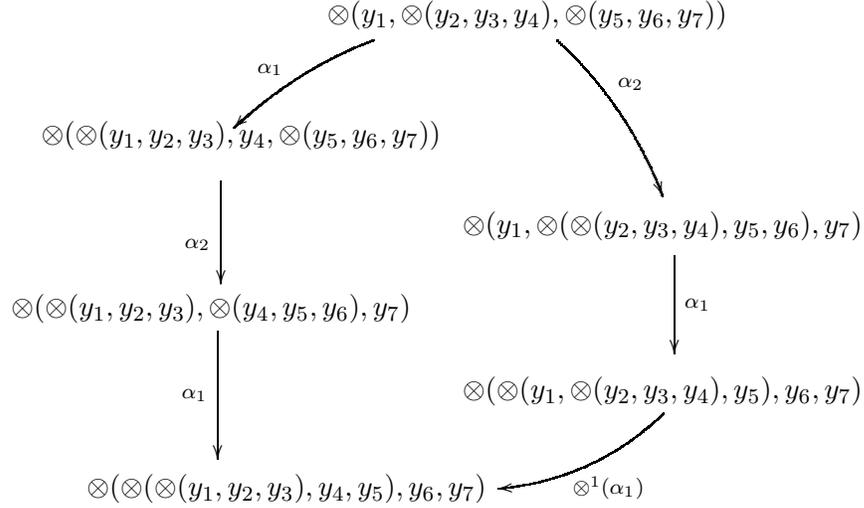
\[
    \begin{xy}
      (10,4)*+{\tensor(y_1,\tensor(y_2,y_3,y_4), \tensor(y_5,y_6,y_7))}="a";
      (-28,-12)*+{\tensor(\tensor(y_1,y_2,y_3),y_4, \tensor(y_5,y_6,y_7))}="b"; 
      (28,-24)*+{\tensor(y_1,\tensor(\tensor(y_2,y_3,y_4),y_5,y_6),y_7)}="c";
      (-32,-35)*+{\tensor(\tensor(y_1,y_2,y_3),\tensor(y_4,y_5,y_6),y_7)}="d";
      (28,-46)*+{\tensor(\tensor(y_1,\tensor(y_2,y_3,y_4),y_5),y_6,y_7)}="e";
      (-22,-59)*+{\tensor(\tensor(\tensor(y_1,y_2,y_3),y_4,y_5),y_6,y_7)}="f";
    {\ar@{->}_<<<<<<<{\alpha_1}@/_1pc/ "a"; (-29,-11)}
    {\ar@{->}^{\alpha_2}@/^/ "a"; (28,-20)}
    {\ar@{->}_{\alpha_2} (-32,-18); "d"}
    {\ar@{->}^{\alpha_1} (27,-28); (27,-41)}
    {\ar@{->}_{\alpha_1}  (-35,-38);(-35,-55)}
    {\ar@{->}^{\tensor^1(\alpha_1)}@/^/ (23,-49); (1,-59)}
  \end{xy}
  \]
\\

We wish to apply Theorem \ref{thm:coherenceinv} to $\Cn$ in order to show
that it is a coherent categorification of $C_n$. In order to do this,
we need to find a positive orientation of $\Cn$ that is terminating
and locally commuting-confluent. Let $\mathbb{C}_n^\rightarrow$ be the positive
subtheory of $\C_n$ that contains $\alpha_i$ for $0 < i < n$. This is
equivalent, as a rewriting theory, to the system $C_n^\rightarrow$
introduced in the last section. Therefore, we know from Proposition
\ref{prop:cnterminating} that $\mathbb{C}_n^\rightarrow$ is terminating. So, it
remains to show that $\mathbb{C}_n^\rightarrow$ is locally
commuting-confluent. 

\begin{lemma}
  $\mathbb{C}_n^\rightarrow$ is locally commuting-confluent.
\end{lemma}
\begin{proof}
  By Lemma \ref{lem:cp}, we only need to show that every critical span
  in $\mathbb{C}_n^\rightarrow$ is commuting-joinable. Suppose that
  $\alpha_i$ and $\alpha_j$ overlap, where
  \[
  \alpha_i :
  \tensor(t_1^i,\tensor(t_{i+1}^{i+n}),t_{i+n+1}^{2n-1}) 
  \to
  \tensor(t_1^{i-1},\tensor(t_i^{i+ n-1}), t_{i+n}^{2n-1}).
  \]
 Suppose first that
 $\tensor(t_1^i,\tensor(t_{i+1}^{i+n}),t_{i+n+1}^{2n-1})$ is an
 $\alpha_j$-reduct. For the overlap to be nontrivial, we must have $j
 \ne i$. If $j \ne  i \pm 1$, then the critical span arising from the
 overlap is  commuting-joinable by naturality. If $j = i \pm 1$, then
 the critical span arising from the overlap is commuting-joinable by
 the adjacent associativity axiom.  
 If the overlap arises because $\tensor(t_{i+1}^{i+n})$ is an
 $\alpha_j$-reduct, then there are two possibilities. If $j \ne n$,
 then the critical span arising from the overlap is commuting-joinable
 by naturality. Otherwise, the critical span arising from the overlap
 is commuting-joinable by the pentagon axiom. 
 The only other possible overlap arises when some $t_k$ is an
 $\alpha_j$-reduct. In this case, the critical span arising from the
 overlap is commuting-joinable by naturality. 
\end{proof}

Since $\mathbb{C}_n^\rightarrow$ is terminating and locally
commuting-confluent, we may apply Theorem \ref{thm:coherenceinv}.

\begin{theorem}\label{thm:coherencecatalan}
  $\Cn$ is a coherent categorification of $C_n$. \qed
\end{theorem}

With Theorem \ref{thm:coherencecatalan} in hand, we can obtain a
presentation for $F_{n,1}$, which generalises the presentation for $F$
given in \cite{Dehornoy:thompson}. 

\begin{corollary}
  $\S_G(\Cn) \cong F_{n,1}$
\end{corollary}
\begin{proof}
  By Theorem \ref{thm:coherencecatalan} and Theorem \ref{thm:group},
  we have $\S_G(\Cn) \cong \Struct_G(C_n)$. It follows then from Theorem
  \ref{thm:Fn1} that $\S_G(\Cn) \cong F_{n,1}$.
\end{proof}

In the following section, we shall obtain a coherent categorification
of $SC_n$ and, thereby, a presentation of $G_{n,1}$.

\section{Symmetric Catalan categories and $G_{n,1}$ }\label{sec:scatalancat}

Our goal in this section is to construct a coherent categorification
of symmetric catalan algebras. The coherence theorem for catalan
categories, Theorem \ref{thm:coherencecatalan}, reduces this problem
to ensuring that any two sequences of transpositions of the objects
appearing in a term realise the same permutation. In other words, our
categorification needs to somehow encode a presentation of the
symmetric group whose generators correspond to transpositions of
adjacent variables. Such a presentation is well known, having been
constructed by Moore \cite{Moore:presentation}. This presentation has
generators $T_1,\dots, T_{n-1}$ and the following relations:
\begin{eqnarray*}
  T_i^2 = 1 & \text{for $1 \le i \le n-1$}\\
  (T_iT_{i+1})^3 = 1 & \text{for $1 \le i \le n-2$}\\
  (T_iT_k)^2 = 1 &\text{for $1 \le i \le k-2$}
\end{eqnarray*}
With this presentation in mind, we may now construct a reasonable
categorification of $SC_n$. Recall our shorthand that for $1 \le i \le
n$ and a morphism $\rho:t_i \to t_i'$, we have
  \[
  \tensor^i(\rho) = \tensor(1_{t_1},\dots,
  1_{t_{i-1}},\rho,1_{t_{i+1}},\dots,1_{t_n}).
  \]

\begin{definition}
  For $n \ge 2$, the rewriting $2$-theory for symmetric $n$-catalan
  categories, denoted $\SCn$, is
  the extension of the theory for $n$-catalan categories
  with an invertible reduction rule $\tau_i$ for $1 \le i \le n-1$
  such that  \[
  \tau_i(t_1^n) : \tensor(t_1^{i-1},t_i,t_{i+1},t_{i+2}^n)
  \stackrel{\sim}{\longrightarrow}
  \tensor(t_1^{i-1},t_{i+1},t_{i},t_{i+2}^n),
  \]
  satisfying the following axioms:\\
  
  \noindent\textbf{Involution axiom:} For $1 \le i \le n-1$, the
  following diagram commutes:
  \[
  \begin{xy}
    (0,0)*+{\tensor(t_1^n)}="1";
    (0,-20)*+{\tensor(t_1^n)}="2";
    (45,-20)*+{\tensor(t_1^{i-1},t_{i+1},t_i,t_{i+2}^n)}="3";
    {\ar@{->}_{1} "1"; "2"}
    {\ar@{->}^{\tau_i} "1"; "3"}
    {\ar@{->}^>>>>>>>>{\tau_i} "3"; "2"}
  \end{xy}
  \]
   \medskip
  \noindent\textbf{Compatibility axiom:} For $2 \le i \le n$ and $1
  \le j \le n-2$, the following diagram commutes, where $W = w_1^i$ and $Z = z_1^{n-i}$: 
  \[
  \begin{xy}
    (0,0)*+{\tensor(W,x,\tensor(y_1^n),Z)}="1";
    (-30,-20)*+{\tensor(W,\tensor(x,y_1^{n-1}),y_n,Z)}="2";
    (33,-20)*+{\tensor(W,x,\tensor(y_1^{j-1},y_{j+1},y_j,y_{j+2}^n),Z)}
    ="3";
    (0,-42)*+{\tensor(W,\tensor(x,y_1^{j-1},y_{j+1},y_j,y_{j+2}^{n-1}),
      y_n,Z)}="4";
    {\ar@{->}_{\alpha_{i-1}}@/_/ "1"; "2"}
    {\ar@{->}^{\tensor^i(\tau_j)}@/^/ "1"; "3"}
    {\ar@{->}_<<<<<<{\tensor^{i-1}(\tau_{j+1})}@/_/ "2"; "4"}
    {\ar@{->}^<<<<<<{\alpha_{i-1}}@/^/ (27,-23); "4"}
  \end{xy}
  \]

  \noindent\textbf{$3$-cycle axiom:} For $1 \le i \le n-2$, the
  following diagram commutes:
  \[
  \begin{xy}
    (0,0)*+{\tensor(t_1^n)}="1";
    (-30,-20)*+{\tensor(t_1^{i-1},t_{i+1},t_i,t_{i+2}^n)}="2";
    (30,-20)*+{\tensor(t_1^i,t_{i+2},t_{i+1},t_{i+3}^n)}="3";
    (-30,-40)*+{\tensor(t_1^{i-1},t_{i+1},t_{i+2},t_{i},t_{i+3}^n)}="4";
    (30,-40)*+{\tensor(t_1^{i-1},t_{i+2},t_{i},t_{i+1},t_{i+3}^n)}="5";
    (0,-60)*+{\tensor(t_1^{i-1},t_{i+2},t_{i+1},t_{i},t_{i+3}^n)}="6";
    {\ar@{->}_{\tau_i}@/_/ "1"; "2"}
    {\ar@{->}^{\tau_{i+1}}@/^/ "1"; "3"}
    {\ar@{->}_{\tau_{i+1}} "2"; "4"}
    {\ar@{->}^{\tau_i} (26,-24); (26,-37)}
    {\ar@{->}_<<<<<<<{\tau_{i}}@/_/ (-32,-44); (-15,-56)}
    {\ar@{->}^>>>>>>>{\tau_{i+1}}@/^/ (22,-45); "6"}
  \end{xy}
  \]

  \noindent\textbf{Hexagon axiom:} For $1 \le i \le n-1$, the
  following diagram commutes, where $W = w_1^{i-1}$ and $Z =
  z_1^{n-i-1}$:
  \[
  \begin{xy}
    (0,0)*+{\tensor(W,\tensor(x_1^n),y,Z)}="1";
    (-30,-20)*+{\tensor(W,y,\tensor(x_1^n),Z)}="2";
    (30,-20)*+{\tensor(W,x_1,\tensor(x_2^n,y),Z)}="3";
    (-30,-40)*+{\tensor(W,\tensor(y,x_1^{n-1}),x_n,Z)}="4";
    (30,-40)*+{\tensor(W,x_1,\tensor(y,x_2^n),Z)}="5";
    (0,-60)*+{\tensor(W,\tensor(x_1,y,x_2^{n-1}),x_n,Z)}="6";
    {\ar@{->}_{\tau_i}@/_/ "1"; "2"}
    {\ar@{->}^{\alpha_i^{-1}}@/^/ "1"; "3"}
    {\ar@{->}_{\alpha_i} "2"; "4"}
    {\ar@{->}^{\tensor^{i+1}(\tau_{n-1}\of\dots\of\tau_1)} (26,-24);
      (26,-37)} 
    {\ar@{->}_<<<<<<<{\tensor^i(\tau_1)}@/_/ (-32,-44); (-15,-56)}
    {\ar@{->}^>>>>>>>{\alpha_i}@/^/ (22,-45); "6"}
  \end{xy}
  \]
\end{definition}

The hexagon axiom ensures that we may replace a transposition of the
form $\tau_i(t_1^{i-1},\tensor(u_1^n),t_{i}^{n-1})$ with a sequence of
transpositions involving only the terms $t_1^{n-1}$ and
$u_1^n$. One might posit the commutativity of a diagram that serves the
same purpose for a morphism of the form
$\tau_i(t_1^{i},\tensor(u_1^n),t_{i+1}^{n-1})$. Doing so leads to the
\emph{dual hexagon diagram}, which has the following form, for $2
\le i \le n$ and $W = w_1^{i-2}$ and $Z = z_1^{n-i}$: 
\[
\begin{xy}
  (0,0)*+{\tensor(W,x,\tensor(y_1^n),Z)}="1";
  (-30,-20)*+{\tensor(W,\tensor(y_1^n),x,Z)}="2";
  (30,-20)*+{\tensor(W,\tensor(x,y_1^{n-1}),y_n,Z)}="3";
  (-30,-40)*+{\tensor(W,y_1,\tensor(y_2^n,x),Z)}="4";
  (30,-40)*+{\tensor(W,\tensor(y_1^{n-1},x),y_n,Z)}="5";
  (0,-60)*+{\tensor(W,y_1,\tensor(y_2^{n-1},x,y_n),Z)}="6";
    {\ar@{->}_{\tau_i}@/_/ "1"; "2"}
    {\ar@{->}^{\alpha_i}@/^/ "1"; "3"}
    {\ar@{->}_{\alpha_i^{-1}} "2"; "4"}
    {\ar@{->}^{\tensor^{i}(\tau_{1}\of\dots\of\tau_{n-1})} (26,-24);
      (26,-37)} 
    {\ar@{->}_<<<<<<<{\tensor^{i+1}(\tau_{n-1})}@/_/ (-32,-44); (-15,-56)}
    {\ar@{->}^>>>>>>>{\alpha_i^{-1}}@/^/ (22,-45); "6"}
\end{xy}
\]

\begin{lemma}\label{lem:dualhex}
  The dual hexagon diagram commutes in $\Fr(\SCn)$.
\end{lemma}
\begin{proof}
  Tracing around the dual hexagon diagram, we obtain the following
  morphism:
  \begin{equation}\label{eq:hex1}
    \alpha_i^{-1}\of\tensor^{i+1}(\tau_{n-1})\of\alpha_i\of
    \tensor^i(\tau_1\of\dots\of\tau_{n-1})^{-1}\of\alpha_i^{-1}.
  \end{equation}
  In order to show that the dual hexagon diagram commutes, we need to
  show that $(\ref{eq:hex1}) = \tau_i$. By functoriality, we have:
  \[\tensor^i(\tau_1\of\dots\tau_{n-1})^{-1} =
  \tensor^i(\tau_1)^{-1}\of\dots\of\tensor^i(\tau_{n-1})^{-1}.
  \]
  By functoriality and the involution axiom, we have
  $\tensor^i(\tau_j)^{-1} = \tensor^i(\tau_j)$. From the compatibility
  axiom, we also know that $\tensor^i(\tau_j) =
  \alpha_i^{-1}\of\tensor^{i+1}(\tau_{j-1})\of\alpha_i$. It follows
  from these observations that:
  \begin{eqnarray}
    (\ref{eq:hex1}) &=& \alpha_i^{-1}\of \tensor^{i+1}(\tau_{n-1})\of
    \dots\of\tensor^{i+1}(\tau_1)\of\alpha\of\tensor^i(\tau_1)
    \of\alpha^{-1}\\
    &=& \alpha_i^{-1}\of\tensor^{i+1}(\tau_{n-1}\of\dots\of\tau_1)
    \of\alpha_i \of\tensor^i(\tau_1)\of\alpha_i^{-1}\\
    &=& \alpha_i^{-1}\of\tensor^{i+1}(\tau_{n-1}\of\dots\of\tau_1)
    \of\alpha_i \of\tensor^i(\tau_1)^{-1}\of\alpha_i^{-1}.\label{eq:hex2}
  \end{eqnarray}
  It follows from the hexagon axiom that $(\ref{eq:hex2}) =
  \tau_i$. By the involution axiom, we then have:
  \[
  \tau_i\of\alpha_i^{-1}\of\tensor^{i+1}(\tau_{n-1})\of\alpha_i
  \of\tensor^i(\tau_1\of\dots\of\tau_{n-1})^{-1}\of\alpha_{i}^{-1} = 1.
  \]
  Therefore, the dual hexagon diagram commutes in $\Fr(\SCn)$.
\end{proof}

In the $n=2$ case, the axiomatisation of $\SCn$ reduces to the theory
of a coherently associative and commutative bifunctor given in Example
\ref{ex:DehornoyPresentations}. The main
result of this section establishes that $\SCn$ is a suitable
generalisation of this case. 

\begin{theorem}\label{thm:coherencescatalan}
  $\SCn$ is a coherent categorification of $SC_n$. 
\end{theorem}
\begin{proof}
  By Theorem \ref{thm:coherencecatalan} and Corollary
  \ref{cor:categorification}, we may assume that all of the
  associativity maps are strict equalities. Thus, an object of
  $\Fr(\SCn)$ may be represented as $\tensor(t_1^m)$, where each
  $t_i$ is a variable and $m = n + k(n-1)$, for some $k \ge
  0$. Lemma   \ref{lem:dualhex} and the hexagon axiom imply that it
  suffices to   consider transpositions of adjacent variables. So, for
  a given object $t := \tensor(t_1^m)$, we  
  need only   consider the $m-1$ induced transposition natural
  isomorphisms  
  \[
    T_i(t_1^m) : \tensor(t_1^{i-1},t_i,t_{i+1},t_{i+2}^m) \to
  \tensor(t_1^{i-1},t_{i+1},t_{i},t_{i+2}^m).
  \]
  In order to establish coherence, we have to show that every
  permutation of $t_1^m$ is unique. That is, we have to show that the
  induced transposition maps satisfy the defining relations for the
  symmetric group of order $m$. 

  The compatibility axiom implies that each $T_i$ is unique.  By the
  naturality of the maps $T_i$, we have $T_i\of T_k = 
  T_k\of T_i$ for all $1 \le i \le k-2$. The involution axiom implies
  that $T_i^2 = 1$. Thus, it only remains to establish that $(T_i\of
  T_{i+1})^3 = 1$. For $n=2$, we may use the proof from Mac Lane
  \cite{MacLane_natural}. Suppose that $n \ge 3$. Since the
  associativity maps are taken to be strict equalities, we may assume
  that $t$ has  the form
  $\tensor(R,\tensor(S,t_i,t_{i+1},t_{i+2},U),V)$, where $R,S,U$ and
  $V$ are sequences of variables. The result then follows from
  the $3$-cycle axiom. 
\end{proof}

We can now construct a presentation of $\Struct_G(SC_n)$ and,
therefore, of $G_{n,1}$, which generalises the presentation for $V$
given in \cite{Dehornoy:thompson}.  

\begin{corollary}
  $\S_G(\SCn) \cong G_{n,1}$
\end{corollary}
\begin{proof}
  By Theorem \ref{thm:coherencescatalan} and Theorem \ref{thm:group},
  we have $\S_G(\SCn) \cong \Struct_G(SC_n)$. It follows then
  from Theorem 
  \ref{thm:Gn1} that $\S_G(\Cn) \cong G_{n,1}$.
\end{proof}

In this chapter, we have seen how the seemingly abstract general
coherence theorems developed in Chapter \ref{ch:unf} can have very
powerful applications. Indeed, the proof that $\Cn$ is a coherent
categorification of $C_n$ was relatively routine. Unfortunately, not
all rewriting $2$-theories are of the form required for the theorems
from Chapter \ref{ch:unf} to be applicable. In the following chapter,
we develop more general coherence theorems that relax those
assumptions somewhat.\blanknonumber
\chapter{Coherence for incomplete theories}\label{ch:nunf}
%\begin{comment}

In Chapter \ref{ch:unf}, we developed a general Mac Lane coherence
theorem for terminating and confluent rewriting $2$-theories. This
result has wide applicability, including the investigation of
catalan categories presented in Chapter \ref{ch:catalan}. 

Unfortunately, it is simply not the case that every coherent rewriting
$2$-theory has unique normal forms. For instance, the theory
consisting of a unary function symbol $F$ and the single reduction
rule $F(x) \to F(F(x))$ is non-terminating, but easily seen 
to be coherent. A stronger counterexample to the hope that
coherent structures have unique normal forms is provided by the theory
of iterated monoidal categories \cite{iterated} whose coherence
problem we investigate in the following chapter. These structures
arise as a  categorical model of iterated loop spaces and fail to be
confluent, so the tools of Chapter \ref{ch:unf} do not apply.

We are thus faced with the problem of determining sufficient conditions
for coherence in terms of the underlying rewriting system of a
$2$-theory that do not rely on either termination or
confluence. This leads to the related problem of determining whether,
for any finitely presented labelled rewriting theory,
there is always a finite set of diagrams whose commutativity implies
the commutativity of all diagrams built from the theory. 

This chapter sets out to solve several related coherence questions by
vigourously pursuing the idea that two morphisms with the same
source and target in a free covariant structure on a discrete category
commute precisely when they admit a planar subdivision such that each
face is an instance of naturality, or of functoriality or of one of
the coherence axioms. The guiding intuition behind this approach is
that a span that cannot be completed into a square can never appear in
such a  subdivision. 

Section \ref{sec:subdivisions} lays the foundations for this chapter
by providing precise definitions of the various concepts related to
subdivisions of parallel pairs of arrows and determining conditions
that ensure that each parallel pair of arrows has only finitely many
subdivisions. This quickly leads, in Section
\ref{sec:lambekcoherence}, to a general Lambek coherence
theorem. Section \ref{sec:maclanecoherence} provides a more refined
analysis of the possible subdivisions of a parallel pair of reductions
in a finitely presented labelled rewriting theory and exploits this
analysis to obtain a general Mac Lane coherence theorem. Finally,
Section \ref{sec:finiteness} constructs examples of labelled rewriting
theories that cannot be made coherent via only finitely many
coherence axioms.

\section{Subdivisions}\label{sec:subdivisions}

When one is working with rewriting $2$-theories or categorical
algebraic structures more generally, one typically draws diagrams
representing morphisms in the free structure. The purpose of this
section is to formalise these diagrams as ambient isotopy classes of
planar directed graphs. This provides a mathematical setting for the
manner in which one typically shows that a parallel pair of morphisms
is equal: by finding a subdivision of the pair whose faces commute by
virtue of functoriality, naturality and the coherence axioms. Within
this setting, we examine properties that the underlying rewriting
theory of a $2$-theory must satisfy in order to ensure that each
parallel pair of morphisms admits only finitely many such
subdivisions. This forms the basis for the coherence theorems
developed in the remainder of the chapter.

A subdivision of a parellel pair of reductions is, in the first
instance, a collection of reductions having the same source and
target. This collection forms a graphical structure.

\begin{definition}
  An \emph{st-graph} is a labelled directed graph $G$ (possibly with
  loops and multiple edges) together with
  two distinguished vertices $u$ and $v$, called the source and target
  of $G$ respectively, such that for any other
  vertex $w \in G$, there exist paths $u \to w$ and $w \to v$ in $G$.
\end{definition}

By Lemma \ref{lem:irred}, we know that every reduction
generated by a rewriting $2$-theory is a composite of singular
reductions. Before we introduce the graph associated to a labelled
rewriting theory, we need to deal with a subtlety that arises due to
the presence of an equational theory on terms. Let $\R := \rtt$ be a
labelled rewriting theory. By the functoriality of the function
symbols $F \in \F$, every equation in $[\E_\F]$ induces an equation on
reductions. Thus, we may form the quotient $\Sing(\R)/[\E_\F]$. We call
a member of $\Sing(\R)/\E_\F$ an \emph{absolutely singular
  reduction}. 

\begin{definition}[Reduction graph]
  Let $\L := \ltr_X$ be a labelled rewriting theory.  The
  expression   $\Red(\L)$ denotes the \emph{reduction graph} of
  $\L$. This graph has 
  \begin{itemize}
    \item Vertices: The set $\Fr_{\langle \F \,|\, \E_\F\rangle}(X)$.
    \item Edges: Absolutely singular reductions in $\Fr_\L(X)$. 
  \end{itemize}
  The reduction graph of a rewriting $2$-theory is the reduction graph
  of its underlying labelled rewriting theory.
\end{definition}

A subdivision corresponds to a particular way of embedding an st-graph
in the oriented plane. Given a graph $G$, we use $|G|$ to denote its
geometric realisation. We write $\Rr^2$ for the plane with the
clockwise orientation. We use $G(s,t)$ to denote the set of paths from
$s$ to $t$ in $G$.

\begin{definition}
  Let $G$ be a graph and $\alpha,\beta \in G(s,t)$. A
  \emph{pre-subdivision} of $\langle\alpha,\beta\rangle$ is a pair
  $(S,\varphi)$ such that:
  \begin{enumerate}
    \item $S$ is an st-graph with source $s$ and target $t$.
    \item $\{\alpha,\beta\} \subseteq S \subseteq G$.
    \item $\varphi:|S| \hookrightarrow \Rr^2$ is a planar embedding.
    \item For every edge $\gamma \in S$, the image $\varphi(|\gamma|)$
      is contained in the region of $\Rr^2$ bounded by
      $\varphi(|\alpha|)$ and $\varphi(|\beta|)$.
  \end{enumerate}
  We use $\mathrm{PSub}_G(\alpha,\beta)$ to denote the set of all
  pre-subdivisions of $\langle\alpha,\beta\rangle$ in $G$.
\end{definition}

The definition of pre-subdivisions admits too many different embeddings
of the same graph. To this end, we define a useful equivalence
relation on pre-subdivisions.  In the present context, we say that two
embeddings $f,g:G \hookrightarrow \Rr^2$ are \emph{ambiently isotopic}
if there is an isotopy $h$ of the identity map of $\Rr^2$ such that
$h|_f = g$. In other words, $f$ and $g$ are ambiently isotopic if they 
differ only by a continuous deformation of $\Rr^2$. Intuitively, $f$
and $g$ are ambiently isotopic when they differ only by the size and
shape of their faces. 

Given a graph $G$ and $\alpha,\beta \in G(s,t)$, let
$\langle S_1,\varphi\rangle$ and $\langle S_2,\psi\rangle$ be
pre-subdivisions of $\langle\alpha,\beta\rangle$. Define $\sim$ to be
the equivalence relation on $\mathrm{PSub}_G(\alpha,\beta)$ generated
by setting $\langle S_1,\varphi\rangle \sim \langle S_2,\psi\rangle$
if:
\begin{enumerate}
  \item $S_1 = S_2$.
  \item $\varphi$ and $\psi$ are ambiently isotopic.
\end{enumerate}

Ambient isotopy is still not quite enough to identify all subdivisions
representing the same categorical diagram. The reason for this is that
reflecting the plane about some axis maps a subdivision to an
equivalent categorical diagram. Let $E(2)$ be the Euclidean group of
the plane --- the group of all rotations, translations and
reflections of the plane. 

We write $\Sub_G(s,t)$ for the quotient
$(\mathrm{PSub}_G(s,t)/\!\sim)/\!E(2)$ .

\begin{definition}
  For a directed graph $G$ and $\alpha,\beta \in G(s,t)$, a
  \emph{subdivision} of $\langle\alpha,\beta\rangle$ is a member of
  $\Sub_G(s,t)$. For a labelled rewriting theory $\L$,
  a subdivision of a parallel pair of reductions $\alpha,\beta \in
  \Fr(\L)$ is a subdivision of $\langle\alpha,\beta\rangle$ in
  $\Red(\L)$. The set of all such subdivisions is denoted
  $\Sub_{\L}(\alpha,\beta)$.
\end{definition}

Recall that a directed graph $G$ is \emph{locally finite} if $G(s,t)$
is finite for all vertices $s,t \in G$. The following sequence of
lemmas establishes a correspondence between local finiteness and
finitely many subdivisions. 

\begin{lemma}\label{lem:fin_image}
 For a directed graph $G$ and a finite planar $st$-subgraph $S \le
 G(s,t)$ with source $s$ and target $t$, there are only finitely many
 subdivisions of  $\alpha,\beta \in G(s,t)$ having graph $S$. 
\end{lemma}
\begin{proof}
  Since we only consider embeddings of $S$ up to ambient isotopy and
  Euclidean group action, a
  subdivision with graph $S$ is completely determined by the set of
  edges mapped to the region bounded by $\varphi(|\gamma_1|)$ and
  $\varphi(|\gamma_2|)$ for every 
  parallel pair of paths $\gamma_1,\gamma_2 \in S$. Since $S$ is
  finite, there are only finitely many possibilities for this.
\end{proof}

\begin{lemma}\label{lem:subgraph}
  An st-graph with source $s$ and target $t$ is finite if and only if
  it has finitely many planar st-subgraphs with source $s$ and target $t$.
\end{lemma}
\begin{proof}
($\Rightarrow$) A finite graph has finitely many subgraphs, so it
certainly has finitely many planar subgraphs.

($\Leftarrow$) Suppose that $G$ is an infinite st-graph with source $s$
and target $t$. Each path from $s$ to $t$ in $G$ determines a planar
subgraph of $G$, hence $G$ has infinitely many planar subgraphs with
source $s$ and target $t$.
\end{proof}

Combining Lemma  \ref{lem:fin_image} and Lemma \ref{lem:subgraph},
we obtain the desired correspondence.

\begin{lemma}\label{lem:fin_sub}
If $G$ is a directed graph containing vertices $s$ and $t$, then
$G(s,t)$ is finite if and only if $\Sub_G(\alpha,\beta)$ is finite for
all $\alpha,\beta \in G(s,t)$\qed
\end{lemma}

\subsection{Ensuring local finiteness}

By Lemma \ref{lem:fin_sub}, in order to ensure that every parallel
pair of paths in a directed graph has finitely many subdivisions, we
need only establish that the graph is locally finite. To this end, we
make the following definition.

\begin{definition}
  Let $G$ be a directed graph. A \emph{quasicycle} in $G$ is a pair
  $(T,t)$ such that:
  \begin{enumerate}
    \item $T$ is an infinite chain $t_0 \to t_1 \to \dots$ in $G$ 
    \item $t$ is a vertex in $G$.
    \item $G$ contains a path $t_i \to t$ for all $i \in \mathbb{N}$. 
  \end{enumerate}
\end{definition}

\begin{figure}[h]
\[
\vcenter{
\xymatrix{
  {t_0} \ar[r] \ar@/_1pc/[ddrr] & {t_1} \ar[r] \ar@/_/[ddr] & 
  {t_2} \ar[r] \ar[dd] & {t_3} \ar[r] \ar@/^/[ddl] &
  {t_4} \ar[r] \ar@/^1pc/[ddll]  & {\dots}\\
  \\
  && {t}
}
}
\]
\caption{A quasicycle}\label{fig:quasicycle}
\end{figure}
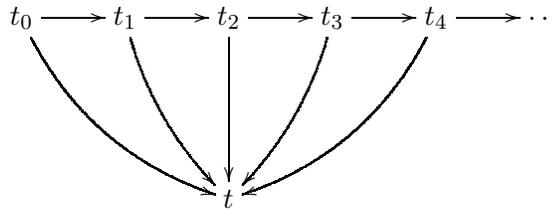

Quasicycles earn their name by being a slightly weaker notion than a
cycle. Figure \ref{fig:quasicycle} gives an example of a quasicycle
that is not a cycle. On the other hand, we have the following easy
result.

\begin{lemma}\label{lem:cycle}
  Let $C$ be a directed cycle and $c$ be a vertex in $C$. Then,
  $(C,c)$ is a quasicycle. \qed
\end{lemma}

For a directed graph $G$ and a vertex $s \in G$, we use $\Out_G(s)$ to
denote the set $\{t \in V(G) : G \textrm{ contains an edge } s\to
t\}$. We say that $G$ is \emph{finitely branching} if $\Out_G(s)$ is
finite for all vertices $s \in G$. One of our main technical tools is
the following graphical version of K\"{o}nig's Tree Lemma.

\begin{lemma}\label{lem:quasi}
  A finitely branching directed graph is locally finite if and only if
  it contains no quasicycles.
\end{lemma}
\begin{proof}
  Let $G$ be a labelled finitely branching directed graph.

  ($\Rightarrow$) Suppose that $G$
  contains a quasicycle $(T,t)$, where $T = t_0
  \stackrel{\alpha_0}{\to} t_1 \stackrel{\alpha_1}\to \dots$. If $t_i
  = t$ for some $i \in \N$ then $G(t_i,t_j)$ is infinite for all $j >
  i$. So, suppose that $t_i \ne t$ for all $i \in \N$. Since $t_i \to
  t$ for all $i \in \N$, there must be infinitely many pairs
  $(i,\beta_i)$, where $i \in \N$ and $\beta_i:t_i\to t$ is a path that
  does not factor through $t_j$ for any $j > i$. So, $G(t_0,t)$ is
  infinite.

  ($\Leftarrow$) Suppose that
  $G(s,t)$ is infinite. Since $\Out_G(s)$ is finite, it follows from the
  pigeon hole principle that there must exist some $s_0 \in \Out_G(s)$
  and an edge $\alpha_0:s \to s_0$ such that $G(s_0,t)$ is
  infinite. Continuing recursively, we obtain an infinite chain $s
  \stackrel{\alpha_0}{\to} s_0 \stackrel{\alpha_1}{\to} s_1
  \stackrel{\alpha_2}{\to} \dots$ such that $G$ contains  a path $s_i
  \to t$ for all $i \in \N$. So, $G$ contains a quasicycle. 
\end{proof}

By making use of the reduction graph of a labelled rewriting theory,
we can shift our terminology for directed graphs to labelled rewriting
theories. 

\begin{definition}
  A labelled rewriting theory $\L$ is \emph{quasicycle-free} if
  every quasicycle in  $\Red(\L)$ contains cofinitely many
  identity reductions. It is locally finite if   $\Red(\L)$ is
  locally finite and it is finitely branching if $\Red(\L)$ is
  finitely branching.
\end{definition}

Recall that an equation $s = t$ is called \emph{balanced} if $s$ and
$t$ contain precisely the same variables and it is called
\emph{linear} if it is balanced and each variable appears precisely
once in each of $s$ and $t$.

\begin{definition}
  A labelled rewriting theory $\L := \ltr$ is \emph{term-linear} if
  $\E_\F$ contains only linear equations.
\end{definition}

A reduction rule $\alpha:[s]\to[t]$ is called \emph{non-increasing} if
$\var(t) \subseteq \var(s)$. A labelled rewriting
theory $\L$ is \emph{non-increasing} if every reduction rule in $\L$
is non-increasing. 

\begin{proposition}\label{prop:linearfinite}
A finitely presented labelled rewriting theory is finitely branching
if  it is term-linear and non-increasing. 
\end{proposition}
\begin{proof}
  Let $\L := \ltr$ be a finitely  presented non-increasing term-linear
  labelled rewriting theory. Without loss of generality, we may assume
  that $\T = \{\rho\}$. Suppose that the vertex $[s]$ in $\Red(\L)$ is
  infinitely branching. Since $\rho$ is non-increasing, there must be
  infinitely  many terms $s_1,s_2,\dots \in [s]$ containing the same
  number of unary and binary function symbols as $\rho$, such that
  each $s_i$ contains a $\rho$-redex as a subterm. But this is
  impossible, since $\E_\F$ is linear.
\end{proof}

A labelled rewriting theory that is not term-linear may be infinitely
branching, even if it is finitely presented and non-increasing.

\begin{example}
  Let $\L$ be the labelled rewriting theory consisting of the binary
  function symbol $\F$, the equation $s =
  F(s,s)$ and the reduction rule $\rho: t \to t'$. Then, in
  $\Fr(\L)$, we have: 
  \[ 
  t = F(t,t) = F(F(t,t),t) = F(F(F(t,t),t),t) = \dots
  \]
  The reduction rule $\rho$ induces maps from $[t]$ to:
  \begin{equation}\label{eq:infterm}
  t', F(t',t), F(F(t',t),t), F(F(F(t',t),t),t),\dots
  \end{equation}
  Since the terms in (\ref{eq:infterm}) are pairwise unequal, $\L$ is
  infinitely branching.  
\end{example}

Lemmas \ref{lem:fin_sub} and \ref{lem:quasi} imply that a finitely
branching quasicycle-free labelled rewriting theory has only finitely  
many subdivisions for every parallel pair of reductions. A ready
supply of such theories is provided by the following observation,
which follows immediately from the definitions. 

\begin{lemma}
  A terminating labelled rewriting theory is quasicycle-free.\qed  
\end{lemma}

By Lemma \ref{lem:cycle}, a quasicycle-free directed graph is
acyclic. The following theorem establishes that every face of a
subdivision in an acyclic graph is itself a parallel pair of paths. It
was originally discovered by Power \cite{Power:pasting2} in his
investigation of pasting diagrams in $2$-categories.

\begin{theorem}[Power \cite{Power:pasting2}]\label{thm:power}
  A planar $st$-graph is acyclic if and only if every face has a
  unique source and target.\qed
\end{theorem}

Theorem \ref{thm:power} readily leads to the following result by
induction over the number of faces in a subdivision.

\begin{proposition}\label{prop:faces}
  Let $\R := \rtt$ be an acyclic rewriting $2$-theory and let
  $\alpha,\beta \in \Red(\R)(s,t)$. Then, the following statements are
  equivalent: 
  \begin{enumerate}
    \item $\alpha = \beta$ in $\Fr(\R)$. 
    \item There is a subdivision of $\langle\alpha,\beta\rangle$ in
      $\Red(\R)(s,t)$ such that each face commutes in
      $\Fr(\R)$.
    \item There is a subdivision of $\langle\alpha,\beta\rangle$ in
      $\Red(\R)(s,t)$ such that each face is either an instance of
      functoriality, or an instance of naturality or an instance of
      one of the equations in $\E_\T$. \qed
  \end{enumerate}
\end{proposition}

In the following section, we use the tools developed so far to tackle
the Lambek coherence problem.

\section{Lambek coherence}\label{sec:lambekcoherence}

With Proposition \ref{prop:faces} and Lemma \ref{lem:fin_sub}, one may
be inclined to think that a Lambek coherence theorem should be
immediately forthcoming, since we know that every quasicycle-free
finitely branching rewriting $2$-theory 
has only finitely many subdivisions for each parallel pair of
reductions and we can just check every face to see whether it is an
instance of functoriality, naturality or a cohenrence axiom. There is,
however, one catch --- we may not be able to decide whether a given face
is an instance of an axiom.

\begin{definition}[Unification]
  Let $\F$ be a ranked set of function symbols on a set $X$ and
  $\E_\F$ be an equational theory on $\Fr_\F(X)$. An
  $\E_\F$-unification problem is a finite set:
  \[
  \Gamma = \{(s_1,t_1),\dots, (s_n,t_n)\},
  \]
  where for $1 \le i \le n$, we have that $s_i$ and $t_i$ are in
  $\Fr_\F(X)$. A \emph{unifier} for  $\Gamma$ is a homomorphism
  $\sigma: X \to \Fr_\F(X)$ such 
  that $\sigma(s_i) =_{\E_\F} \sigma(t_i)$ for all $1 \le i \le
  n$. The set $\Gamma$ is \emph{unifiable} if it admits at least one
  unifier. 
\end{definition}

Unification theory is an important technical component of automated
reasoning and logic programming, as it provides a means for testing
whether two sequences of terms are syntactic variants of each other. A
good survey of the field is provided by
\cite{Baader:unification}. In the case where the theory $\E_\F$ is
empty, the unification problem is readily shown to be
decidable (see \cite{Baader:unification} for details). Unfortunately,
the equational unification problem is in general undecidable. 

\begin{definition}
  A labelled rewriting theory $\rtt$ has \emph{decidable term
    unification} if $\langle \F \,|\, \E_\F\rangle$ has a decidable
  unification problem. 
\end{definition}

We can finally establish a general Lambek coherence theorem.

\begin{theorem}[Lambek Coherence]\label{thm:lambek}
  A finitely branching quasicycle-free rewriting $2$-theory with
  decidable term unification is Lambek Coherent.  
\end{theorem}
\begin{proof}
  Let $\R$ be a rewriting $2$-theory satisfying
  the hypotheses and let $\alpha,\beta \in \Red(\R)(s,t)$. By
   Lemma \ref{lem:fin_sub}, we can enumerate the subdivisions of
   $\langle\alpha,\beta\rangle$. Since 
  each subdivision has only finitely many faces and $\R$ has decidable
  term unification, we may apply Proposition \ref{prop:faces}  to determine
  whether every face of a subdivision commutes in $\Fr(\R)$.
\end{proof}

Unfortunately, we may not be able to determine whether a labelled
rewriting theory is quasicycle-free.

\begin{corollary}
  It is undecidable whether a finitely branching rewriting $2$-theory
  theory with decidable term unification is quasicycle-free.
\end{corollary}
\begin{proof}
  The rewriting $2$-theory constructed in the proof of Theorem
  \ref{thm:undecidable}  has an empty equational theory on
  terms and so has decidable term unification. It
  follows from Theorem \ref{thm:lambek} that, were we able to
  determine whether the theory is quasicycle-free, then we would
  be able to decide whether a finite monoid presentation has a
  decidable word problem. 
\end{proof}

As a particular application of Theorem \ref{thm:lambek}, any
terminating rewriting $2$-theory with an empty equational theory on
terms is Lambek coherent. This includes, amongst others, categories
with a directed associativity  \cite{Laplaza:associative}. The
unification problem for an associative binary symbol $\tensor$
together with an identity $I$ for $\tensor$ is decidable 
\cite{Baader:unification}. It follows then, from Theorem
\ref{thm:lambek} that the following rewriting $2$-theories are Lambek
coherent (in each case we need only check that the $2$-theory is
terminating): 
\begin{itemize}
  \item Distributive categories with strict associativities and strict
    units \cite {Laplaza:distributive}.
  \item Weakly distributive categories with strict associativity and
    strict units \cite{CockettSeely:wdc}.
\end{itemize}

An example of a non-terminating theory that is Lambek-coherent
is provided by the system $F(x) \to F(F(x))$, since this is easily
seen to be quasicycle-free. 

In the following section, we continue our investigation of quasicycle
free theories and derive sufficient conditions for such a system
to be Mac Lane coherent.

\section{Mac Lane coherence}\label{sec:maclanecoherence}

The last section was concerned with deciding whether a given parallel
pair of morphisms is equal or, equivalently, whether a given diagram 
in general position commutes. Our rough goal in this section is to
find a minimal set of diagrams in general position whose commutativity
implies the commutativity of all other such diagrams in $\Fr(\R)$
for some rewriting $2$-theory $\R$. To this end,
we define what it means for one subdivision to be finer than
another. The driving idea is that we only wish to consider those
subdivisions that do not embed into a finer subdivision.

\begin{definition}
  Let $G$ be a directed graph and
  $\alpha,\beta \in G(s,t)$ and $(S_1,\varphi),(S_2,\psi)
  \in \Sub_G(\alpha,\beta)$. We say that $(S_1,\varphi)$ is
  \emph{coarser} than $(S_2,\psi)$ if there is a graph embedding
  $\Lambda:S_1\to S_2$ making the following diagram
  commute. In this case, we also say that $(S_2,\psi)$ is \emph{finer}
  than $(S_1,\varphi)$ and we write $(S_1,\varphi) \preceq
  (S_2,\psi).$ 
  \[
  \xymatrix@R=-0.1pc{
    {S_1} \ar[r]^{|\cdot |} \ar[dddd]_{\Lambda}& {|S_1|} \ar@/^/[ddr]^{\varphi}
    \ar[dddd]_{|\Lambda|}\\
    \\
    &&{\R^2}\\
    \\
    {S_2} \ar[r]_{|\cdot |} &{|S_2|} \ar@/_/[uur]_{\psi}
  }
  \]
  We define the refinement order to be the antisymmetric closure of
  $\preceq$. 
\end{definition}

It is immediate from the definitions that the set of subdivisions of a
parallel pair of morphisms forms a poset under refinement. We shall
abuse notation slightly in the following definition and write
$\preceq$ for the refinement order.   

\begin{definition}
  Let $G$ be a directed graph and 
  $\alpha,\beta \in G(s,t)$. A \emph{maximal subdivision} of
  $\langle\alpha,\beta\rangle$ is a maximal element of
  $(\Sub_G(\alpha,\beta),\preceq)$. 
\end{definition}

The idea behind the definition of a maximal subdivision is that these
are precisely the ones which cannot be further subdivided. This leads
to the following lemma. 

\begin{lemma}\label{lem:maximalface}
  A finitely branching quasicycle-free rewriting $2$-theory is Mac
  Lane coherent if and only if every parallel pair of 
  reductions in general position admits a maximal subdivision, each
  face of which commutes. 
\end{lemma}
\begin{proof}
  The direction ($\Leftarrow$) follows from induction over the number
  of faces. For the other direction,
  let $\R$ be a finitely branching quasicycle-free rewriting
  $2$-theory. Let $\alpha,\beta \in \Red_\R(s,t)$. Since $\R$ 
  is quasicycle-free and finitely branching, it follows from 
  Lemma \ref{lem:fin_sub} that $\Sub_{\R}(\alpha,\beta)$ is
  finite. Therefore, $\langle\alpha,\beta\rangle$ admits a maximal
  subdivision $(S,\varphi)$. By Theorem \ref{thm:power}, every face of
  $(S,\varphi)$ has a unique source and target. Since $\R$ is Mac Lane
  coherent, each of these faces commutes.
\end{proof}

In order to make Lemma \ref{lem:maximalface} effective, we need to
characterise those parallel pairs of morphisms that can occur as faces
of a maximal subdivision.

\begin{definition}[Zig-zag subdivision]
Let $G$ be a directed graph and $\alpha,\beta \in G(s,t)$. Suppose
that  
\begin{align*}
  \alpha &= s \stackrel{\alpha_0}{\to}a_0\stackrel{\alpha_1}{\to}\dots
  \stackrel{\alpha_{n-1}}{\to}a_{n-1}\stackrel{\alpha_n}{\to}t\\ 
  \beta &=   s \stackrel{\beta_0}{\to}b_0\stackrel{\beta_1}{\to}\dots
  \stackrel{\beta_{m-1}}{\to}b_{m-1}\stackrel{\beta_m}{\to}t
\end{align*}
and that each $\alpha_i$ and $\beta_i$ is singular. Let $U$ be the
forgetful functor from directed graphs to graphs that forgets the
direction of edges. A
\emph{zig-zag subdivision} of $\langle\alpha,\beta\rangle$ is a
subdivision $(S,\varphi)$ of $\langle\alpha,\beta\rangle$ such that
$U(S)$ contains a path from $U(a_i)$ to $U(b_j)$ for some pair
$(i,j)$, with $0 \le i \le n-1$ and $0 \le j \le m-1$. We call the
preimage of this path \emph{the zig-zag of $S$}.
\end{definition}

\begin{figure}[ht]
    \begin{center}
      \begin{tabular}{ll}
        $
        \def\objectstyle{\scriptstyle}
        \def\labelstyle{\scriptstyle}
        \begin{xy}
          (0,0)*+{\cdot}="a";
          (15,10)*+{\cdot}="b";
          (15,-10)*+{\cdot}="c";
          (33,0)*+{\cdot}="d";
          {\ar@{->}@/_/ "c";"d"};
          {\ar@{->}@/^/ "a"; "b"};
          {\ar@{->}@/_/ "a"; "c"};
          {\ar@{->}@/^/ "b";"d"};
          {\ar@{->} "c";"b"}
        \end{xy}
        $
        & \quad
        $
        \def\objectstyle{\scriptstyle}
        \def\labelstyle{\scriptstyle}
        \begin{xy}
          (0,0)*+{\cdot}="a";
          (12,10)*+{\cdot}="b";
          (12,-10)*+{\cdot}="c";
          (33,0)*+{\cdot}="d";
          (17,0)*+{\cdot}="e";
          {\ar@{->}@/_/ "c";"d"};
          {\ar@{->}@/^/ "a"; "b"};
          {\ar@{->}@/_/ "a"; "c"};
          {\ar@{->}@/^/ "b";"d"};
          {\ar@{->} "e"; "d"};
          {\ar@{->} "b";"e"};
          {\ar@{->} "c";"e"};
        \end{xy}
        $
        \\ \\
        $
        \def\objectstyle{\scriptstyle}
        \def\labelstyle{\scriptstyle}
        \begin{xy}
          (0,0)*+{\cdot}="a";
          (21,10)*+{\cdot}="b";
          (21,-10)*+{\cdot}="c";
          (33,0)*+{\cdot}="d";
          (17,0)*+{\cdot}="e";
          {\ar@{->}@/_/ "c";"d"};
          {\ar@{->}@/^/ "a"; "b"};
          {\ar@{->}@/_/ "a"; "c"};
          {\ar@{->}@/^/ "b";"d"};
          {\ar@{->} "a"; "e"};
          {\ar@{->} "e";"b"};
          {\ar@{->} "e";"c"};
        \end{xy}
        $

        & \qquad

        $
        \def\objectstyle{\scriptstyle}
        \def\labelstyle{\scriptstyle}
        \begin{xy}
          (0,0)*+{\cdot}="a";
          (16,10)*+{\cdot}="b";
          (16,-10)*+{\cdot}="c";
          (33,0)*+{\cdot}="d";
          (18,3)*+{\cdot}="e";
          (14,-3)*+{\cdot}="f";
          {\ar@{->}@/_/ "c";"d"};
          {\ar@{->}@/^/ "a"; "b"};
          {\ar@{->}@/_/ "a"; "c"};
          {\ar@{->}@/^/ "b";"d"};
          {\ar@{->} "a"; "f"};
          {\ar@{->} "b";"e"};
          {\ar@{->} "f";"c"};
          {\ar@{->} "f";"e"};
          {\ar@{->} "e";"d"};
        \end{xy}
        $

      \end{tabular}
    \end{center}
\caption{A few zig-zag subdivisions.}\label{fig:zigzag}
\end{figure}
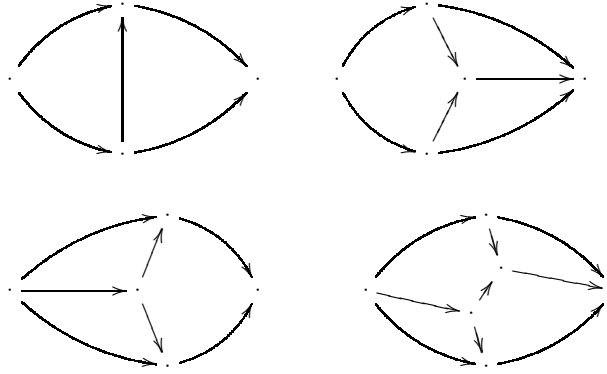

\begin{definition}[Diamond]
  Let $G$ be a directed graph. A pair $\alpha,\beta \in G(s,t)$ is
  called a \emph{diamond} if it does not admit a zig-zag subdivision.
\end{definition}

The idea behind the definition of a diamond is that any subdivision
containing a face that admits a zig-zag subdivision cannot be a
maximal subdivision. This is made precise in the following
proposition.

\begin{proposition}\label{prop:diamond}
  Let $G$ be an acyclic directed graph and $\alpha,\beta \in
  G(s,t)$. Every face of a maximal subdivision of
  $\langle\alpha,\beta\rangle$ is a diamond.  
\end{proposition}
\begin{proof}
  Let $G$ be an acyclic directed graph and let $(S,\varphi)$ be a maximal 
  subdivision of $\alpha,\beta \in G(s,t)$. By Theorem
  \ref{thm:power}, every face of $S$ has a unique source and
  target. That is, every face consists of a parallel pair of
  reductions $\eta,\psi:u \to v$. Suppose that $\langle
 \eta,\psi\rangle$ is a face of $S$ that is not a diamond. That
  is, it admits a zig-zag subdivision. So, we have
  \begin{align*}
    \eta &= u \stackrel{\eta_1}{\to} w \stackrel{\eta_2}{\to}
    v\\ 
    \psi &= u \stackrel{\psi_1}{\to} x \stackrel{\psi_2}{\to} v,
  \end{align*}
  and a zig-zag $\gamma$ between $w$ and $x$ that is a part of a
  subdivision of $\langle\eta, \psi\rangle$. By maximality, $\gamma$
  must be contained in $S$. Since $\langle\eta, \psi\rangle$ is a
  face, $\varphi(|\gamma|)$ cannot lie in the region bounded by
  $\varphi(|\eta|)$ and $\varphi(|\psi|)$.  So, we are in one of the
  situations depicted in Figure \ref{fig:cantembed}.

  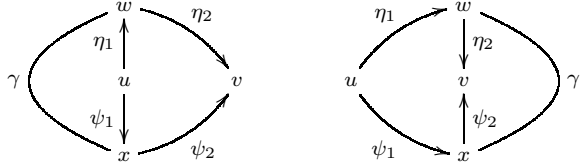
\begin{figure}[h]
    \begin{tabular}{ll}
      $
      \def\objectstyle{\scriptstyle}
      \def\labelstyle{\scriptstyle}
      \begin{xy}
        (0,0)*+{u}="u";
        (0,10)*+{w}="w";
        (0,-10)*+{x}="x";
        (15,0)*+{v}="v";
        {\ar@{-}@/^3pc/^{\gamma} "x";"w"};
        {\ar@{->}^{\eta_1} "u";"w"};
        {\ar@{->}_{\psi_1} "u"; "x"};
        {\ar@{->}@/_/_{\psi_2} "x"; "v"};
        {\ar@{->}@/^/^{\eta_2} "w";"v"};
      \end{xy}
      $

      &\qquad

      $
      \def\objectstyle{\scriptstyle}
      \def\labelstyle{\scriptstyle}
      \begin{xy}
        (0,0)*+{v}="v";
        (0,10)*+{w}="w";
        (0,-10)*+{x}="x";
        (-15,0)*+{u}="u";
        {\ar@{->}_{\psi_2} "x";"v"};
        {\ar@{->}^{\eta_2} "w";"v"};
        {\ar@{->}@/^/^{\eta_1} "u";"w"};
        {\ar@{->}@/_/_{\psi_1} "u";"x"};
        {\ar@{-}@/^3pc/^{\gamma} "w"; "x"};
      \end{xy}
      $
    \end{tabular}
    \caption{Possible embeddings of $\gamma$.}\label{fig:cantembed} 
  \end{figure}
  
  Suppose that we are in the situation depicted in the left hand
  diagram of Figure \ref{fig:cantembed}. Since $\gamma$ is contained
  in $S$ and since $S$ is an st-graph, there is a path $s
  \stackrel{\delta}{\to} u$. By planarity, $\delta$ must factor
  through a vertex in $\gamma$ or $\eta_2$ or $\psi_2$. If $\delta$
  factors through a vertex in $\eta_2$ or $\psi_2$ then it is clear
  that $G$ contains a cycle, contradicting the fact that $G$ is
  acyclic. So, we must have $s \stackrel{\delta_1}{\to} z
  \stackrel{\delta_2}{\to} u$ for some vertex $z$ in
  $\gamma$. However, since $\gamma$ appears in a subdivision of
  $\langle \eta,\psi\rangle$, there is a path $u \stackrel{\zeta}{\to}
  z$ in $G$. Then, $\delta_2\cdot\zeta$ forms a cycle in $G$,
  contradicting the fact that $G$ is acyclic. So, $\gamma$ cannot be
  embedded as in the left hand picture of Figure
  \ref{fig:cantembed}. Dually, it cannot be embedded as in the right
  hand picture of Figure \ref{fig:cantembed}.
  
  Therefore, the zig-zag $\gamma$ must be embedded within the face
  bounded by $\langle\eta,\psi\rangle$, contradicting the maximality
  of $(S,\varphi)$. So, $\langle \eta,\psi\rangle$ must be a diamond.  
\end{proof}

Combining Lemma \ref{lem:maximalface} and Proposition
\ref{prop:diamond}, we obtain our general version of the Mac Lane
Coherence theorem. 

\begin{theorem}[Coherence]\label{thm:maclane}
  A finitely branching quasicycle-free rewriting $2$-theory
  $\R$ is Mac Lane coherent if and only if every diamond in 
  $\Red(\R)$ commutes in $\Fr(\R)$. \qed
\end{theorem}

Theorem \ref{thm:maclane} says that in order to show that a
finitely branching rewriting $2$-theory is Mac Lane coherent, we need to
do two things:
\begin{enumerate}
  \item Show that $\Fr(\R)$ is quasicycle-free.
  \item Show that every diamond commutes.
\end{enumerate}

At the outset, showing that every diamond commutes can be a daunting
task. We can guide our investigations by  exploiting the properties of
critical spans.

\begin{definition}
Let $\R$ be a rewriting $2$-theory and let $\varphi_1$ and
$\varphi_2$ be singular morphisms in $\Fr(\R)$. We call
$\langle\varphi_1,\psi_1\rangle$ the \emph{source span} in a diagram of the
following form:
\[
    \def\objectstyle{\scriptstyle}
    \def\labelstyle{\scriptstyle}
    \vcenter{
      \xymatrix{
        {\cdot} \ar[r]^{\varphi_1} \ar[d]_{\psi_1} &  {\cdot} \ar[d]^{\varphi_2} \\
        {\cdot} \ar[r]_{\psi_2} & {\cdot},
      }
    }
\]
\end{definition}
 
If $\varphi_1$ and $\psi_1$ are singular, then there are three
possibilities for a diamond with source span $\langle \varphi_1,
\psi_1\rangle$:
\begin{enumerate}
  \item $\varphi_1$ and $\psi_1$ rewrite disjoint subterms. 
  \item $\varphi_1$ and $\psi_1$ rewrite nested subterms.
  \item $\varphi_1$ and $\psi_1$ rewrite overlapping
    subterms. Without loss of generality, we may assume that
    $\langle\varphi_1,\psi_1\rangle$ forms a critical span. 
\end{enumerate}

By analogy with Lemma \ref{lem:cp}, one
may hope to reduce the problem to only examining diamonds whose source
span is critical. Unfortunately, as the following two 
examples show, there may be more than one diamond whose source span
performs a given pair of nested or disjoint rewrites.

\begin{example}\label{ex:nested}
In this example we construct a terminating rewriting $2$-theory that
has more than one diamond with the same source span performing a
nested pair of rewrites. Let $\R$ be the $2$-theory consisting of unary
functor  symbols $I,J$ and $H$, together with the following reduction
rules: 

\begin{align*}
  I(x) &\to J(x)\\
  I(J(x)) &\to H(x)\\
  J(I(x)) &\to H(x)\\
\end{align*}
Then, $\Fr(\R)$ contains the following diagram:

\[
\vcenter{
  \def\objectstyle{\scriptstyle}
  \def\labelstyle{\scriptstyle}
  \begin{xy}
    (0,0)*+{I(I(a))}="a";
    (12,10)*+{J(I(a))}="b";
    (12,-10)*+{I(J(a))}="c";
    (40,0)*+{H(a)}="d";
    (24,0)*+{J(J(a))}="e";
    {\ar@/^0.2pc/ "b";"e"};
    {\ar@/_0.2pc/ "c";"e"};
    {\ar@{->}@/_/ "c";"d"};
    {\ar@{->}@/^/ "a"; "b"};
    {\ar@{->}@/_/ "a"; "c"};
    {\ar@{->}@/^/ "b";"d"};
  \end{xy}
}
\]
 Since there is no reduction $J(J(a)) \to H(a)$, both parallel
 reductions form diamonds. 
\end{example}

\begin{example}\label{ex:disjoint}
  In this example we construct a terminating rewriting $2$-theory that
  has more   than one diamond with the same source span performing a
  disjoint pair of rewrites. Let $\R$ be the rewriting $2$-theory consisting of
  unary functor  symbols $I$ and $J$, the binary functor symbol
  $\tensor$ and the   following reduction rules:  
  \begin{align*}
    I(x) &\to J(x)\\
    J(x)\tensor I(x) &\to H(x)\\
    I(x)\tensor J(x) &\to H(x)\\
  \end{align*}
  Then, $\Fr(\R)$ contains the following diagram:

  \[
  \vcenter{
    \def\objectstyle{\scriptstyle}
    \def\labelstyle{\scriptstyle}
    \begin{xy}
      (0,0)*+{I(A)\tensor I(A)}="a";
      (12,10)*+{I(A) \tensor J(A)}="b";
      (12,-10)*+{J(A) \tensor I(A)}="c";
      (40,0)*+{H(A)}="d";
      (24,0)*+{J(A) \tensor J(A)}="e";
      {\ar@/^0.2pc/ "b";"e"};
      {\ar@/_0.2pc/ "c";"e"};
      {\ar@{->}@/_/ "c";"d"};
      {\ar@{->}@/^/ "a"; "b"};
      {\ar@{->}@/_/ "a"; "c"};
      {\ar@{->}@/^/ "b";"d"};
    \end{xy}
  }
  \]
   Since there is no reduction $J(A)\tensor J(A) \to H(A)$, both parallel
   reductions form diamonds. 
\end{example}

Examples \ref{ex:nested} and \ref{ex:disjoint} serve to warn us that
the collection of diamonds behaves a lot more subtly than the
collection of spans, which are the typical objects of study in
traditional term rewriting theory. In the next section, we look at
when a labelled rewriting theory cannot be made into a Mac Lane coherent
rewriting $2$-theory by only finitely many coherence axioms.

%\end{comment}
\section{Finite Mac Lane coherence}\label{sec:finiteness}

In light of Theorem \ref{thm:maclane}, we have a reasonable strategy
for determining whether a given rewriting $2$-theory is Mac Lane
coherent. However, we are still left with the problem of determining
whether a given finitely presented labelled rewriting theory can be
extended to a finitely presented Mac Lane coherent rewriting
theory. Bearing in mind the results of the previous section, we have
two reasonable candidates for ensuring this property: quasicycle
freeness and termination. In this section, we show that neither of
these conditions suffice in general.
 
Given a rewriting $2$-theory $\R := \rtt$, we say that the labelled
rewriting theory $\ltr$ is the \emph{reduct} of $\R$.

\begin{definition}[Finitely Mac Lane coherent]
  A finitely presented labelled rewriting theory is \emph{finitely Mac
    Lane coherent} if it is the reduct of a finitely presented Mac
  Lane coherent rewriting $2$-theory.
\end{definition}

It is not  a priori obvious whether there exist theories that are
\emph{not} finitely Mac Lane coherent.  We can simplify our
investigation of this point somewhat by focusing on the most basic
diamonds.  

\begin{definition}[Basic diamond]
  A diamond $\Delta_1$ appearing in the reduction graph of a labelled
  rewriting theory $\L$ is \emph{basic} if it satisfies the following
  properties:
  \begin{enumerate} 
  \item  For every substitution $\sigma$ and diamond $\Delta_2$, if
    $\Delta_1 = \Delta_2^\sigma$, then $\sigma$ is a variable
    renaming.  
  \item For every unary functor $F \in \Fr(\L)$ and every diamond
    $\Delta_2$, if $F(\Delta_2) = \Delta_1$, then $F = 1$. 
  \end{enumerate}
\end{definition}

The following theorem is immediate from Theorem \ref{thm:maclane}.

\begin{theorem}\label{lem:reduced}
A finitely presented, finitely branching labelled rewriting theory
$\L$ is finitely Mac Lane coherent if and only if $\Red(\L)$ contains
finitely many basic diamonds, up to variable renaming. \qed
\end{theorem}

In the following example, we construct an example of a quasicycle-free
labelled rewriting theory that is not finitely Mac Lane coherent.

\begin{example}\label{ex:infiniteqf}
  Let $\L$ be the labelled rewriting theory containing unary function
  symbols   $F,G,I$ and $H$, together with the following reduction
  rules: 
  \begin{align*}
    I(x) &\to G(I(x))\\
    I(x) &\to F(I(x))\\
    F(x) &\to F(F(x))\\
    G(x) &\to G(G(x))\\
    F(x) &\to H(x)\\
    G(x) &\to H(x)
  \end{align*}
  In order to show that $\L$ is quasicycle-free, it suffices to show
  that there is no term $t \in \Fr(\L)$ such that there are
  infinitely many reductions with target $t$ in $\Fr(\L)$. Let
  $\L^{-1}$ be the labelled rewriting theory with the same function
  symbols as $\L$ and a reduction rule $t \to s$ for every reduction
  rule $s \to t$ in $\L$. By Proposition \ref{prop:linearfinite},
  $\L^{-1}$ is finitely branching, so $\L$ is quasicycle-free. 

  However, $\Fr(\L)$ contains the following diagram:  
  \[
  \xymatrix@-0.2pc{
    & {G(I(a))} \ar[r]\ar[d] &{G^2(I(a))} \ar[r]\ar[d] & {G^3(I(a))}
    \ar[r]\ar[d]& {\dots}\\ 
    I(a) \ar@/^/[ur] \ar@/_/[dr] & {H(I(a))} & {H^2(I(a))} & {H^3(I(a))} &
    {\dots}\\ 
    & F(I(a)) \ar[r]\ar[u] & {F^2(I(a))}\ar[r]\ar[u] & {F^3(I(a))}
    \ar[r]\ar[u] & 
    {\dots} 
  }
  \]
  Since there are no reductions $H^i(a) \to H^j(a)$ for $i\ne j$, no
  finite collection of diamonds with source $I(a)$ implies the
  commutativity   of all others. So, $\L$ contains infinitely many
  substitution-reduced diamonds and it follows from Theorem
  \ref{lem:reduced} that it is not finitely Mac Lane coherent.
\end{example}

Example \ref{ex:infiniteqf} works by constructing infinitely many
substitution-reduced diamonds sharing a common source
span. Terminating rewriting theories are far better behaved.

\begin{lemma}\label{lem:finiteout}
  Let $\L$ be a finitely branching terminating labelled rewriting
  theory. Then, for every term $t \in \Fr(\L)$, the set 
  \[\Out(t) = \{t' \in \Fr(\L)  :  \exists \text{ a reduction } t
  \to t' \text{ in $\L$}\}\]
  is finite.
\end{lemma}
\begin{proof}
  Let $\L$ be a finitely branching terminating labelled rewriting
  theory and let $t \in \Fr(\L)$. Suppose that $\Out(t)$ is
  infinite. Let $\Out(t)_n$ be the set of terms $t' \in \Out(t)$
  such that a path of minimal length $t \to t'$ in $\Fr(\L)$ contains $n$
  edges. Since $\L$ is finitely branching and $\Out(t)$ is infinite,
  $\Out(t)_n$ is finite and nonempty for all $n \in \N$. This implies
  that $\L$ is not terminating, contradicting our assumptions. Thus,
  $\Out(t)$ is finite.
\end{proof}

It follows from the above lemma that there are only finitely many
substitution-reduced diamonds with a given source span in a finitely
branching, terminating labelled rewriting theory. One may be led by
this observation to posit that such a theory is necessarily finitely
Mac Lane coherent. However, there is still the possibility that there
are infinitely many distinct substitution-reduced diamonds, since the
diamonds may possess different source spans. This problem proves
to be insurmountable, as demonstrated in the following example.

\begin{example}
  In this example, we construct a finitely branching, terminating
  labelled rewriting theory that is not finitely Mac Lane
  coherent. Let $\L$ be the labelled rewriting theory consisting of
  the following function symbols:
  \begin{itemize}
    \item Nullary: $W$
    \item Unary: $S,S',T,T'$
    \item Binary: $F$
  \end{itemize}
  together with the following reduction rules:
  \begin{align*}
    F(F(a,b),c) &\stackrel{\pi}{\longrightarrow} F(a,b)\\
    F(S(a),b) &\stackrel{\alpha}{\longrightarrow} W\\
    F(a,T(b)) &\stackrel{\beta}{\longrightarrow} W\\
    S(a) &\stackrel{\sigma}{\longrightarrow} S'(a)\\
    T(a) &\stackrel{\tau}{\longrightarrow} T'(a)
  \end{align*}
  
  Then, $\Fr(\L)$ contains the following infinite sequence of
  diamonds:

  \[
  \vcenter
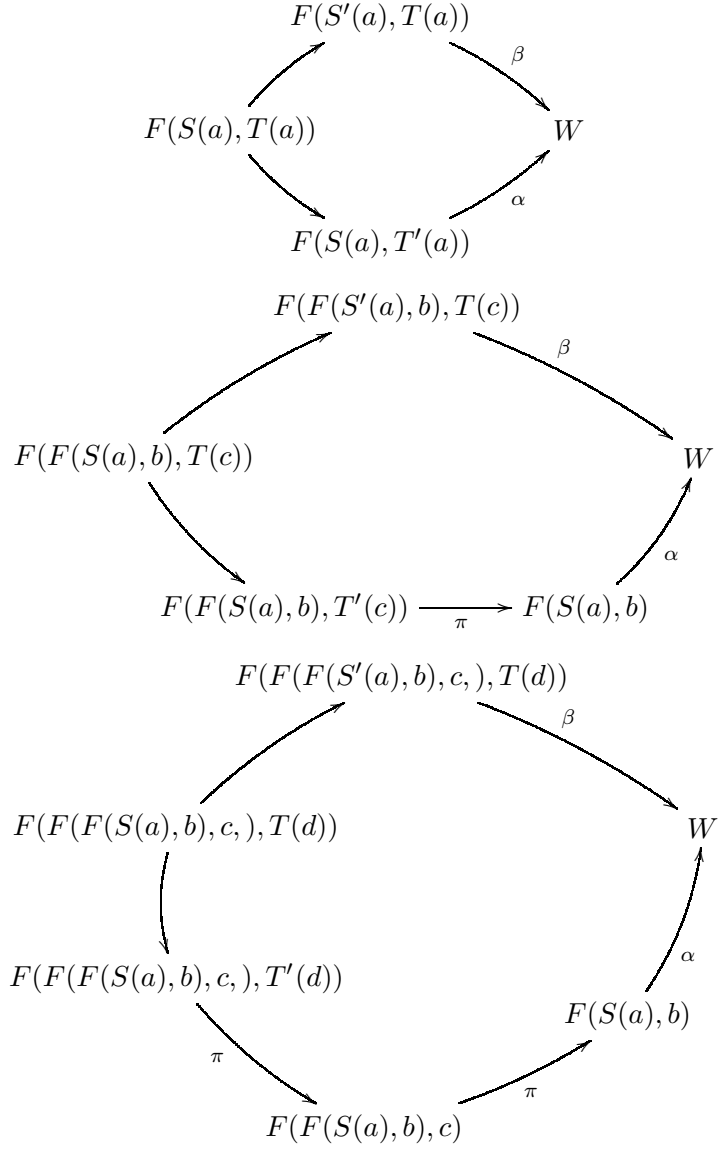
{
    \begin{xy}
      (-5,0)*+{F(S(a),T(a))}="a";
      (15,15)*+{F(S'(a),T(a))}="b";
      (15,-15)*+{F(S(a),T'(a))}="c";
      (40,0)*+{W}="d";
      {\ar@{->}@/^/ "a"; "b" };
      {\ar@{->}@/_/ "a";"c"};
      {\ar@{->}@/^/^<<<<<{\beta} "b"; "d" };
      {\ar@{->}@/_/_<<<<<{\alpha} "c";"d"};%
    \end{xy}
  }
  \]

  \[
  \vcenter{
    \begin{xy}
      (-5,0)*+{F(F(S(a),b),T(c))}="a";
      (30,20)*+{F(F(S'(a),b),T(c))}="b";
      (15,-20)*+{F(F(S(a),b),T'(c))}="c";
      (55,-20)*+{F(S(a),b)}="d";
      (70,0)*+{W}="e";
      {\ar@{->}@/^/ "a"; "b"};
      {\ar@{->}@/_/ "a";"c"};
      {\ar@{->}_<<<<<{\pi} "c"; "d"};
      {\ar@{->}@/^/^<<<<<{\beta} "b"; "e" };
      {\ar@{->}@/_/_<<<<<{\alpha} "d";"e"};%
    \end{xy}
  }
  \]

  \[
  \vcenter{
    \begin{xy}
      (0,0)*+{F(F(F(S(a),b),c,),T(d))}="a";
      (30,20)*+{F(F(F(S'(a),b),c,),T(d))}="b";
      (0,-20)*+{F(F(F(S(a),b),c,),T'(d))}="c";
      (25,-40)*+{F(F(S(a),b),c)}="d";
      (60,-25)*+{F(S(a),b)}="e";
      (70,0)*+{W}="f";
      {\ar@{->}@/^/ "a"; "b"};
      {\ar@{->}@/_/ "a";"c"};
      {\ar@{->}@/_/_<<<<<{\pi} "c"; "d"};
      {\ar@{->}@/_/_<<<<<{\pi} "d"; "e"};
      {\ar@{->}@/^/^<<<<<{\beta} "b"; "f" };
      {\ar@{->}@/_/_<<<<<{\alpha} "e";"f"};%
    \end{xy}
  }
  \]
  \[
  \begin{xy}
    (0,0)*+{};
    (10,-10)*+{\vdots};
  \end{xy}
  \]
  Since no diamond in the above sequence is a substitution-instance of
  another, it follows from Theorem \ref{lem:reduced} that $\L$ is not
  finitely Mac Lane coherent.
\end{example}

In this chapter, we have developed very general tools for
investigating coherence problems in non-confluent and non-terminating
rewriting $2$-theories. In the following chapter, we apply these
tools to a concrete theory arising in algebraic topology. 

\blanknonumber
\chapter{Iterated monoidal  categories}\label{ch:iterated} 

The coherence theorems developed in Chapter \ref{ch:nunf} are
primarily useful for investigating non-confluent and/or
non-terminating categorical structures. As we have seen previously in
Chapter \ref{ch:unf} a vast array of categorical structures suffer
from neither of these deficiencies. This might lead one to suspect
that any ``natural'' categorical structure is both confluent and
terminating. Unfortunately this is not the case. In this chapter, we
investigate the theory of iterated monoidal categories
\cite{iterated}, which arise naturally as a categorical model of
iterated loop spaces. This theory posesses two features making
its coherence problem difficult: it has a non-trivial equational
theory at the term level and it is non-confluent. A coherence theorem
is developed in \cite{iterated}, which says that there is a unique map
in an $n$-fold monoidal category between two terms without repeated
variables. The proof proceeds via an intricate double induction on the
number of variables and the dimension of the outermost tensor product
in the target of a morphism. In this chapter, we exploit Theorem
\ref{thm:maclane} to provide a more conceptually straightforward proof
of this theorem. 

\section{Definitions and basic properties}

An $n$-fold monoidal category contains $n$ monoidal structures linked
via ``interchange'' maps. The presentation given in \cite{iterated}
endows each tensor product with \emph{strict} associativity and unit
constraints. One of the interesting features of this structure is that
the $n$ tensor products all have \emph{the same} unit. This fact,
coupled with the equational theory on terms, allows for some
unexpected interplay between the interchange maps. For instance, it is
not immediately obvious that the structure is non-confluent. This
section introduces $n$-fold monoidal categories and explores some of
the subtleties that arise. 

\begin{definition}
  The rewriting $2$-theory for $n$-fold monoidal categories is denoted $\M_n$
  and consists of the following.
  \begin{enumerate}
    \item $n$ binary functor symbols: $\tensor_1,\dots,\tensor_n$
    \item A nullary functor symbol $I$
    \item For $1 \le i \le n$:
          \begin{align*}
            a \,\tensor_i\,(b\,\tensor_i\,c) &= (a\,\tensor_i\, b)\,\tensor_i\, c\\
            a \,\tensor_i\, I = a\\
            I \,\tensor_i\, a = a   
          \end{align*}
    \item For each pair $(i,j)$ such that $1\le i < j \le n$, there is
      a reduction rule, called interchange:
      $$\eta^{ij}_{a,b,c,d}: (a\,\tensor_j\,b)\,\tensor_i\,(c\,\tensor_j\,d)\to
      (a\,\tensor_i\, c)\,\tensor_j\, (b\,\tensor_i\,d)$$
  \end{enumerate}
  The interchange rules are subject to the following
  conditions: 
  \begin{enumerate}
    \item Internal unit condition: $\eta^{ij}_{a,b,I,I} =
      \eta^{ij}_{I,I,a,b} = id_{a\,\tensor_j\,b}$
    \item External unit condition: $\eta^{ij}_{a,I,b,I} =
      \eta^{ij}_{I,a,I,b} = id_{a\,\tensor_i\, b}$
    \item Internal associativity condition: The following diagram commutes:
      \[
      \def\objectstyle{\scriptstyle}
      \def\labelstyle{\scriptstyle}
      \vcenter{
        \xymatrix{
          {(a\,\tensor_j\,b)\,\tensor_i\,(c\,\tensor_j\,d)\,\tensor_i\,
            (e\,\tensor_j\,f)} 
          \ar[rrr]^{\eta^{ij}_{a,b,c,d}\,\tensor_i\, id_{e\,\tensor_i\,f}}
          \ar[dd]^{id_{a\,\tensor_j\,b}\,\tensor_i\,\eta^{ij}_{c,d,e,f}}
          &&& 
          {((a\,\tensor_i\,c)\,\tensor_j\,(b\,\tensor_i\,d))\,\tensor_i\,
            (e\,\tensor_j\,f)}
          \ar[dd]^{\eta^{ij}_{a\,\tensor_i\,c,b\,\tensor_i\,d,e,f}}
          \\\\
          {(a\,\tensor_j\,b)\,\tensor_i\,((c\,\tensor_i\,e)\,\tensor_j\,
            (d\,\tensor_i\,f))} 
          \ar[rrr]_{\eta^{ij}_{a,b,c\,\tensor_i\,e,d\,\tensor_i\,f}}
          &&& 
          {(a\,\tensor_i\,c\,\tensor_i\,e)\,\tensor_j\,(b\,\tensor_i\,d\,
            \tensor_i\,f)}
        }
      }
      \]
    \item External associativity condition: The following diagram
      commutes:
      \[
      \def\objectstyle{\scriptstyle}
      \def\labelstyle{\scriptstyle}
      \vcenter{
        \xymatrix{
          {(a\,\tensor_j\,b\,\tensor_j\,c)\,\tensor_i\,(d\,\tensor_j\,e\,
            \tensor_j\,f)}
          \ar[rrr]^{\eta^{ij}_{a\,\tensor_j\,b,c,d\,\tensor_j\,e,f}}
          \ar[dd]^{\eta^{ij}_{a,b\,\tensor_j\,c,d,e\,\tensor_j\,f}}
          &&&
          {((a\,\tensor_j\,b)\,\tensor_i\,(d\,\tensor_j\,c))\,\tensor_j\,
            (c\,\tensor_i\,f)}
          \ar[dd]^{\eta^{ij}_{a,b,d,c}\,\tensor_j\,id_{c\,\tensor_i\,f}}
          \\\\
          {(a\,\tensor_i\,d)\,\tensor_j\,((b\,\tensor_j\,c)\,\tensor_i\,
            (e\,\tensor_j\,f))} 
          \ar[rrr]^{id_{a\,\tensor_i\,d}\,\tensor_j\,\eta^{ij}_{b,c,e,f}}
          &&&
          {(a\,\tensor_i\,d)\,\tensor_j\,(b\,\tensor_i\,e)\,\tensor_j\,
            (c\,\tensor_i\,f)}
        }
      }
      \]
    \item Giant hexagon condition: The following diagram commutes:
      \[
      \def\objectstyle{\scriptstyle}
      \def\labelstyle{\scriptstyle}
      \vcenter{
        \begin{xy}
           (35,65)*+{((a\,\tensor_k\,b)\,\tensor_j\,(c\,\tensor_k\,d))\,
             \tensor_i\,  
             ((e\,\tensor_k\,f)\,\tensor_j\,(g\,\tensor_k\,h))}="a";
          (0,45)*+{((a\,\tensor_j\,c)\,\tensor_k\,(b\,\tensor_j\,d))\,
            \tensor_i\,
            ((e\,\tensor_j\,g)\,\tensor_k\,(f\,\tensor_j\,h))}="b";
          (70,45)*+{((a\,\tensor_k\,b)\,\tensor_i\,(e\,\tensor_k\,f))\,
            \tensor_j\,
            ((c\,\tensor_k\,d)\,\tensor_i\,(g\,\tensor_k\,h))}="c";
          (0,20)*+{((a\,\tensor_j\,c)\,\tensor_i\,(e\,\tensor_j\,g))\,
            \tensor_k\,
            ((b\,\tensor_j\,d)\,\tensor_i\,(f\,\tensor_j\,h))}="d";
          (70,20)*+{((a\,\tensor_i\,e)\,\tensor_k\,(b\,\tensor_i\,f))\,
            \tensor_j\,
            ((c\,\tensor_i\,g)\,\tensor_k\,(d\,\tensor_i\,h))}="e";
          (35,0)*+{((a\,\tensor_i\,e)\,\tensor_j\,(c\,\tensor_i\,g))\,
            \tensor_k\,
            ((b\,\tensor_i\,f)\,\tensor_j\,(d\,\tensor_i\,K)}="f";
          {\ar@{->}^{\eta^{ik}\,\tensor_j\,\eta^{ik}} "c";"e"}
          {\ar@{->}@/_/_<<<<<<{\eta^{ij}\,\tensor_k\,\eta^{ij}} (5,17);"f"}
          {\ar@{->}@/^/^>>>>>>{\eta^{ij}} "a";"c"}
          {\ar@{->}@/_/_>>>>>>{\eta^{jk}\,\tensor_k\,\eta^{jk}} "a";"b"}
          {\ar@{->}_{\eta^{ik}} "b";"d"}
          {\ar@{->}@/^/^<<<<<<{\eta^{jk}} (64,17);"f"}
        \end{xy}
      }
      \]
      In the giant hexagon, $(i,j,k)$ is such that $1 \le i < j < k
      \le n$ and the labels have the evident components.
      \end{enumerate}
\end{definition} 

Since the terms appearing in each reduction rule of $\M_n$ are linear,
we immediately obtain the following lemma. 

\begin{lemma}
  A reduction $[s] \to [t]$ in $\Fr(\M_n)$ is in general position if
  and only if $s$ and $t$ contain no repeated variables. \qed 
\end{lemma}

Because of the fact that an $n$-fold monoidal category is strictly
associative and has a strict unit, we can derive various maps
via Eckmann-Hilton style arguments. Two of these maps will be of
particular use to us. In the following, we assume that $(i,j)$ is such
that $1 \le i < j \le n$. The derived maps are as follows:

\begin{enumerate}
 \item Dimension raising: $
   \xymatrix@1{a\,\tensor_i\, b \ar[r]^{\iota^{ij}_{a,b}} & a \,\tensor_j\, 
     b.}$ This represents the following composition:
   \[
   \def\objectstyle{\scriptstyle}
   \def\labelstyle{\scriptstyle}
   \vcenter{
     \xymatrix@1{
       {a\,\tensor_i\, b}\ar[r]^-{=} & {(a\,\tensor_j\, 
         I)\,\tensor_i\, (I\,\tensor_j\, b)}
       \ar[r]^-{\eta^{ij}} & {(a\,\tensor_i\,  I)\,\tensor_j\, 
         (I\,\tensor_i\,  b)} \ar[r]^-{=} & {a\,\tensor_j\,  b}
     }
   }
   \]
 \item Twisted dimension raising: $
   \xymatrix@1{a\,\tensor_i\, b \ar[r]^{\tau^{ij}_{a,b}} & b \,\tensor_j\, 
     a.}$ This represents the following composition:
   \[
   \def\objectstyle{\scriptstyle}
   \def\labelstyle{\scriptstyle}
   \vcenter{
     \xymatrix@1{
       {a\,\tensor_i\, b}\ar[r]^-{=} & {(I\,\tensor_j\, 
         a)\,\tensor_i\, (I\,\tensor_j\, b)}
       \ar[r]^-{\eta^{ij}} & {(I\,\tensor_i\,  b)\,\tensor_j\, 
         (a\,\tensor_i\,  I)} \ar[r]^-{=} & {b\,\tensor_j\,  a}
     }
   }
   \]
\end{enumerate}

With the above maps, it is easy to see that iterated monoidal
categories do not have unique normal forms.

\begin{lemma}
 If $n \ge 2$, then $\Fr(\M_n)$ is not confluent.
\end{lemma}
\begin{proof}
  The following span is not joinable:
  \[
  \def\objectstyle{\scriptstyle}
  \def\labelstyle{\scriptstyle}
  \xymatrix{
    {a\,\tensor_i\, b} \ar[r]^{\iota^{in}_{a,b}} \ar[d]_{\tau^{in}_{a,b}}
    & {a\, \tensor_n\, b}\\
    {b\, \tensor_n\, a}
  }
  \]
\end{proof}

In the following section, we tackle the coherence problem for $\M_n$. 

\section{Proving coherence}

Our first step in investigatng the coherence problem for iterated
monoidal categories is to bring them into the realm of applicability of
Theorem \ref{thm:maclane}.

\begin{proposition}\label{prop:imcterm}
  $\M_n$ is quasicycle-free.
\end{proposition}
\begin{proof}
 Let $\M_n^{-1}$ be the rewriting $2$-theory that arises by replacing
 the reduction rules $\eta^{ij}$ in $\M_n$ with the following
 reduction rules, where $1 \le i < j \le n$:
 $$\xi^{ij}_{a,b,c,d}:  (a\,\tensor_i\, c)\,\tensor_j\,
 (b\,\tensor_i\,d) \to (a\,\tensor_j\,b)\,\tensor_i\,(c\,\tensor_j\,d).
 $$
 Given an object $[s] \in \Fr(\M_n^{-1})$, we may assume that $s$
 contains no instances of $I$. It follows that if there is a reduction $[s]
 \to [t]$ in $\Fr(\M_n^{-1})$, then $t$ contains the same variables as $s$,
 as well as the same number of function symbols. Since there are only
 finitely many such possibilities, $\M_n$ is quasicycle-free.
\end{proof}

Let $t$ be a term in $\Fr(\M_n)$. For a set $X \subseteq \var(t)$, we
write $t - X$ to denote the term resulting from substituting
$I$ for each variable in $X$. For instance $(a\tensor_i
b)\tensor_j(c \tensor_i d) - \{b,d\} = a\tensor_j c$. We say that a
term $u$ is \emph{in} a term $t$ and write $u \in t$ if there is some
$X \subseteq \var(t)$ such that $t - X = u$. Of crucial importance to
us is the following result of \cite{iterated}.

\begin{theorem}[\cite{iterated}]\label{thm:maps}
 Let $t$ and $u$ be terms in $\Fr(\M_n)$. A  necessary and sufficient
 condition for the 
  existence of a reduction $t \to u$ in $\Fr(\M_n)$ is that, for each $a,b
  \in \var(t)$, if $a\tensor_i b \in t$, then one of the following holds:
  \begin{itemize}
    \item There is some $j \ge i$ such that $a \tensor_j b \in u$
    \item There is some $j > i$ such that $b \tensor_j a \in u$
  \end{itemize}
\end{theorem}

Theorem \ref{thm:maps} gives us the technical tool that we need in
order to show that various parallel pairs of maps are not diamonds. We
begin our analysis of the collection of diamonds of $\Fr(\M_n)$
with diamonds whose source span rewrites disjoint subterms.

\begin{lemma}\label{lem:itdisjoint}
  Let $a \tensor_i b \in \Fr(\M_n)$ and suppose that there are maps
  $\varphi : a \to a'$ and $\psi:b \to b'$. Then, in the following
  diagram, the square labelled (d) is a commutative diamond and there
  is a map $a'\tensor_i b' \to c:$
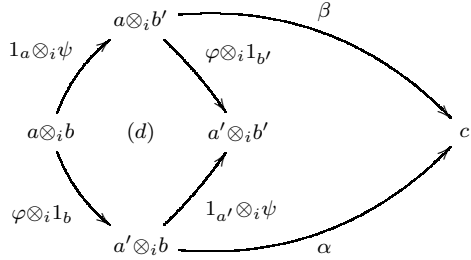
  \[
  \vcenter{
    \def\objectstyle{\scriptstyle}
    \def\labelstyle{\scriptstyle}
    \begin{xy}
      (0,0)*+{a\tensor_i b}="a";
      (12,15)*+{a \tensor_i b'}="b";
      (12,-15)*+{a' \tensor_i b}="c";
      (55,0)*+{c}="d";
      (25,0)*+{a'\tensor_i b'}="e";
      (12,0)*+{(d)};
      {\ar@/^0.2pc/^{\varphi \tensor_i 1_{b'}} "b";"e"};
      {\ar@/_0.2pc/_{1_{a'} \tensor_i \psi} "c";"e"};
      {\ar@{->}@/_1.5pc/_{\alpha} "c";"d"};
      {\ar@{->}@/^/^{1_a\tensor_i \psi} "a"; "b"};
      {\ar@{->}@/_/_{\varphi \tensor_i 1_b} "a"; "c"};
      {\ar@{->}@/^1.5pc/^{\beta} "b";"d"};
    \end{xy}
  }
  \]
\end{lemma}
\begin{proof}
  The square labelled (d) commutes by functoriality and it is easy to
  see that it does not admit a zig-zag subdivision, so it is a
  diamond. The tricky part is showing the existence of a map
  $a'\tensor_i b' \to c$. 

  Let $x,y \in \var(a'\tensor_i b')$ and suppose that $x \tensor_k y
  \in a'\tensor_i b'$. There are a few cases to 
  consider.
  \begin{itemize}
    \item If $x,y \in a'$, then $\alpha$ implies that there is some $m
      \ge k$ such that $x \tensor_m y \in c$ or there is some $m> k$ such
      that $y \tensor_m x \in c$.
    \item If $x,y \in b'$, then $\beta$ implies that there is some $m
      \ge k$ such that $x \tensor_m y\in c$ or there is some $m> k$ such
      that $y \tensor_m x \in c$. 
    \item If $x\in a'$ and $y \in b'$, then $x \tensor_i y \in a'
      \tensor b$. So, by $\alpha$, there is some $m \ge i$ such that $x
      \tensor_m y \in c$ or there is some $m> i$ such that $y \tensor_m x
      \in c$  
  \end{itemize}
Putting all of the above facts together, it follows from Theorem
\ref{thm:maps} that there is a map $a'\tensor_i b' \to c$.
\end{proof}

Next, we investigate diamonds whose initial span rewrites nested
subterms. For a term $a$ and a subterm $b \le a$, we write $a\{b\}$ to
represent this nested term. 

\begin{lemma}\label{itnested}
  Let $a\{b\} \in \Fr(\M_n)$ and suppose that there are maps
  $\varphi : a\{b\} \to a'\{b\}$ and $\psi:b \to b'$. Then, in the following
  diagram, the square labelled (d) is a commutative diamond and there
  is a map $a'\{b'\} \to c$: 
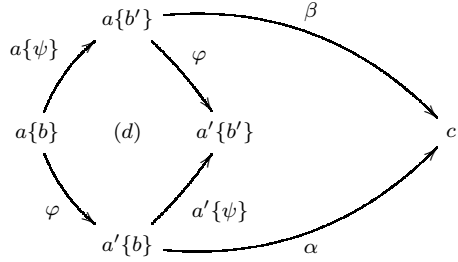
  \[
  \vcenter{
    \def\objectstyle{\scriptstyle}
    \def\labelstyle{\scriptstyle}
    \begin{xy}
      (0,0)*+{a\{b\}}="a";
      (12,15)*+{a\{b'\}}="b";
      (12,-15)*+{a'\{b\}}="c";
      (55,0)*+{c}="d";
      (25,0)*+{a'\{b'\}}="e";
      (12,0)*+{(d)};
      {\ar@/^0.2pc/^{\varphi} "b";"e"};
      {\ar@/_0.2pc/_{a'\{\psi\}} "c";"e"};
      {\ar@{->}@/_1.5pc/_{\alpha} "c";"d"};
      {\ar@{->}@/^/^{a\{\psi\}} "a"; "b"};
      {\ar@{->}@/_/_{\varphi} "a"; "c"};
      {\ar@{->}@/^1.5pc/^{\beta} "b";"d"};
    \end{xy}
  }
  \]
\end{lemma}
\begin{proof}
  The square labelled (d) commutes by naturality. The rest of the
  proof is similar to that of Lemma \ref{lem:itdisjoint}.
\end{proof}

We now know that source spans of the only remaining diamonds in
$\Fr(\M_n)$ rewrite overlapping terms.

\subsection{Interchange + associativity}

Let $j > i$. The first way in which interchange and associativity can
interact is in the term $X \tensor_i (c\tensor_j d) \tensor_i
(e\tensor_j f)$. Without loss of generality, we may assume that $X = a
\tensor_j b$, because we could always take $X = X \tensor_j I$. The
resulting span then gets completed into the internal associativity
axiom. One may then apply Theorem \ref{thm:maps} to show that there is
no other diamond with the same initial span.

The second way in which interchange can interact with associativity is
in the term $(a\tensor_j b)\tensor_i (c\tensor_j d\tensor_j e)$. In
this case, we get the following square, where the labels have the
evident components.

\[
\def\objectstyle{\scriptstyle}
\def\labelstyle{\scriptstyle}
\xymatrix{
  {(a\tensor_j b)\tensor_i (c\tensor_j d \tensor_j e)} \ar[rr]^{\eta}
  \ar[d]_{\eta} && {(a\tensor_i (c\tensor_j d))\tensor_j (b\tensor_i
    e)}\ar[d]^{\delta\tensor_j 1}\\
    {(a\tensor_i c)\tensor_j (b\tensor_i (d\tensor_j e))}
    \ar[rr]_{1\tensor_j \tilde{\delta}} &&
    {(a\tensor_i c)\tensor_j d \tensor_j (b\tensor_i e)}
}
\]

The above square commutes by substituting $(a\tensor_j I \tensor_j
b)\tensor_i (c\tensor_j d\tensor_j e)$ for the source and using
the external associativity axiom. Theorem \ref{thm:maps} easily yields
that there can be no other diamonds with the same initial span.

Similarly, a critical span arises at $(a\tensor_j b\tensor_j
c)\tensor_i(d\tensor_j e)$. The analysis is similar to the previous
case by inserting a unit to obtain $(a\tensor_j b \tensor_j
c)\tensor_i (d\tensor_j I\tensor_j e)$.

\subsection{Interchange + interchange}
Let $i < j < k$. An overlap between interchange rules occurs at $(a\tensor_j
b)\tensor_i ((c\tensor_k d)\tensor_j (e\tensor_j f))$. Since we have
strict units, we may assume that $a = a_1 \tensor_k a_t$ and $b = b_1
\tensor_k b_2$. We then obtain the initial span of the giant hexagon
axiom. The hexagon forms a diamond and it follows from Theorem
\ref{thm:maps} that there are no other diamonds with this initial
span.

\subsection{Interchange + units}

The critical spans arising from the interaction of interchange with
units yield the various Eckmann-Hilton maps. As we have seen, these
are not always joinable. When they are, they commute by the following
lemma. 

\begin{lemma}
  The following diagrams commute in  $\Fr(\M_n)$, where $1 \le i < j < k \le n$:
  \begin{center}
    \begin{tabular}{cc}
      $
      \def\objectstyle{\scriptstyle}
      \def\labelstyle{\scriptstyle}
      \xymatrix{
        &{a\tensor_j b} \ar[dr] & {} \\
        {a\tensor_i b} \ar[ur] \ar[rr] \ar@{}[urr]|{(1)} && {a\tensor_k b}
      }$

      &

      $
      \def\objectstyle{\scriptstyle}
      \def\labelstyle{\scriptstyle}
      \xymatrix{
        &{b\tensor_j a} \ar[dr] & {} \\
        {a\tensor_i b} \ar[ur] \ar[rr] \ar@{}[urr]|{(2)} && {a\tensor_k b}
      }$\\

      $
      \def\objectstyle{\scriptstyle}
      \def\labelstyle{\scriptstyle}
      \xymatrix{
        &{a\tensor_j b} \ar[dr] & {} \\
        {a\tensor_i b} \ar[ur] \ar[rr] \ar@{}[urr]|{(3)} && {b\tensor_k a}
      }$

      &

      $
      \def\objectstyle{\scriptstyle}
      \def\labelstyle{\scriptstyle}
      \xymatrix{
        &{b\tensor_j a} \ar[dr] & {} \\
        {a\tensor_i b} \ar[ur] \ar[rr] \ar@{}[urr]|{(4)} && {b\tensor_k a}
      }$
    \end{tabular}
  \end{center}
\end{lemma}
\begin{proof}
  This follows from \cite[Lemma 4.22]{iterated}. More explicitly it
  follows from the giant hexagon axiom by making the following
  substitutions: 

  \bigbreak
  
  \begin{enumerate}
    \item $a\tensor_ib = ((a\tensor_k I)\tensor_j(I\tensor_k I))
      \tensor_i ((I\tensor_kI)\tensor_j(I\tensor_k b))$
      \medbreak
    \item $a\tensor_ib = ((I\tensor_k I)\tensor_j(a\tensor_k I))
      \tensor_i ((I\tensor_k b)\tensor_j(I\tensor_k I))$
      \medbreak
    \item $a\tensor_ib = ((I\tensor_k a)\tensor_j (I\tensor_k I))
      \tensor_i ((I\tensor_k I) \tensor_j (b\tensor_k I))$
      \medbreak
    \item $a\tensor_ib = ((I\tensor_k I)\tensor_j (I\tensor_k a))
      \tensor_i ((b\tensor_k I) \tensor_j (I\tensor_k I))$
  \end{enumerate}
\end{proof}

\subsection{Putting it all together}

We have seen that $\Fr(\M_n)$ is quasicycle-free and that every diamond
in $\Fr(\M_n)$ commutes. We can therefore apply Theorem
\ref{thm:maclane} to obtain the coherence theorem for iterated
monoidal categories.

\begin{theorem}
 If $a$ and $b$ are terms of $\Fr(\M_n)$ having no repeated variables,
 then there is at most one reduction $a \to b$ in $\Fr(\M_n)$. \qed
\end{theorem}

Theorem \ref{thm:maclane} provided a valuable strategy for proving the
above coherence theorem for iterated monoidal categories. Although some
careful combinatorial investigations were still required, the overall
proof is conceptually straightforward. This demonstrates that even
reasonably complicated coherence problems for quasicycle-free
rewriting $2$-theories may be comparatively easily attacked with the
tools from Chapter \ref{ch:nunf}.\blanknonumber
\chapter{Conclusion}\label{ch:conclusions}

We have developed rewriting $2$-theories as an abstract framework for
studying coherence problems for covariant categorical
structures. While general coherence theorems have been developed
previously for
certain classes of covariant structures
\cite{power:coherence,Lack:descent}, the work has typically
been at an abstract categorical level and so does not yield any
techniques for constructing specific coherence diagrams. More recent
work on this problem has yielded an approach to obtaining coherence
axioms for invertible theories \cite{Fiore:laplaza}. However, the
coherence axioms chosen in \cite{Fiore:laplaza} are \emph{all} of the
diagrams in general position. Certainly, this vastly over-axiomatises 
most theories and the authors in \cite{Fiore:laplaza} note:

\begin{quote}
  ``It is not clear what general scheme would select coherence
  diagrams ``correctly'' in accordance with what one expects for
  specific examples of algebraic structures known.''
\end{quote}

The work in Chapter \ref{ch:unf} on complete rewriting $2$-theories
and in Chapter \ref{ch:nunf} on quasicycle-free rewriting $2$-theories
does, however, provide a general scheme for selecting coherence
diagrams in accordance with what one would expect for particular
algebraic structures. For complete rewriting $2$-theories, one only
needs to select a joining of each critical span in order to obtain a
complete set of coherence axioms. For quasicycle-free rewriting
$2$-theories, a complete set of coherence axioms is provided by the
basic diamonds. It is, however, generally more difficult to construct
coherence axioms in the quasicycle-free case. This is demonstrated by
the intricate investigation required in Chapter \ref{ch:iterated} for
iterated monoidal categories as opposed to the relatively
straightforward investigation of Catalan categories in Chapter
\ref{ch:catalan}.  

Our combinatorial approach to coherence has the additional benefit of
retaining a close link to classical one-dimensional universal
algebra. This has allowed us, in Chapter \ref{ch:structure}, to
use coherent categorifications of balanced equational theories to
build presentations of the associated structure monoids and
groups. Combined with our general techniques in chapters \ref{ch:unf}
and \ref{ch:nunf} for constructing coherence axioms, this provides a
powerful toolkit for developing presentations of groups and
monoids. This combination came to the fore in Chapter
\ref{ch:catalan}, where we constructed new presentations for the
higher Thompson groups $F_{n,1}$ and the Higman-Thompson groups $G_{n,1}$.

In the following section, we outline some additional questions raised
by the thesis. 

\section*{Further questions}

Both of our general coherence theorems for incomplete theories,
Theorem \ref{thm:lambek} and Theorem \ref{thm:maclane}, rely on 
the underlying structure being quasicycle-free. One may well call this
condition into question and wonder whether we can get away with a
weaker condition. For Lambek coherence, quasicycle-freeness does not
capture all covariant structures known to be Lambek coherent. For
example, braided monoidal categories are certainly not quasicycle free
and yet their Lambek coherence problem is solvable via the
Reidemeister moves \cite{JoyalStreet:Braided}. However, employing the
Reidemeister moves adds an additional rewrite system to the
reductions, thus expanding the amount of information available. 

\begin{q}
Are there general properties that ensure Lambek coherence for non
quasicycle-free rewriting $2$-theories? 
\end{q}

The reliance on quasicycle-freeness for Mac Lane coherence seems more
fundamental. However, two crucial ingredients of our theory rely
predominantly on acyclicity: Theorem \ref{thm:power} establishes that
the faces of a subdivision are themselves $st$-graphs, while
Proposition \ref{prop:diamond} shows that the faces of a maximal
subdivision are diamonds.  

\begin{q} What conditions on an acyclic rewriting $2$-theory
  ensure Mac Lane coherence?
\end{q}

It seems likely that acyclic rewriting $2$-theories in which every
diamond commutes are Mac Lane coherent. The major obstruction to
showing this is that maximal subdivisions are no longer guaranteed to
exist.  

As noted in Chapter \ref{ch:twostructures}, rewriting $2$-theories
with an empty set of coherence axioms correspond to the unconditional
fragment of rewriting logic. Meseguer has shown a strong connection
between rewriting logic and models of concurrency
\cite{Meseguer_rl}: the congruence classes of terms correspond to
states of the system, while reductions correspond to
processes. A parallel pair of reductions that are equal correspond in
this framework to a truly concurrent pair of processes. That is, they
correspond to a pair of processes that may be safely run in
parallel. In this way, one may view coherence axioms as specifications
that certain parallel pairs of processes that seemingly interact with
one another may in fact be safely run in parallel. In this way, the
Lambek coherence problem asks about the existence of a decision 
procedure for determining which processes may be safely run in
parallel. This computational interpretation of the Lambek coherence
problem motivates a more refined investigation of decision procedures
for the commutativity of diagrams arising from rewriting $2$-theories.  

\begin{q}\label{prob:comlexity}
What is the computational complexity of deciding whether a diagram
commutes in the structure generated by a Lambek coherent rewriting
$2$-theory?  
\end{q}

As we saw in Chapter \ref{ch:twostructures}, determining whether a
finitely presented rewriting $2$-theory is Mac Lane coherent is in
general undecidable. However, this does not rule out the possibility
of developing algorithms for tackling Mac Lane coherence. Indeed,
there exist many successful algorithms for determining whether a term
rewriting theory is terminating, even though this problem is also
undecidable in general --- a powerful such algorithm is provided by
the dependency pairs method \cite{DependencyPairs1,DependencyPairs2}. 

\begin{project}
  Develop algorithms for constructing coherent categorifications of
  labelled rewriting theories and for determining whether a given
  rewriting $2$-theory is Mac Lane coherent. 
\end{project}

A finite presentation of a coherent categorification of an equational
theory leads to an infinite set of singular morphisms. Thus, the
presentations constructed in Chapter \ref{ch:structure} yield an
infinite  presentation of the associated structure monoid or
group. This presentation has the nice property of 
imbueing the orbit graph of the resulting monoid or group with the
geometry of the categorical structure. However, many structure groups
are in fact finitely presentable. In particular, this is the case for
the groups $F_{n,1}$ and $G_{n,1}$. 

\begin{q}
Are there properties of a coherent categorification of an equational
theory that imply the finite presentability of the associated
structure monoid or group?
\end{q}

Our investigations have been wide-ranging, touching on topics from
category theory, computer science, universal algebra and group
theory. This has allowed us to to use computational insights to prove
theorems about mathematical objects, to construct general coherence
theorems that yield information about the actual coherence diagrams
and to build presentations of groups using very general
techniques. Hopefully future work will continue to exploit techniques
across traditional subject boundaries, so as to illustrate connections
and to foster dialogue.  \blanknonumber

\backmatter

\bibliographystyle{alpha}
\bibliography{papers_cnf,papers,papers_new,papers_structure}

\begin{thebibliography}{BFSV03}

\bibitem[Art00]{DependencyPairs1}
Thomas Arts.
\newblock System description: The dependency pair method.
\newblock In {\em RTA '00: Proceedings of the 11th International Conference on
  Rewriting Techniques and Applications}, pages 261--264, London, UK, 2000.
  Springer-Verlag.

\bibitem[BFSV03]{iterated}
C.~Balteanu, Z.~Fiedorowicz, R.~Schw{\"a}nzl, and R.~Vogt.
\newblock Iterated monoidal categories.
\newblock {\em Adv. Math.}, 176(2):277--349, 2003.

\bibitem[Bro87]{Brown:finiteness}
Kenneth~S. Brown.
\newblock Finiteness properties of groups.
\newblock In {\em Proceedings of the Northwestern conference on cohomology of
  groups (Evanston, Ill., 1985)}, volume~44, pages 45--75, 1987.

\bibitem[BS81]{BurrisSankappanavar}
Stanley Burris and H.~P. Sankappanavar.
\newblock {\em A course in universal algebra}, volume~78 of {\em Graduate Texts
  in Mathematics}.
\newblock Springer-Verlag, New York, 1981.

\bibitem[BS94]{Baader:unification}
Franz Baader and J{\"o}rg~H. Siekmann.
\newblock Unification theory.
\newblock In {\em Handbook of logic in artificial intelligence and logic
  programming, Vol.\ 2}, Oxford Sci. Publ., pages 41--125. Oxford Univ. Press,
  New York, 1994.

\bibitem[CFP96]{Thompsonintro}
J.~W. Cannon, W.~J. Floyd, and W.~R. Parry.
\newblock Introductory notes on {R}ichard {T}hompson's groups.
\newblock {\em Enseign. Math. (2)}, 42(3-4):215--256, 1996.

\bibitem[CG90]{Curien:coherence}
Pierre-Louis Curien and Giorgio Ghelli.
\newblock Coherence of subsumption.
\newblock In {\em CAAP '90 (Copenhagen, 1990)}, volume 431 of {\em Lecture
  Notes in Comput. Sci.}, pages 132--146. Springer, Berlin, 1990.

\bibitem[CS97]{CockettSeely:wdc}
J.~R.~B. Cockett and R.~A.~G. Seely.
\newblock Weakly distributive categories.
\newblock {\em J. Pure Appl. Algebra}, 114(2):133--173, 1997.

\bibitem[Deh93]{Dehornoy:varieties}
Patrick Dehornoy.
\newblock Structural monoids associated to equational varieties.
\newblock {\em Proc. Amer. Math. Soc.}, 117(2):293--304, 1993.

\bibitem[Deh00]{Dehornoy:braids}
Patrick Dehornoy.
\newblock {\em Braids and self-distributivity}, volume 192 of {\em Progress in
  Mathematics}.
\newblock Birkh\"auser Verlag, Basel, 2000.

\bibitem[Deh05]{Dehornoy:thompson}
Patrick Dehornoy.
\newblock Geometric presentations for {T}hompson's groups.
\newblock {\em J. Pure Appl. Algebra}, 203(1-3):1--44, 2005.

\bibitem[Deh06]{Dehornoy:preprint}
Patrick Dehornoy.
\newblock Using groups for investigating rewrite systems.
\newblock http://arxiv.org/abs/cs/0609102, 2006.

\bibitem[DJ90]{Dershowitz:handbook}
Nachum Dershowitz and Jean-Pierre Jouannaud.
\newblock Rewrite systems.
\newblock In {\em Handbook of theoretical computer science, Vol.\ B}, pages
  243--320. Elsevier, Amsterdam, 1990.

\bibitem[Dyd77a]{Dydak:b}
Jerzy Dydak.
\newblock 1-movable continua need not be pointed 1-movable.
\newblock {\em Bull. Acad. Polon. Sci. S\'er. Sci. Math. Astronom. Phys.},
  25(6):559--562, 1977.

\bibitem[Dyd77b]{Dydak:a}
Jerzy Dydak.
\newblock A simple proof that pointed {FANR}-spaces are regular fundamental
  retracts of {ANR}'s.
\newblock {\em Bull. Acad. Polon. Sci. S\'er. Sci. Math. Astronom. Phys.},
  25(1):55--62, 1977.

\bibitem[FH93]{FreydHeller}
Peter Freyd and Alex Heller.
\newblock Splitting homotopy idempotents. {II}.
\newblock {\em J. Pure Appl. Algebra}, 89(1-2):93--106, 1993.

\bibitem[FHK]{Fiore:laplaza}
Thomas~M. Fiore, Po~Huc, and Igor Kriz.
\newblock Laplaza sets, or how to select coherence diagrams for pseudo
  algebras.
\newblock \emph{{A}dvances in {M}athematics}, to appear.

\bibitem[GS01]{GugStrass:mell}
Alessio Guglielmi and Lutz Stra{\ss}burger.
\newblock Non-commutativity and {MELL} in the calculus of structures.
\newblock In {\em Computer science logic (Paris, 2001)}, volume 2142 of {\em
  Lecture Notes in Comput. Sci.}, pages 54--68. Springer, Berlin, 2001.

\bibitem[Gug07]{Gugl:si}
Alessio Guglielmi.
\newblock A system of interaction and structure.
\newblock {\em ACM Transactions on Computational Logic}, 8(1):1--64, 2007.

\bibitem[Hig74]{Higman:thompson}
Graham Higman.
\newblock {\em Finitely presented infinite simple groups}.
\newblock Department of Pure Mathematics, Department of Mathematics, I.A.S.
  Australian National University, Canberra, 1974.
\newblock Notes on Pure Mathematics, No. 8 (1974).

\bibitem[HM05]{DependencyPairs2}
Nao Hirokawa and Aart Middeldorp.
\newblock Automating the dependency pair method.
\newblock {\em Inf. Comput.}, 199(1-2):172--199, 2005.

\bibitem[Jan96]{janssen:compositionality}
Theo M.~V. Janssen.
\newblock Compositionality.
\newblock In Johan {van Benthem} and Alice {ter Meulen}, editors, {\em Handbook
  of Logic and Language}, pages 417--473. Elsevier, Amsterdam, 1996.

\bibitem[Joh87]{Johnson:thesis}
Michael Johnson.
\newblock {\em Pasting diagrams in $n$-categories with applications to
  coherence theorems and categories of paths.}
\newblock PhD thesis, The University of Sydney, 1987.

\bibitem[JS93]{JoyalStreet:Braided}
Andr{\'e} Joyal and Ross Street.
\newblock Braided tensor categories.
\newblock {\em Adv. Math.}, 102(1):20--78, 1993.

\bibitem[KB70]{KnuthBendix}
Donald~E. Knuth and Peter~B. Bendix.
\newblock Simple word problems in universal algebras.
\newblock In {\em Computational {P}roblems in {A}bstract {A}lgebra ({P}roc.
  {C}onf., {O}xford, 1967)}, pages 263--297. Pergamon, Oxford, 1970.

\bibitem[KdV03]{KlopdeV:FO}
Jan~Willem Klop and Roel de~Vrijer.
\newblock First-order term rewriting systems.
\newblock In {\em Term rewriting systems}, volume~55 of {\em Cambridge Tracts
  Theoret. Comput. Sci.}, pages 24--59. Cambridge Univ. Press, Cambridge, 2003.

\bibitem[Kel64]{Kelly:monoidal}
G.~M. Kelly.
\newblock On {M}ac{L}ane's conditions for coherence of natural associativities,
  commutativities, etc.
\newblock {\em J. Algebra}, 1:397--402, 1964.

\bibitem[Kel72]{Kelly_coherence}
G.~M. Kelly.
\newblock An abstract approach to coherence.
\newblock In {\em Coherence in categories}, pages 106--147. Lecture Notes in
  Math., Vol. 281. Springer, Berlin, 1972.

\bibitem[KL93]{KellyLack:2monads}
G.~M. Kelly and Stephen Lack.
\newblock Finite-product-preserving functors, {K}an extensions and
  strongly-finitary {$2$}-monads.
\newblock {\em Appl. Categ. Structures}, 1(1):85--94, 1993.

\bibitem[KS74]{KellyStreet:review}
G.~M. Kelly and Ross Street.
\newblock Review of the elements of {$2$}-categories.
\newblock In {\em Category Seminar (Proc. Sem., Sydney, 1972/1973)}, pages
  75--103. Lecture Notes in Math., Vol. 420. Springer, Berlin, 1974.

\bibitem[Lac02]{Lack:descent}
Stephen Lack.
\newblock Codescent objects and coherence.
\newblock {\em J. Pure Appl. Algebra}, 175(1-3):223--241, 2002.
\newblock Special volume celebrating the 70th birthday of Professor Max Kelly.

\bibitem[Lam68]{Lambek:coh1}
Joachim Lambek.
\newblock Deductive systems and categories. {I}. {S}yntactic calculus and
  residuated categories.
\newblock {\em Math. Systems Theory}, 2:287--318, 1968.

\bibitem[Lap72a]{Laplaza:associative}
Miguel~L. Laplaza.
\newblock Coherence for associativity not an isomorphism.
\newblock {\em J. Pure Appl. Algebra}, 2(2):107--120, 1972.

\bibitem[Lap72b]{Laplaza:distributive}
Miguel~L. Laplaza.
\newblock Coherence for distributivity.
\newblock In {\em Coherence in categories}, pages 29--65. Lecture Notes in
  Math., Vol. 281. Springer, Berlin, 1972.

\bibitem[Law04]{Lawvere:functorial}
F.~William Lawvere.
\newblock Functorial semantics of algebraic theories and some algebraic
  problems in the context of functorial semantics of algebraic theories.
\newblock {\em Repr. Theory Appl. Categ.}, (5):1--121 (electronic), 2004.
\newblock Reprinted from Proc. Nat. Acad. Sci. U.S.A. {\bf 50} (1963), 869--872
  and {\it Reports of the Midwest Category Seminar. II}, 41--61, Springer,
  Berlin, 1968.

\bibitem[Mar51]{Markov:undecidable}
A.~Markov.
\newblock The impossibility of certain algorithms in the theory of associative
  systems.
\newblock {\em Doklady Akad. Nauk SSSR (N.S.)}, 77:19--20, 1951.

\bibitem[Mel02]{Mel:residual}
Paul-Andr{\'e} Melli{\`e}s.
\newblock Axiomatic rewriting theory. {VI}. {R}esidual theory revisited.
\newblock In {\em Rewriting techniques and applications}, volume 2378 of {\em
  Lecture Notes in Comput. Sci.}, pages 24--50. Springer, Berlin, 2002.

\bibitem[Mes92]{Meseguer_rl}
Jos{\'e} Meseguer.
\newblock Conditional rewriting logic as a unified model of concurrency.
\newblock {\em Theoret. Comput. Sci.}, 96(1):73--155, 1992.
\newblock Second Workshop on Concurrency and Compositionality (San Miniato,
  1990).

\bibitem[ML63]{MacLane_natural}
Saunders Mac~Lane.
\newblock Natural associativity and commutativity.
\newblock {\em Rice Univ. Studies}, 49(4):28--46, 1963.

\bibitem[ML76]{MacLane:TopLogic}
Saunders Mac~Lane.
\newblock Topology and logic as a source of algebra.
\newblock {\em Bull. Amer. Math. Soc.}, 82(1):1--40, 1976.

\bibitem[Moo96]{Moore:presentation}
Eliakim~Hastings Moore.
\newblock Concerning the abstract groups of order $k!$ and $\frac{1}{2}k!$
  holohedrically isomorphic with the symmetric and the alternating
  substitution-groups on k letters.
\newblock {\em Proc. London Math. Soc.}, 28:357--367, 1896.

\bibitem[MT73]{McKenzieThompson:V}
Ralph McKenzie and Richard~J. Thompson.
\newblock An elementary construction of unsolvable word problems in group
  theory.
\newblock In {\em Word problems: decision problems and the Burnside problem in
  group theory (Conf., Univ. California, Irvine, Calif. 1969; dedicated to
  Hanna Neumann)}, volume~71 of {\em Studies in Logic and the Foundations of
  Math.}, pages 457--478. North-Holland, Amsterdam, 1973.

\bibitem[New42]{Newman:lemma}
M.H.A. Newman.
\newblock On theories witha combinatorial definition of `equivalence'.
\newblock {\em Annals of Mathematics}, 43:223--243, 1942.

\bibitem[Pow89]{power:coherence}
A.~J. Power.
\newblock A general coherence result.
\newblock {\em J. Pure Appl. Algebra}, 57(2):165--173, 1989.

\bibitem[Pow90]{Power:pasting2}
A.~J. Power.
\newblock A {$2$}-categorical pasting theorem.
\newblock {\em J. Algebra}, 129(2):439--445, 1990.

\bibitem[Pow99]{Power:enriched}
John Power.
\newblock Enriched {L}awvere theories.
\newblock {\em Theory Appl. Categ.}, 6:83--93 (electronic), 1999.
\newblock The Lambek Festschrift.

\bibitem[Pow05]{Power:discrete}
John Power.
\newblock Discrete {L}awvere theories.
\newblock In {\em Algebra and coalgebra in computer science}, volume 3629 of
  {\em Lecture Notes in Comput. Sci.}, pages 348--363. Springer, Berlin, 2005.

\bibitem[Rey91]{Reynolds:coherence}
John~C. Reynolds.
\newblock The coherence of languages with intersection types.
\newblock In {\em Theoretical aspects of computer software (Sendai, 1991)},
  volume 526 of {\em Lecture Notes in Comput. Sci.}, pages 675--700. Springer,
  Berlin, 1991.

\bibitem[Ros07]{Rosicky:homotopy}
J.~Rosick{\'y}.
\newblock On homotopy varieties.
\newblock {\em Adv. Math.}, 214(2):525--550, 2007.

\bibitem[Sco92]{Scott:tour}
E.~A. Scott.
\newblock A tour around finitely presented infinite simple groups.
\newblock In {\em Algorithms and classification in combinatorial group theory
  (Berkeley, CA, 1989)}, volume~23 of {\em Math. Sci. Res. Inst. Publ.}, pages
  83--119. Springer, New York, 1992.

\bibitem[Sta63]{Stasheff:associativity}
James~Dillon Stasheff.
\newblock Homotopy associativity of {$H$}-spaces. {I}, {II}.
\newblock {\em Trans. Amer. Math. Soc. 108 (1963), 275-292; ibid.},
  108:293--312, 1963.

\bibitem[Sta99]{Stanley:enumerative}
Richard~P. Stanley.
\newblock {\em Enumerative combinatorics. {V}ol. 2}, volume~62 of {\em
  Cambridge Studies in Advanced Mathematics}.
\newblock Cambridge University Press, Cambridge, 1999.

\bibitem[Ste94]{Stell:sesqui}
John~G. Stell.
\newblock Modelling term rewriting systems by sesqui-categories.
\newblock {\em Cat\'{e}gories, Alg\`{e}bres, Esquisses et n\'{e}o-esquisses},
  pages 121--126, 1994.

\bibitem[Tho80]{Thompson:F}
Richard~J. Thompson.
\newblock Embeddings into finitely generated simple groups which preserve the
  word problem.
\newblock In {\em Word problems, II (Conf. on Decision Problems in Algebra,
  Oxford, 1976)}, volume~95 of {\em Stud. Logic Foundations Math.}, pages
  401--441. North-Holland, Amsterdam, 1980.

\bibitem[vOdV03]{terese_equivalence}
Vincent van Oostrom and Roel de~Vrijer.
\newblock Equivalence of reductions.
\newblock In {\em Term rewriting systems}, volume~55 of {\em Cambridge Tracts
  Theoret. Comput. Sci.}, pages 301--474. Cambridge Univ. Press, Cambridge,
  2003.

\end{thebibliography}

\printindex
\end{document}